\documentclass[11pt,a4paper]{article}

\usepackage[english]{babel}
\usepackage[T1]{fontenc}
\usepackage[utf8]{inputenc}
\usepackage[
  margin=2.5cm,
  includefoot,
  footskip=30pt,
]{geometry}
\usepackage{lmodern}
\usepackage{amsmath}
\usepackage{amssymb}
\usepackage{amsthm}
\usepackage{booktabs}
\usepackage{array}
\usepackage{graphicx}
\usepackage{cite}
\usepackage[dvipsnames]{xcolor}
\usepackage{tikz}\usetikzlibrary{matrix, calc, arrows, shapes,
decorations.pathreplacing, decorations.markings, decorations.pathmorphing,
backgrounds,fit,positioning,shapes.symbols,chains,shadings,fadings}
\usepackage{tikz-cd}
\usetikzlibrary{decorations.markings,angles}
\usepackage{stmaryrd} \SetSymbolFont{stmry}{bold}{U}{stmry}{m}{n}
\usepackage{mathrsfs}
\usepackage{extarrows}
\usepackage{mathtools}
\usepackage{tensor}
\usepackage{enumerate}
\usepackage{adjustbox}
\usepackage[mathcal]{eucal}
\usepackage{hyperref}
\hypersetup{
	colorlinks,
	urlcolor={red!55!black},
	citecolor={green!60!black},
	linkcolor={red!55!black}
}
\usepackage{bm}
\usepackage{tabularx}
\usepackage{cleveref}
\usepackage[strict]{changepage}
\usepackage{tensor}

\mathchardef\mhyphen="2D
\def\on{\operatorname}

\makeatletter
\providecommand{\leftsquigarrow}{%
  \mathrel{\mathpalette\reflect@squig\relax}%
}
\newcommand{\reflect@squig}[2]{%
  \reflectbox{$\m@th#1\rightsquigarrow$}%
}
\makeatother

\def\ww{node[white]{$\bullet$}node[black]{$\circ$}}

\definecolor{ao}{rgb}{0.0, 0.5, 0.0}

\newtheorem{theorem}{Theorem}[section]
\newtheorem{lemma}[theorem]{Lemma}

\newtheorem{proposition}[theorem]{Proposition}
\newtheorem{corollary}[theorem]{Corollary}
\newtheorem{introthm}{Theorem}

\newtheorem{introprop}{Proposition}

\theoremstyle{definition}
\newtheorem{construction}[theorem]{Construction}
\newtheorem{definition}[theorem]{Definition}
\newtheorem{notation}[theorem]{Notation}
\newtheorem{remark}[theorem]{Remark}
\newtheorem{example}[theorem]{Example}

\newcommand\noloc{%
  \nobreak
  \mspace{6mu plus 1mu}
  {:}
  \nonscript\mkern-\thinmuskip
  \mathpunct{}
  \mspace{2mu}
}

\newcommand\cocolon{%
  \nobreak
  \mspace{6mu plus 1mu}
  {:}
  \nonscript\mkern-\thinmuskip
  \mathpunct{}
  \mspace{2mu}
}

\newcommand{\rgraph}{{\bf G}}

\newcommand{\A}{\mathcal{A}}
\newcommand{\B}{\mathcal{B}}
\newcommand{\C}{\mathcal{C}}
\newcommand{\D}{\mathcal{D}}
\newcommand{\E}{\mathcal{E}}
\newcommand{\V}{\mathcal{V}}
\newcommand{\N}{\mathcal{N}}

\newcommand{\cof}{\on{cof}}
\newcommand{\fib}{\on{fib}}
\newcommand{\Tr}{\on{Tr}}

\newcommand{\HH}{\on{HH}}

\newcommand{\id}{\on{id}}

\newcommand{\unit}{\on{u}}
\newcommand{\counit}{\on{cu}}
\newcommand{\glsec}{\Gamma}
\newcommand{\cptglsec}{\Gamma^{\on{dual}}}

\title{Relative Calabi--Yau structures and perverse schobers on surfaces}
\author{Merlin Christ}
\date{\today}

\begin{document}
\maketitle

\abstract{We give a treatment of relative Calabi--Yau structures on functors between $R$-linear stable $\infty$-categories, with $R$ any $\mathbb{E}_\infty$-ring spectrum, generalizing previous treatments in the setting of dg categories. Using their gluing properties, we further construct relative Calabi--Yau structures on the global sections of perverse schobers, i.e.~categorified perverse sheaves, on surfaces with boundary. We treat examples coming from Fukaya categories and representation theory. In a related direction, we define the monodromy of a perverse schober parametrized by a ribbon graph on a framed surface and show that it forms a local system of stable $\infty$-categories.} 
\tableofcontents

\section{Introduction}

Let $k$ be a field. A $k$-linear triangulated category $C$ with finite dimensional Homs is called $n$-Calabi--Yau if there exists an isomorphism of vector spaces
\[\on{Ext}_C^i(X,Y)\simeq \on{Ext}^{n-i}_C(Y,X)^*\,,\]
bifunctorial in $X,Y\in C$. 
To obtain a well-behaved notion of $n$-Calabi--Yau structure on a proper $k$-linear stable $\infty$-category $\C$, one can ask for a trivialization $S\simeq [n]$ of the Serre functor $S$, i.e.~the functor satisfying the following duality for $k$-linear derived Homs in $\C$:
\[
\on{Mor}_{\C}(X,Y)\simeq \on{Mor}_{\C}(Y,S(X))^*\in \mathcal{D}(k)\,.
\]
The trivialization $S\simeq [n]$ is called a weak right $n$-Calabi--Yau structure on $\C$. Note that the natural transformations $\on{Mor}(\on{id}_{\C},S)$ describe the dual Hochschild homology $\on{HH}(\C)^*$ of $\C$. The identification $S\simeq [n]$ may thus additionally be required to be $S^1$-invariant, i.e.~to lie in the image of a dual cyclic homology class under the morphism $\on{HH}_{S^1}(\C)^*\to \on{HH}(\C)^*$. This leads to the notion of a right $n$-Calabi--Yau structure on $\C$. The importance of this $S^1$-invariance comes from the relation with topological field theories \cite{Lur09,Cos07}.

There is a similar notion of left $n$-Calabi--Yau structure on a smooth $k$-linear stable $\infty$-category $\C$, where instead of the Serre functor, one asks for a trivialization of the bimodule left dual $\on{id}_\C^!$ of the evaluation bimodule. The endofunctor $\on{id}_\C^!$ is sometimes called the inverse dualizing bimodule, as it is inverse to the Serre functor if $\C$ is smooth and proper.\\

In this paper, we will be concerned with generalizations of left and right Calabi--Yau structures to $R$-linear functors between $R$-linear stable $\infty$-categories, where $R$ is a base $\mathbb{E}_\infty$-ring spectrum. Calabi--Yau structures on functors are also referred to as relative Calabi--Yau structures. This notion was suggested by To\"{e}n \cite{Toe14} and fully worked out in the setting of dg categories by Brav--Dyckerhoff \cite{BD19}. The first half of this paper concerns a careful lift of this theory to the setting of $R$-linear stable $\infty$-categories.

There are many natural examples of relative Calabi--Yau categories, the known ones usually coming from Fukaya categories, representation theory, topology and algebraic geometry. Applications of relative Calabi--Yau structures include the constructions of shifted symplectic structures and Lagrangian structures on moduli spaces of objects \cite{BD21}, the construction of additive categorifications of cluster algebras with coefficients \cite{Wu21,Chr22b,KW23}, of $2$-Calabi--Yau exact $\infty$-categories/extriangulated categories from right $2$-Calabi--Yau functors \cite{Chr22b}, and of framed $E_2$-algebra structures on Hochschild cohomology \cite{BR23}, to name a few.

Relative Calabi--Yau structures posses the remarkable feature that they can be glued together along suitable pushouts or pullbacks of categories to produce new relative or absolute Calabi--Yau structures. As emphasized by Brav--Dyckerhoff \cite{BD19}, Calabi--Yau structures should be considered as noncommutative orientations and their gluing property as a noncommutative version of the gluing property of oriented manifolds with boundary along boundary components. 

The gluing properties of relative Calabi--Yau structures can be used to construct relative Calabi--Yau structures on functors with target the topological Fukaya categories associated with framed marked surfaces \cite{BD19}. These topological Fukaya categories can be seen as the global sections of perverse schobers on surfaces, i.e.~categorified perverse sheaves in the sense of \cite{KS14}. More generally, perverse schobers allow to define Fukaya categories of surfaces 'with coefficients'. Using the framework of \cite{Chr22}, we parametrize such perverse schobers by a ribbon graph homotopic to the surface. More specifically, such a parametrized perverse schober then amounts by definition to a certain constructible sheaf of stable $\infty$-categories on the ribbon graph, which is locally described by spherical adjunctions. 

In the second half of this paper, we discuss ways to construct relative Calabi--Yau structures on the global sections of more general perverse schobers on surfaces. The main results of this second half can be summarized as follows:

\begin{itemize}
\item We associate a local system of stable $\infty$-categories to a parametrized perverse schober on a framed marked surface encoding its monodromy on the surface away from its singularities. We also establish the independence of this local system on the chosen parametrizing ribbon graph. Further, we prove that parametrized perverse schobers without singularities are determined, up to equivalence, by their monodromy local systems.
\item In the special case of perverse schobers without singularities, we prove the existence of a relative Calabi--Yau structure on its $\infty$-category of global sections given the invariance of a local (negative) cyclic homology class under the monodromy action. This generalizes Brav--Dyckerhoff's result \cite{BD19} on relative Calabi--Yau structures on topological Fukaya categories of framed surfaces. 
\item We describe conditions which guarantee the existence of relative Calabi--Yau structures on the local and on the global sections of singular parametrized perverse schobers. 
\item We apply our results to construct relative Calabi--Yau structures on classes of examples, including Fukaya--Seidel categories, periodic topological Fukaya categories of marked surfaces, the derived categories of relative Ginzburg algebras associated with $n$-angulated surfaces, as well as variants of the latter which are linear over an arbitrary $\mathbb{E}_\infty$-ring spectrum. 
\end{itemize}

The remainder of the introduction is structured as follows. We begin in \Cref{introsec:CY} by reviewing the notion of a relative Calabi--Yau structure over a base $\mathbb{E}_\infty$-ring spectrum. We proceed in \Cref{introsec:schober} with describing our results on relative Calabi--Yau structures on perverse schobers. In \Cref{introsec:examples}, we describe the implications of our results for Fukaya--Seidel categories and other Fukaya-type categories.

\subsection{Relative Calabi--Yau structures}\label{introsec:CY}

Let $R$ be the base $\mathbb{E}_\infty$-ring spectrum. Our setting for the definition of relative Calabi--Yau structures is that of stable, presentable, dualizable, $R$-linear $\infty$-categories and dualizable (in particular colimit preserving), $R$-linear functors between them. In the following, we sketch the definition of relative Calabi--Yau structures and describe the gluing properties. The definition makes use of the functoriality of $R$-linear Hochschild homology, as well as its $S^1$-action, which we obtain from the formalism of traces \cite{HSS17,TV15}. 

Consider a dualizable $R$-linear functor $F\colon \D\to \C$ between dualizable $R$-linear $\infty$-categories. To define the notion of a right Calabi--Yau structure on $F$ (also sometimes called a relative right Calabi--Yau structure on $F$), we assume that $\C,\D$ are proper as $R$-linear $\infty$-categories. The $R$-linear $\infty$-category $\C$ being proper means that the evaluation functor $\on{ev}_\C\colon \C^\vee \otimes \C\to \on{RMod}_R$ admits an $R$-linear right adjoint, which can be identified with an endofunctor $\on{id}_\C^*$ of $\C$. If $\C$ is compactly generated, the functor $\on{id}_\C^*$ is a Serre functor on $\C$. The natural transformations between $\on{id}_\C^*$ and the identity functor are described by the dual Hochschild homology $\on{HH}(\C)^*$. In a similar way, a class $\sigma\colon R[n]\to \HH(\D,\C)^*\coloneqq \on{cof}(\HH(\C)^*\to\HH(\D)^*)$ in the dual relative Hochschild homology of $F$ defines a map $\alpha\colon \on{id}_\C\to \on{id}_\C^*[1-n]$ together with a null-homotopy of the composite map $\on{id}_\D\to \on{id}_\D^*[1-n]$ contained in the following diagram
\[
\begin{tikzcd}
{\on{id}_{\mathcal{D}}} \arrow[r] \arrow[d, dashed] & {F^*(\on{id}_\C)} \arrow[d, "F^*(\alpha)"] \arrow[r] & \on{cof} \arrow[d, dashed]      \\
\on{fib} \arrow[r]                                               & F^*(\on{id}_\C^*)[1-n] \arrow[r]     & \on{id}_{\mathcal{D}}^*[1-n]
\end{tikzcd}
\]
with horizontal fiber and cofiber sequences.  This null-homotopy allows us to fill in the dashed arrows. We call $\sigma$ a weak right $n$-Calabi--Yau structure on $F$ if the vertical maps in the above diagram are equivalences. A right $n$-Calabi--Yau structure on $F$ then further consists of a lift of $\sigma$ to a relative dual cyclic homology class. If $F$ admits a right $n$-Calabi--Yau structure, we also say that $\D$ is relative right $n$-Calabi--Yau. Non-relative right Calabi--Yau structures correspond to the case $\C=0$. 

The above definition is thus analogous to the definition of relative Calabi--Yau structures given in the setting of dg categories by Brav--Dyckerhoff \cite{BD19}. We will show in \Cref{lem:dg_CY_to_infty_CY} below that a dg functor $f\colon A\to B$ admits a weak Calabi--Yau structure if and only if the $k$-linear functor $\D(f)\colon \D(A)\to \D(B)$ between the compactly generated derived $\infty$-categories admits a weak Calabi--Yau structure. The possibility to consider relative Calabi--Yau structures on dualizable $k$-linear $\infty$-categories thus makes the $\infty$-categorical setting slightly more general than the dg categorical setting. The definition of weak relative left and right Calabi--Yau structures for compactly generated categories linear over a commutative ring spectrum was also previously described in a model categorical framework in the unpublished Master's thesis \cite{Lei17}. 

We again summarize the above definition as follows: If $\D$ is right $n$-Calabi--Yau, we have a trivialization $\on{id}_\D^*[-n]\simeq \on{id}_{\D}$ of the shifted Serre functor. If instead $\D$ is relative right Calabi--Yau, we have some natural transformation $\on{id}_\D^*[-n]\to \on{id}_{\D}$, together with an identification of its cofiber with $F^*(\on{id}_\C)=GF$, where $G$ is the right adjoint of $F$. To get a well behaved theory, it is however important that this is not just any identification of the cofiber of $\on{id}_\D^*[-n]\to \on{id}_{\D}$ with $GF$, but rather that there is a specific such cofiber sequence coming from a relative dual Hochschild class. 

Left Calabi--Yau structures for functors between smooth $R$-linear $\infty$-categories are defined similarly, by replacing the right adjoint of the evaluation functor by the left adjoint, corresponding to an endofunctor $\on{id}_\C^!$, and dual cyclic homology by negative cyclic homology.\\

\noindent {\bf Gluing Calabi--Yau structures}

Relative Calabi--Yau structure can be glued together along pushouts or pullbacks of $\infty$-categories. There are also variants of this for $(\infty,2)$-categorical lax pushouts and pullbacks, see \cite[Lem.~6.3.3]{CDW23}. 

For the gluing of left Calabi--Yau structures, consider a pushout diagram of smooth, dualizable $R$-linear $\infty$-categories and dualizable functors:
\[
\begin{tikzcd}
                        &                                                                    & \B_3 \arrow[d] \\
                        & \B_2 \arrow[rd, "\ulcorner", phantom] \arrow[r] \arrow[d] & \C_2 \arrow[d] \\
\B_1 \arrow[r] & \C_1 \arrow[r]                                            & \D            
\end{tikzcd}
\]

\begin{introthm}[\Cref{thm:leftCYglue},\cite{BD19} for $R=k$ a field]\label{introthm:leftCYglue}
If the functors $\B_1\times \B_2\to \C_1$ and $\B_2\times \B_3\to \C_2$ carry $R$-linear left $n$-Calabi--Yau structures, which are compatible at $\B_2$, then the functor $\B_1\times \B_3\to \D$ inherits an $R$-linear left $n$-Calabi--Yau structure.
\end{introthm}

For the gluing of right Calabi--Yau structures, we consider a pullback diagram of proper dualizable $R$-linear $\infty$-categories and dualizable functors as follows:
\[
\begin{tikzcd}
\D \arrow[r] \arrow[d] \arrow[rd, "\lrcorner", phantom] & \C_2 \arrow[d] \arrow[r] & \B_3 \\
\C_1 \arrow[r] \arrow[d]                                & \B_2                     &      \\
\B_1                                                    &                          &     
\end{tikzcd}
\]

\begin{introthm}[\Cref{thm:rightCYglue}]\label{introthm:rightCYglue}
If the functors $\C_1\to \B_1\times \B_2$ and $\C_2\to \B_2\times \B_3$ carry $R$-linear right $n$-Calabi--Yau structures, which are compatible at $\B_2$, then the functor $\D\to \B_1\times \B_3$ inherits an $R$-linear right $n$-Calabi--Yau structure.
\end{introthm}

We note that the analogue of \Cref{introthm:rightCYglue} was not previously known for right Calabi--Yau structures on dg categories \cite{BD19}.

\subsection{Perverse schobers and relative Calabi--Yau structures}\label{introsec:schober}

Perverse schobers are a, in general conjectural, categorification of perverse sheaves, proposed by Kapranov-Schechtman \cite{KS14}. In this paper, we use the framework of perverse schobers parametrized by ribbon graphs of \cite{Chr22}. This describes perverse schobers on marked surfaces with boundary in terms of constructible sheaves valued in stable $\infty$-categories defined on a spanning ribbon graph embedded in the surfaces. Concretely, a perverse schober $\mathcal{F}$ parametrized by a graph ${\rgraph}$ is encoded as a functor $\mathcal{F}\colon \on{Exit}({\rgraph})\to \on{St}$. Here $\on{St}$ denotes the $\infty$-category of stable $\infty$-categories and the domain denotes the exit path category of $\rgraph$, whose objects are the vertices and edges of $\rgraph$ and whose morphisms describe the incidence between vertices and edges. The limit of this functor, denoted by $\glsec(\rgraph,\mathcal{F})$, is called the $\infty$-category of global sections of $\mathcal{F}$.

We remark that the usage of enhanced triangulated categories (such as the stable $\infty$-categories we employ in this paper) in our treatment of perverse schobers is essential, since there is no sensible theory of homotopy (co)limits of non-enhanced triangulated categories, which is needed for such a sheaf theory.\\

\noindent {\bf Monodromy of perverse schobers} 

A perverse sheaf on a topological surface restricts to a cochain complex of locally constant sheaves on the top dimensional stratum, which is the complement of the discrete set of singularities. The cohomology sheaves of this cochain complex are trivial except in degree $-1$ (in the typical convention), thus defining a local system of vector spaces. We discuss in \Cref{subsec:monodromy}, and sketch in the following, how to associate a similar local system of stable $\infty$-categories to parametrized perverse schobers.

If we only consider connected surfaces, then all generic, i.e.~non-singular, stalks of a given perverse sheaf are equivalent. The same is true for a $\rgraph$-parametrized perverse schober $\mathcal{F}$: the value of $\mathcal{F}$ at any edge of $\rgraph$ is independent of the chosen edge, up to equivalence, and should be considered as the generic stalk. Denote the generic stalk by $\N$. Locally at each vertex $v$ of $\Gamma$, a perverse schober is described by a spherical adjunction $\mathcal{V}\leftrightarrow \mathcal{N}$, with $\mathcal{V}$ called the $\infty$-category of vanishing cycles at $v$. If $\mathcal{V}\neq 0$, we call the vertex $v$ a singularity of $\mathcal{F}$. 

Let $\rgraph_0$ be the set of vertices of $\rgraph$. Given a perverse schober $\mathcal{F}$ with set of singularities $P\subset \rgraph_0$, we wish to associate a local system valued in $\on{St}$ on ${\bf S}\backslash P$, which we describe as a group homomorphism 
\[  \pi_1({\bf S}\backslash P)\longrightarrow \pi_0\on{Aut}(\mathcal{N})\]
to the group of equivalence classes of autoequivalences of the generic stalk $\N$. 

There is in general no \textit{canonical} choice of such a local system. We can however canonically define a local system $\mathcal{L}\mathcal{F}$ on the total space of the frame bundle $\on{Fr}({\bf S}\backslash \rgraph_0)\to {\bf S}\backslash \rgraph_0$. The fiber of the frame bundle has the homotopy type of the circle $S^1$, the monodromy of the local system along the fiber is given $[2]$. Suppose now that we choose a framing $\xi$ of the surface ${\bf S}\backslash P$, meaning a section of its frame bundle. We can then pull back to a local system $\xi^*\mathcal{L}\mathcal{F}$ on ${\bf S}\backslash \rgraph_0$, and crucially, this local system extends to ${\bf S}\backslash P$. This defines the desired monodromy local system of $\mathcal{F}$. We stress that this local system depends on the choice of framing $\xi$.

Note that in the special case that $\N$ is $2$-periodic, i.e.~$[2]\simeq\on{id}_\N$, the local system on the frame bundle has trivial monodromy on the fiber. It thus already reduces to a local system on ${\bf S}\backslash P$ and no choice of framing is required as input.

Perverse schobers without singularities are fully determined by their monodromy:

\begin{introprop}[\Cref{prop:schobersfrommonodromy}]\label{introprop:monodromy}
Let $\xi$ be a framing of ${\bf S}$. Let $\mathcal{F}_1,\mathcal{F}_2$ be two $\rgraph$-parametrized perverse schobers without singularities with identical generic stalk $\N$. Then there exists an equivalence $\mathcal{F}_1\simeq \mathcal{F}_2$ if and only if the the corresponding local systems 
\[  \xi^*\mathcal{L}\mathcal{F}_1,\xi^*\mathcal{L}\mathcal{F}_2\colon \pi_1({\bf S}\backslash P)\longrightarrow \pi_0\on{Aut}(\mathcal{N})\]
are equivalent.
\end{introprop}

The notion of a non-singular parametrized perverse schober is thus non-canonically equivalent to the notion of a local system of stable $\infty$-categories on the surface. Note that what we refer to as the global sections of the non-singular perverse schober is however very different to the global sections of a local system. The former type of global sections describes a generalized topological Fukaya category and categorifies the first cohomology of the surface relative the complement in the boundary of the marked points.

Our results about the monodromy of perverse schobers relate with the problem of defining the topological Fukaya category of a marked surface over an arbitrary base $\infty$-category $\N$: Without further assumptions on $\N$, a choice of framing of the surface is required. Then there exists a perverse schober (unique up to equivalence) with generic stalk $\N$ and trivial monodromy relative to the chosen framing. Its $\infty$-category of global sections describes the desired $\N$-valued topological Fukaya category. In the case $\N=\D(k)$, this $\infty$-categorical topological Fukaya category recovers the derived $\infty$-category of the dg categorical topological Fukaya category, or equivalently of the $A_\infty$-categorical partially wrapped Fukaya-category. If $\N$ is $2$-periodic, then no choice of framing is required, there is already a perverse schober with a well-defined trivial monodromy, whose global sections give the $\N$-valued topological Fukaya category. In the setting of dg categories, this problem of constructing topological Fukaya categories (up to a contractible space of choices) was fully solved by Dyckerhoff--Kapranov \cite{DK18,DK15} using the formalism of $2$-Segal objects. Their construction in fact additionally supplies a choice of perverse schober with trivial monodromy for every choice spanning ribbon graph. Their construction was extended to the case of $\N$ the stable $\infty$-category of right modules over the $2$-periodic sphere spectrum by Lurie \cite{LurWaldhausen}.\\

\noindent {\bf Local Calabi--Yau structures}

Let $\mathcal{F}$ be an $R$-linear $\rgraph$-parametrized perverse schober. 
Locally at any vertex $v$ of the graph ${\rgraph}$, with incident edges $e_1,\dots,e_m$, the perverse schober $\mathcal{F}$ is given by a collection of functors $\mathcal{F}(v)\to \mathcal{F}(e_i)\simeq \N$, $1\leq i\leq m$. One can show that these functors arise, up to suitable equivalence, from a single spherical adjunction $F \colon \V\leftrightarrow \N\noloc G$ via an explicit construction based on the relative Waldhausen $S_\bullet$-construction, see \Cref{prop:localmodel}. We say that the adjunction $F\dashv G$ underlies $\mathcal{F}$ at $v$.

\begin{introprop}[Combine \Cref{prop:localmodel,prop:loccy}]\label{introprop:localCY}
Let $\mathcal{F}$ be a ${\rgraph}$-parametrized perverse schober, with underlying spherical adjunction $F\colon \V\leftrightarrow \N\noloc G$ near a vertex $v$ of ${\rgraph}$.
\begin{enumerate}[(1)]
\item If $F$ admits a right $n$-Calabi--Yau structure, which restricts to a right $(n-1)$-Calabi--Yau structure on $\N$, then the functor $\mathcal{F}(v)\to \prod_{i=1}^n\mathcal{F}(e_i)$ also admits a right $n$-Calabi--Yau structure.
\item If $G$ admits a left $n$-Calabi--Yau structure, which restricts to a left $(n-1)$-Calabi--Yau structure of $\N$, then the right adjoint $\prod_{i=1}^n\mathcal{F}(e_i)\to \mathcal{F}(v)$ of the above functor also admits a left $n$-Calabi--Yau structure.
\end{enumerate}
\end{introprop}

We further prove a novel criterion for a spherical functor $F\colon \V\to \N$ between compactly generated, proper $R$-linear $\infty$-categories, where $\N$ is weakly right $(n-1)$-Calabi--Yau, to admit a weak right $n$-Calabi--Yau structure: this is the case if and only if its twist functor $T_\V$ is equivalent to the shifted Serre functor $\on{id}_\V^*[1-n]$, see \Cref{prop:sphericalrightCY}.\\

\noindent {\bf Global Calabi--Yau structures}

Given a $\rgraph$-parametrized perverse schober $\mathcal{F}$, we can evaluate global sections at the external (i.e.~boundary) edges of $\rgraph$, whose set is denoted by $\rgraph_1^\partial$. This yields a functor 
\[ \prod_{e\in {\rgraph}_1^\partial}\on{ev}_{e}\colon \glsec({\rgraph},\mathcal{F})\longrightarrow \prod_{e\in {\rgraph}_1^\partial} \mathcal{F}(e)\,.\]
The right adjoint of this functor is denoted by $\partial \mathcal{F}$. 

Typically, a relative Calabi--Yau structure on the $\infty$-category of global sections $\glsec({\rgraph},\mathcal{F})$ arises in the smooth setting as a left Calabi--Yau structure on the functor $\partial \mathcal{F}$ and in the proper setting as right Calabi--Yau structure on the functor $\prod_{e\in {\rgraph}_1^\partial}\on{ev}_{e}$. 

Note that finite limits of proper dualizbale $R$-linear $\infty$-categories in $\on{St}$ (or equivalently in the $\infty$-category $\on{LinCat}_R$ of $R$-linear $\infty$-categories) are not necessarily again dualizable and proper. We can fix this issue by forming the limits in the $\infty$-category $\on{LinCat}_R^{\on{dual}}$ of dualizable $R$-linear $\infty$-categories, see \Cref{cor:limitisproper}. The arising notion of $\infty$-category of global sections is denoted by $\cptglsec({\bf G},\mathcal{F})$, we call these the dualizable global sections. In the proper setting, we should thus ask for the restriction of $\prod_{e\in {\rgraph}_1^\partial}\on{ev}_{e}$ to $\cptglsec({\bf G},\mathcal{F})$ to be relative right Calabi--Yau. 

When the global sections describe the partially wrapped Fukaya category of a surface, the difference between $\glsec({\bf G},\mathcal{F})$ and $\cptglsec({\bf G},\mathcal{F})$ can be explained as follows: the category $\glsec({\bf G},\mathcal{F})$ describes the usual smooth partially wrapped Fukaya category of the marked surface. The proper full subcategory $\cptglsec({\bf G},\mathcal{F})\subset \glsec({\bf G},\mathcal{F})$ consists of those Lagrangians which do not end at the boundary components containing no marked points. 

For perverse schobers without singularities, we prove the following.

\begin{introthm}[\Cref{thm:FukayaCY}]\label{introthm:FukayaCY} 
Let $\mathcal{F}$ be a $\rgraph$-parametrized perverse schober without singularities valued in dualizable $R$-linear $\infty$-categories. Suppose that the generic stalk $\mathcal{N}$ is smooth and admits a left $(n-1)$-Calabi--Yau structure 
\[\eta\colon R[n-1]\to \HH(\mathcal{N})^{S^1}\,.\]
Suppose that the monodromy local system $\HH(\mathcal{L}\mathcal{F})^{S^1}$ on ${\bf S}$ preserves $\eta$. Then the functor 
\[
\partial \mathcal{F}\colon \prod_{e\in {\rgraph}_1^\partial}\mathcal{F}(e) \longrightarrow \glsec({\rgraph},\mathcal{F})
\]
admits a left $n$-Calabi--Yau structure. 

A similar statement holds for relative right Calabi--Yau structures on the dualizable global sections $\cptglsec(\rgraph,\mathcal{F})$.
\end{introthm}

\Cref{introthm:FukayaCY} generalizes Brav--Dyckerhoff's result \cite{BD19} on relative Calabi--Yau structures on topological Fukaya categories of framed marked surfaces (corresponding to the case $\N=\D(k)$ and a perverse schober with trivial monodromy relative to the chosen framing). 

There is no direct analogue of \Cref{introthm:FukayaCY} for general perverse schobers with singularities. Essentially, this is because a perverse schober is not determined up to equivalence by the separate records of monodromy data and local singularity data, see \Cref{ex:monodromy}. There is however an almost immediate consequence of the gluing property of relative Calabi--Yau structures for global sections of perverse schobers, see \Cref{thm:schobercy}, which can be applied in practice by using the local Calabi--Yau structures from \Cref{introprop:localCY}. 

\subsection{Examples: Fukaya categories and Fukaya-type categories}\label{introsec:examples}

\noindent {\bf Fukaya--Seidel categories}

Let $X$ be an exact symplectic manifold of dimension $2n$ and $\pi\colon X\to \mathbb{D}$ a Lefschetz fibration with base the disc. Let $F$ be the regular fiber of $\pi$ and $\on{Fuk}(F)$ the proper Fukaya $A_\infty$-category of compact Lagrangians in $F$. The Fukaya--Seidel $A_\infty$-category $\on{FS}(\pi)$ is equivalent to the directed $A_\infty$-subcategory of $\on{Fuk}(F)$ on the vanishing cycles of the Lefschetz fibrations \cite{Sei08}. The corresponding derived Fukaya--Seidel category is a triangulated category, and admits a canonical enhancement to a $k$-linear stable $\infty$-category $\D(\on{FS}(\pi))$.

The formalism of parametrized perverse schobers on the disc $\mathbb{D}$, considered as a marked surface with a single marked point, realizes the derived Fukaya--Seidel category $\D(\on{FS}(\pi))$ as the global sections of a perverse schober. The generic stalk of the schober is the derived Fukaya category of the fiber $\D(\on{Fuk}(F))$. The singularities of the perverse schober lie at the singular values of the Lefschetz fibration, the corresponding spherical adjunctions arise from the spherical objects in $\on{Fuk}(F)$ given by the vanishing cycles. The ribbon graph parametrizing $\mathcal{F}$ is chosen so that the singular values all lie at $1$-valent vertices. There is a further non-singular $(m+1)$-valent vertex $v$, with $m$ the number of vanishing cycles. The value of $\mathcal{F}$ at $v$ is the directed $\infty$-category $\mathcal{F}(v)\simeq \on{Fun}(\Delta^{m-1},\D(\on{Fuk}(F)))$.

Any spherical object in a weak $(n-1)$-Calabi--Yau category gives rise to a weak $n$-Calabi--Yau spherical functor, see \Cref{lem:sphCY}. The gluing properties of right Calabi--Yau structures thus yield a weak relative right $n$-Calabi--Yau structure on $\D(\on{FS}(\pi))$. This induces the known natural transformation \cite{Sei06} from the Serre functor $\on{id}_{\D(\on{FS}(\pi))}^*\to \on{id}_{\D(\on{FS}(\pi))}[n+1]$.\\

In \Cref{subsec:FScats}, we will give a more detailed account of the above construction in the alternative framework of \cite{GPS24} for the definition of Fukaya--Seidel categories. In this framework, we furthermore prove that the smooth and proper Fukaya--Seidel category admits not only a weak relative right $n$-Calabi--Yau structure, but also a weak relative left $n$-Calabi--Yau structure:

\begin{introthm}[\Cref{thm:FSschober}]
Let $\pi\colon X\to \mathbb{C}_{\geq 0}$ be a Lefschetz fibration as in \Cref{subsec:FScats}.
\begin{enumerate}[(i)]
\item The derived $\infty$-category of the Fukaya--Seidel category $\D(\on{FS}(\pi))$ arises as the $\infty$-category of global sections of the perverse schober $\mathcal{F}$ on the disc from \Cref{constr:FSschober}. 
\item The smooth and proper derived Fukaya--Seidel category $\D(\on{FS}(\pi))$ admits both a weak relative left $n$-Calabi--Yau structure and a weak relative right $n$-Calabi--Yau structure. 
\end{enumerate} 
\end{introthm}

The Fukaya--Seidel category is defined in \cite{GPS24} as a partially wrapped Fukaya category with a stop in the fiber over $\infty$. Part (i) of \Cref{thm:FSschober} should readily generalize to the partially wrapped Fukaya categories arising from Lefschetz fibrations over an arbitrary marked surface. Up to technicalities, this follows from the cosheaf properties of such a partially wrapped Fukaya category shown in \cite{GPS24}. The statement about the relative left Calabi--Yau property of part (ii) of \Cref{thm:FSschober} may be generalized to this setting given an understanding of the action of the monodromy of the Lefschetz fibration on the non-degenerate Hochschild class of the wrapped Fukaya category of the fiber.\\

\noindent {\bf Periodic topological Fukaya categories}

The author's initial motivating example for treating relative Calabi--Yau structures over an arbitrary base was the construction of relative right $2$-Calabi--Yau structures on $1$-periodic topological Fukaya categories of marked surfaces. These can be considered as $\mathbb{Z}/1\mathbb{Z}$-graded versions of the partially wrapped Fukaya categories. Their construction is the topic of \Cref{subsec:periodicFukaya}. Relative right $2$-Calabi--Yau structure induce $2$-Calabi--Yau Frobenius exact $\infty$-structures, see \cite{Chr22b}, which in turn induce $2$-Calabi--Yau Frobenius extriangulated structures on the homotopy $1$-categories. In the case of $1$-periodic topological Fukaya categories, this exact/extriangulated structure allows for the additive categorification of cluster algebras with coefficients associated with surfaces, see \cite{Chr22b}. 

As $k$-linear $\infty$-categories, with $k$ a field, these periodic categories are smooth but not proper, since the Ext-groups are non-zero in infinitely many degrees. This changes when we work with respect to a different base. For $n$ an integer, the derived category of $n$-periodic chain complexes is equivalent to the derived $\infty$-category of the dg algebra $k[t_{n}^\pm]$ of graded Laurent polynomials, with generator $t_n$ in degree $n$. If $n$ is even, then $k[t_{n}^\pm]$ is graded commutative and thus gives rise to an $\mathbb{E}_\infty$-ring spectrum. If $n$ is odd, we can consider $k[t_n^\pm]$ as a $k[t_{2n}^\pm]$-linear algebra object. Over the base $k[t_{n}^\pm]$, or $k[t_{2n}^\pm]$ if $n$ is odd, the derived $\infty$-category $\mathcal{D}(k[t_{n}^\pm])$ is both smooth and proper and admits left and right $n$-Calabi--Yau structures. 

Considering the $n$-periodic topological Fukaya category over the base $k[t_{n}^\pm]$, or $k[t_{2n}^\pm]$ if $n$ is odd, \Cref{introthm:FukayaCY} yields the desired relative $(n+1)$-Calabi--Yau structure on it.\\

\noindent {\bf Relative Ginzburg algebras over any base ring spectrum}

The derived $\infty$-categories of relative Ginzburg algebras of $n$-angulated surfaces arise as the global sections of parametrized perverse schobers, see \cite{Chr22,Chr21b}. In \Cref{subsec:relGinzburg}, we construct relative left $n$-Calabi--Yau structures on these derived $\infty$-categories. In the case $n=3$, this result is a special case of results \cite{Yeu16,Wu21} on relative left $3$-Calabi--Yau structures on relative Ginzburg algebras of ice quivers with potentials or equivalently relative Calabi--Yau completions. In the case $n=3$, these $\infty$-categories can further be expected to describe the derived $\infty$-categories of the partially wrapped Fukaya categories of the threefolds studied in \cite{Smi15}.

The relevant perverse schobers are locally near each vertex described by the spherical adjunction $f^*\colon \mathcal{D}(k)\leftrightarrow \on{Fun}(S^{n-1},\mathcal{D}(k))\noloc f_*$, where $f\colon S^{n-1}\to \ast$ is the map from the singular simplicial set of the $(n-1)$-sphere to the point. The functor $f_*$ is left $n$-Calabi--Yau. Furthermore, the functor $\bar{f}^*$, obtained by restricting $f^*$ to a functor $\mathcal{D}(k)\to \on{Ind}\on{Fun}(S^{n-1},\mathcal{D}^{\on{perf}}(k))$, is right $n$-Calabi--Yau. 

We can replace $\mathcal{D}(k)$ by $\on{RMod}_R$, with $R$ an $\mathbb{E}_\infty$-ring spectrum, to obtain an $R$-linear version of this adjunction. We expect that both the Calabi--Yau structures of $f_*$ and $\bar{f}^*$ can be lifted to the $R$-linear setting, but only prove that we have a weak right Calabi--Yau structure on $\bar{f}^*$. Its existence is proven using the criterion for the existence of weak right Calabi--Yau structures on spherical functors of \Cref{prop:sphericalrightCY}. 
Via gluing, the Calabi--Yau structure on $\bar{f}^*$ yields weak relative right $n$-Calabi--Yau structures on the locally compact global sections of $R$-linear perverse schobers that generalize ($\on{Ind}$-finite) derived categories of relative Ginzburg algebras.

Besides the Calabi--Yau structures for this class of examples, many other classes of examples of $R$-linear relative Calabi--Yau structures also remain to be worked out.\\

\subsection{Acknowledgments}

This paper is based on and refines results from the author's Ph.D.~thesis, advised by Tobias Dyckerhoff. I thank him for helpful discussions and feedback. I thank Gustavo Jasso for pointing out a gap in the proof of \Cref{lem:functorialinnerHom} in a previous version of the article. I further thank Chris Brav, Fabian Haiden and Bernhard Keller for helpful discussions. The author acknowledges support by the Deutsche Forschungsgemeinschaft under Germany’s Excellence Strategy – EXC 2121 “Quantum Universe” – 390833306. This project has received funding from the European Union’s Horizon 2020 research and innovation programme under the Marie Skłodowska-Curie grant agreement No 101034255. The author is a member of the Hausdorff Center for Mathematics at the University of Bonn (DFG GZ 2047/1, project ID 390685813).

\subsection{Notation}

We generally follow the notation and conventions of \cite{HTT,HA}. In particular, we use the homological grading convention. Given an $\infty$-category $\mathcal{C}$ and two objects $X,Y\in \mathcal{C}$, we denote by $\on{Map}_{\C}(X,Y)$ the mapping space. We denote the $\on{Ind}$-completion of $\C$ by $\on{Ind}(\C)=\on{Ind}_{\omega}(\C)$ and the subcategory of ($\omega$-)compact objects of $\C$ by $\C^{\on{c}}$. Given a functor $F\colon \C\to \D$, we denote its left and right adjoints, if existent, by $\on{ladj}(F)$ and $\on{radj}(F)$, respectively. 

\section{Linear \texorpdfstring{$\infty$}{infinity}-categories and Hochschild homology}\label{sec:dualityHH}

In this section, we review background material on $R$-linear $\infty$-categories, with $R$ an $\mathbb{E}_\infty$-ring spectrum, different notions of duals of bimodules, smooth and proper $R$-linear $\infty$-categories and $R$-linear traces and Hochschild homology. Much of this material appears in a similar form in \cite{HA,HSS17,SAG,BD19,BD21}, though partly with less general proofs. The reader will find that most things work in the $R$-linear setting very similarly to those in the setting of dg categories.

\subsection{Linear \texorpdfstring{$\infty$}{infinity}-categories}\label{extsec}\label{subsec:lininftycats}

Let $\on{Cat}_\infty$ be the $\infty$-category of $\infty$-categories and $\mathcal{S}$ the $\infty$-category of spaces. We denote by $\mathcal{P}r^L\subset \on{Cat}_\infty$ the subcategory of presentable $\infty$-categories and left adjoint functors and by $\mathcal{P}r^R\subset \on{Cat}_\infty$ the subcategory of presentable $\infty$-categories and right adjoint functors. The $\infty$-category $\mathcal{P}r^L$ admits a symmetric monoidal structure, such that a commutative algebra object in $\mathcal{P}r^L$ amounts to a symmetric monoidal presentable $\infty$-category $\mathcal{C}$, satisfying that its tensor product $\mhyphen \otimes \mhyphen \colon \mathcal{C}\times \mathcal{C}\rightarrow \mathcal{C}$ preserves colimits in both entries, see \cite[Section 4.8]{HA}. An example of a commutative algebra object in $\mathcal{P}r^L$ is the $\infty$-category $\on{RMod}_R$ of right module spectra over an $\mathbb{E}_\infty$-ring spectrum $R$. Note that if $R=k$ is a commutative ring, then $\on{RMod}_k$ is equivalent as a symmetric monoidal $\infty$-category to the (unbounded) derived $\infty$-category $\mathcal{D}(k)$, see \cite[7.1.2.13]{HA}.

\begin{definition}
Let $R$ be an $\mathbb{E}_\infty$-ring spectrum. The $\infty$-category 
\[ \on{LinCat}_R\coloneqq \on{Mod}_{\on{RMod}_R}(\mathcal{P}r^L)\] 
of modules in $\mathcal{P}r^L$ over $\on{RMod}_{R}$ is called the $\infty$-category of $R$-linear $\infty$-categories. 
\end{definition}

As noted in \cite[Section D.1.5]{SAG}, $R$-linear $\infty$-categories in the above sense are automatically stable. Given $\C\in \on{LinCat}_R$, we denote the result of the action of an element $C\in \on{RMod}_R$ on $X\in \C$ by $C\otimes X\in\C$.

\begin{definition}[$\!\!${\cite[4.2.1.28]{HA}}]
Let $R$ be an $\mathbb{E}_\infty$-ring spectrum. Let $\mathcal{C}$ be an $R$-linear $\infty$-category and let $X,Y\in \mathcal{C}$. A morphism object is an $R$-module $\on{Mor}_{\mathcal{C}}(X,Y)\in \on{RMod}_R$ equipped with a map $\alpha\colon \on{Mor}_{\mathcal{C}}(X,Y)\otimes X\rightarrow Y$ in $\mathcal{C}$ such that for every object $C\in \on{RMod}_R$, the following composite morphisms is an equivalence of spaces
\[ \on{Map}_{\on{RMod}_R}(C,\on{Mor}_{\mathcal{C}}(X,Y))\rightarrow \on{Map}_{\C}(C\otimes X,\on{Mor}_{\mathcal{C}}(X,Y)\otimes X) \xrightarrow{\alpha\circ\mhyphen} \on{Map}_{\mathcal{C}}(C\otimes X,Y)\,.\]
\end{definition}

We thus have $\pi_i\on{Mor}_{\mathcal{C}}(X,Y)\simeq \pi_0\on{Map}_{\mathcal{C}}(X[i],Y)$ for all $i \in \mathbb{Z}$. 

\begin{remark}
Morphism objects always exist and the formation of morphism objects forms a functor 
\[ \on{Mor}_\mathcal{C}(\mhyphen,\mhyphen)\colon \mathcal{C}^{\on{op}}\times \mathcal{C}\longrightarrow \on{RMod}_R\]
which preserves limits in both entries, see \cite[4.2.1.31]{HA}.
\end{remark}

The $\infty$-category $\on{LinCat}_R$ inherits a symmetric monoidal structure, as the module category over a commutative algebra object. We will often make use of this monoidal product and denote it by $\otimes$. The tensor product of $\C,\D\in \on{LinCat}_R$ arises as the geometric realization of the two-sided bar construction $\on{Bar}(\C,\D)_\ast\colon \Delta^{\on{op}}\to \mathcal{P}r^L$, given informally by the formula $\on{Bar}(\C,\D)_n=\C\otimes_{\mathcal{P}r^L}\on{RMod}_R^{\otimes_{\mathcal{P}r^L} n}\otimes_{\mathcal{P}r^L}\D$, where $\otimes_{\mathcal{P}r^L}$ denotes the symmetric monoidal product of $\mathcal{P}r^L$. The symmetric monoidal $\infty$-category $\on{LinCat}_R$ is closed. As observed in \cite[Section 4.1]{HSS17}, using that $\mathcal{P}r^L$ is closed, the internal Hom in $\on{LinCat}_R$, denoted $\on{Lin}_R(\C,\D)$, can be obtained as the limit of a cosimplicial object obtained from replacing the tensor products in the two-sided bar resolution by the right adjoint internal Homs in $\mathcal{P}r^L$. We record the following functoriality of the internal Hom of $\on{LinCat}_R$: 

\begin{lemma}\label{lem:functorialinnerHom}
There is a functor 
\[ \on{Lin}_R(\mhyphen,\mhyphen)\colon \on{LinCat}_R^{\on{op}}\times \on{LinCat}_R\longrightarrow \on{LinCat}_R\]
satisfying 
\begin{equation}\label{eq:linlin=linotimes}
\on{Lin}_R(\B,\on{Lin}_R(\C,\D))\simeq \on{Lin}_R(\B\otimes \C,\D)\,,
\end{equation}
functorial in $\B,\C\in \on{LinCat}_R^{\on{op}}$ and $\D\in \on{LinCat}_R$. We call $\on{Lin}_R(\C,\D)$ the $R$-linear $\infty$-category of $R$-linear functors from $\C$ to $\D$.
\end{lemma}

\begin{proof}
To construct the functor $\on{Lin}_R(\mhyphen,\mhyphen)$, we follow \cite[4.2.1.31]{HA}. Consider the functor 
\[ \on{Map}_{\on{LinCat}_R}(\mhyphen \otimes \mhyphen,\mhyphen)\colon \on{LinCat}_R^{\on{op}}\times \on{LinCat}_R^{\on{op}}\times \on{LinCat}_R\longrightarrow \mathcal{S} \,.\]
Moving the second copy of $\on{LinCat}_R^{\on{op}}$ to the other side, we obtain a functor \[ \on{LinCat}_R^{\on{op}}\times \on{LinCat}_R\rightarrow \on{Fun}(\on{LinCat}_R^{\on{op}},\mathcal{S})\,,\] 
whose image image lies in the full subcategory of representable presheaves, since $\on{LinCat}_R$ is a closed monoidal $\infty$-category. Composing with the inverse of the Yoneda embedding $\on{LinCat}_R\to \on{Fun}(\on{LinCat}_R^{\on{op}},\mathcal{S})$ yields the functor $\on{Lin}_R(\mhyphen,\mhyphen)$. By construction, we have 
\[ 
\on{Map}_{\on{LinCat}_R}(\B\otimes \C,\D)\simeq \on{Map}_{\on{LinCat}_R}(\B,\on{Lin}_R(\C,\D))\,,
\]
functorial in $\B,\C,\D$. It follows that 
\begin{align*} \on{Map}_{\on{LinCat}_R}(\A,\on{Lin}_R(\B,\on{Lin}_R(\C,\D)))&\simeq \on{Map}_{\on{LinCat}_R}(\A\otimes \B\otimes \C,\D)\\
&\simeq \on{Map}_{\on{LinCat}_R}(\A,\on{LinCat}_R(\B\otimes \C,\D))\,,\end{align*}
functorial in $\A,\B,\C\in \on{LinCat}_R^{\on{op}}$ and $\D\in \on{LinCat}_R$. Composing again with the inverse of the Yoneda embedding shows \eqref{eq:linlin=linotimes}, concluding the proof.
\end{proof}

\begin{lemma}\label{lem:internalHomexact}
Let $\C,D\in \on{LinCat}_R$. 
\begin{enumerate}[(1)]
\item The forgetful functor $\on{Lin}_R(\C,\D)\to \on{Fun}(\C,\D)$ is exact and reflects finite limits and colimits.
\item The subcategory of $\on{Lin}_R(\C,\D)$ of functors admitting an $R$-linear right adjoint (i.e.~of dualizable functors) is closed under finite limits and colimits.
\end{enumerate}
\end{lemma}

\begin{proof}
The functor $\on{Lin}_R(\C,\D)\to \on{Fun}(\C,\D)$ factors through the internal Hom $\on{Fun}^L(\C,\D)$ in $\mathcal{P}r^L$ and the functor $\on{Lin}_R(\C,\D)\to \on{Fun}^L(\C,\D)$ is exact and reflects finite limits and colimits. Since $\on{Fun}^L(\C,\D)$ and $\on{Fun}(\C,\D)$ are stable $\infty$-categories, it thus suffices for part (1) to show that the full inclusion $\on{Fun}^L(\C,\D)\subset \on{Fun}(\C,\D)$ preserves finite colimits and loops. Colimits of colimit preserving functors again preserve colimits since colimits commute with colimits. 

For part (2), it suffices to note that finite limit or colimit diagrams in $\on{Lin}_R(\C,\D)$ turn into finite limit or colimit diagrams upon passing to right adjoints. By part (1), the property of the right adjoints to preserve colimits is preserved under such finite limits and colimits.
\end{proof}

\subsection{Dualizable \texorpdfstring{$\infty$}{infinity}-categories}

We fix an an $\mathbb{E}_\infty$-ring spectrum $R$. Recall that an $R$-linear $\infty$-category $\C\in \on{LinCat}_R$ is called dualizable if it admits a duality datum consisting of evaluation and coevaluation functors 
\[\on{ev}_\C\colon \C^\vee\otimes \C\longrightarrow \on{RMod}_R\]
and 
\[\on{coev}_\C\colon \on{RMod}_R\longrightarrow \C \otimes \C^\vee \,,\]
satisfying the triangle identities. 

Recall further that an $\infty$-category $\C$ is called compactly generated, if $\C\simeq \on{Ind}(\C^{\on{c}})$ is equivalent to the $\on{Ind}$-completion of its subcategory of compact objects. 

\begin{definition}
We denote by $\on{LinCat}^{\on{dual}}_R\subset \on{LinCat}_R$ the subcategory consisting of dualizable $R$-linear $\infty$-categories and dualizable functors, meaning $R$-linear functors whose right adjoint preserves colimits and is thus again $R$-linear. 

We denote by $\on{LinCat}_R^{\on{cpt-gen}}\subset \on{LinCat}_R$ the subcategory consisting of compactly generated $\infty$-categories and compact objects preserving functors.
\end{definition} 

An $R$-linear $\infty$-category $\C\in \on{LinCat}_R$ is dualizable if and only if it is compactly assembled, see \cite[D.7.0.7]{SAG}, which is equivalent to $\C$ being a retract of a compactly generated, presentable and stable $\infty$-category in the $\infty$-category $\mathcal{Pr}^L_{\on{St}}\subset \mathcal{P}r^L$ of stable, presentable $\infty$-categories, see \cite[Prop.~D.7.3.1]{SAG}. 

In particular, any compactly generated $R$-linear $\infty$-category $\C$ is dualizable. In this case, the dual is given by the Ind-completion $\C^\vee\coloneqq \on{Ind}(\mathcal{C}^{\on{c},\on{op}})$, where $\mathcal{C}^{\on{c}}$ denotes the subcategory of compact object and $\mathcal{C}^{\on{c},\on{op}}$ its opposite category. If $\C$ is compactly generated, the evaluation functor $\on{ev}_\C$ restricts along 
\[\C^{\on{c},\on{op}}\times \C^{\on{c}}\subset \C^\vee\times \C\longrightarrow \C^\vee\otimes \C\]
to the restriction of the morphism object functor $\on{Mor}_{\C}(\mhyphen,\mhyphen)$, see \cite[D.7.2.3, D.7.7.6]{SAG}. Under the assumption that $\C$ is compactly generated, an $R$-linear functor is dualizable if and only if it preserves compact objects, see \cite[5.5.7.2]{HTT}. We thus have a fully faithful inclusion $\on{LinCat}_R^{\on{cpt-gen}}\subset \on{LinCat}^{\on{dual}}_R$. We remark that this inclusion preserves both limits and colimits, as follows from the results of \cite{Efi24}.

\begin{remark}
The $\infty$-category $\on{LinCat}^{\on{dual}}_R$ admits (small) colimits and these are preserved by the forgetful functor $\on{LinCat}_R^{\on{dual}}\to \on{LinCat}_R$. This follows from combining \cite[Prop.~1.65]{Efi24} and \cite[3.4.4.6]{HA}.

The $\infty$-category $\on{LinCat}^{\on{dual}}_R$ also admits (small) limits, described in \cite{Efi24}. In the setting of compactly generated $\infty$-categories, limits can be very concretely described: On the level of the underlying $\infty$-categories, the limit of a diagram $D\colon Z\to \on{LinCat}^{\on{cpt-gen}}$ is computed by first restricting the values $D(z)$ to compact objects, for all $z\in Z$, computing the limit of the resulting diagram in $\on{Cat}_\infty$, and then passing to the $\on{Ind}$-completion. This essentially follows from the fact that limits in the subcategory $\mathcal{P}r^{L,\on{cpt-gen}}_{\on{St}}\subset \mathcal{P}r^L$ of compactly generated, stable $\infty$-categories and compact objects preserving functors are computed this way, since $\mathcal{P}r^{L,\on{cpt-gen}}_{\on{St}}\simeq \on{St}^{\on{idem}}$ is equivalent to the $\infty$-category of idempotent complete stable $\infty$-categories.
\end{remark}

\begin{definition}
Given a dualizable $R$-linear functor $F\colon \C\to\D$ between dualizable $R$-linear $\infty$-categories with right adjoint $G$, we define the functor $F^\vee\colon \C^\vee\to \D^\vee$ as the composite
\[ \C^\vee \xlongrightarrow{\on{id}_{\C^\vee}\otimes \on{coev}_{\D}} \C^\vee\otimes \D\otimes \D^\vee \xlongrightarrow{\on{id}_{\C^\vee}\otimes G \otimes \on{id}_{\D^\vee}} \C^\vee\otimes \C\otimes \D^\vee\xrightarrow{\on{ev}_{\C^\vee}\otimes \on{id}_{\D^\vee}}\D^\vee\,.\]
Note that if $\C,\D$ are compactly generated, then 
\[ F^\vee\simeq \on{Ind}(f^{\on{op}})\]
is obtained taking the opposite functor of the restriction $f\colon \C^{\on{c}}\to \D^{\on{c}}$ of $F$ to compact objects and then $\on{Ind}$-completing.
\end{definition}

\begin{definition}\label{def:RYoneda}
Given a dualizable $R$-linear $\infty$-category $\C$, we denote by
\[ \mathcal{Y}\colon \C \simeq \on{Lin}_R(\on{RMod}_R,\C)\xlongrightarrow{\on{ev}_\C\circ ( \on{id}_{\C^\vee}\otimes(\mhyphen) )} \on{Lin}_R(\C^\vee,\on{RMod}_R)\]
the $R$-linear Yoneda embedding. Its inverse is given by
\[ 
\on{Lin}_R(\C^\vee,\on{RMod}_R)\xlongrightarrow{((\mhyphen)\otimes \on{id}_\C)\circ \on{coev}_\C} \on{Lin}_R(\on{RMod}_R,\C)\simeq \C\,.
\]
\end{definition}

\begin{lemma}
Let $F\colon \C\to \D$ be a dualizable $R$-linear functor between dualizable $R$-linear $\infty$-categories with right adjoint $G$. Then the following diagram commutes:
\[
\begin{tikzcd}
\mathcal{C}^\vee \arrow[d, "\simeq"'] \arrow[r, "F^\vee"] & \D^\vee \arrow[d, "\simeq"] \\
{\on{Lin}_R(\C,\on{RMod}_R)} \arrow[r, "(\mhyphen)\circ G"]      & {\on{Lin}_R(\D,\on{RMod}_R)}      
\end{tikzcd}
\]
The functor $(\mhyphen)\circ G$ is left adjoint to $(\mhyphen) \circ F$, and the functor $F^\vee$ is hence dualizable. 
\end{lemma}

\begin{proof}
This readily follows from inspecting the definitions and using the triangle identities for the evaluation and coevaluation functors.
\end{proof}

\begin{lemma}\label{lem:trifunctorial}
Let $\mathcal{C}$ be a dualizable $R$-linear $\infty$-category. There exist an equivalence of $R$-linear functors $\on{RMod}_R^\vee\otimes\mathcal{C}^\vee \otimes \mathcal{C}\to \on{RMod}_R$:
\begin{equation}\label{eq:evotimes} \on{ev}_{\mathcal{C}}(C\otimes Y,Z)\simeq \on{ev}_{\on{RMod}_R}(C,\on{ev}_{\mathcal{C}}(Y,Z))\,.\end{equation}
\end{lemma}

\begin{proof}
We note that $\on{RMod}_R^\vee \simeq \on{RMod}_R$ and thus $\C^\vee\otimes \C\simeq \on{RMod}_R^\vee\otimes\mathcal{C}^\vee \otimes \mathcal{C}$. Composing with this equivalence, both functors in \eqref{eq:evotimes} yield $\on{ev}_\C(\mhyphen,\mhyphen)$, showing their equivalence.
\end{proof}

\subsection{Duals of bimodules}

We again fix a base $\mathbb{E}_\infty$-ring spectrum $R$. Suppose we are given two $R$-linear ring spectra $A,A'$. The $\infty$-category of $A$-$A'$-bimodules $_A\!\on{BMod}_{A'}(\on{RMod}_R)$ is equivalent to the $\infty$-category $\on{Lin}_R(\on{RMod}_A,\on{RMod}_{A'})$ of $R$-linear functors between the respective right module $\infty$-categories, see \cite[4.8.4.1, 4.3.2.7]{HA}. In terms of functors, left and right duals of bimodules, if they exist, correspond to left and right adjoints of the corresponding functors. In the following, we will work with functors instead of bimodules.

Let $\C$ be a dualizable $R$-linear $\infty$-category. We are especially interested in the adjoints of functors $\C\otimes \C^\vee\to \on{RMod}_R$ or $\on{RMod}_R\to \C\otimes \C^\vee$. This corresponds as a special case to studying modules over the enveloping algebra $A^e=A\otimes_R A^{\on{rev}}$ of some $R$-linear ring spectrum $A$. We have the following equivalences.

\begin{lemma}\label{lem:ThetaXi}~
\begin{enumerate}[(1)]
\item The $R$-linear functor $\Theta_\C$, defined as the composite
\[
\on{Lin}_R(\C^\vee\otimes \C,\on{RMod}_R)\xlongrightarrow{\on{id}_{\C}\otimes (\mhyphen)} \on{Lin}_R(\C\otimes \C^\vee\otimes \C,\C)\xlongrightarrow{(\mhyphen)\circ (\on{coev}_\C\otimes \on{id}_\C)} \on{Lin}_R(\C,\C)\,,
\]
is an equivalence with inverse $\Theta_\C^{-1}$ given by
\[
\on{Lin}_R(\C,\C)\xlongrightarrow{\on{id}_{\C^\vee}\otimes (\mhyphen)}\on{Lin}_R(\C^\vee\otimes\C,\C^\vee\otimes\C)\xlongrightarrow{ \on{ev}_\C \circ(\mhyphen)}\on{Lin}_R(\C^\vee\otimes \C,\on{RMod}_R)\,.
\]
\item The $R$-linear functor $\Xi_\C$, defined as the composite
\[
\on{Lin}_R(\on{RMod}_R,\C\otimes \C^\vee)\xlongrightarrow{(\mhyphen)\otimes \on{id}_\C} \on{Lin}_R(\C,\C\otimes \C^\vee\otimes \C)\xlongrightarrow{\on{id}_\C\otimes \on{ev}_\C\circ (\mhyphen)} \on{Lin}_R(\C,\C)\,,
\]
is an equivalence with inverse $\Xi_\C^{-1}$ given by
\[
 \on{Lin}_R(\C,\C)\xlongrightarrow{(\mhyphen)\otimes \on{id}_{\C^\vee}}\on{Lin}_R(\C\otimes \C^\vee,\C\otimes \C^\vee)\xrightarrow{(\mhyphen)\circ \on{coev}_\C}\on{Lin}_R(\on{RMod}_R,\C\otimes \C^\vee)\,.\]
\end{enumerate}
\end{lemma}

\begin{proof}
We begin by proving part (1). The equivalence of $\infty$-categories
\[  \on{Lin}_R(\C,\C)\xrightarrow{\on{Lin}_R(\C,\mathcal{Y})} \on{Lin}_R(\C,\on{Lin}_R(\C^\vee,\on{RMod}_R))\simeq \on{Lin}_R(\C^\vee\otimes \C,\on{RMod}_R)\,,\]
with $\mathcal{Y}$ the $R$-linear Yoneda embedding, maps an endofunctor $X\colon \C\to\C$ to $\on{ev}_\C\circ (\on{id}_{\C^\vee}\otimes X)$. This shows that $\Theta_\C^{-1}$ is essentially surjective. The triangle identity for $\on{ev}_\C$ and $\on{coev}_\C$ implies that $\Theta_\C\circ \Theta_\C^{-1}\simeq \on{id}_{\on{Lin}_R(\C,\C)}$. It follows that $\Theta_\C^{-1}$ is faithful, and in fact a split inclusion on Hom spaces. Using that all objects $Y\in \on{Lin}_R(\C^\vee\otimes \C,\on{RMod}_R)$ are of the form $Y\simeq \on{ev}_\C\circ (\on{id}_{\C^\vee}\otimes X)$, we find $\Theta_\C^{-1}\circ \Theta_\C(Y)\simeq Y$. Using that $\Theta_\C$ and $\Theta_\C^{-1}$ are exact, it follows that $\Theta_\C^{-1}$ is also full, showing that $\Theta_\C^{-1}$ is an equivalence. Since $\Theta_\C\circ \Theta_\C^{-1}\simeq \on{id}_{\on{Lin}_R(\C,\C)}$, the inverse of $\Theta_\C^{-1}$ is given by $\Theta_\C$. 

For part (2), a similar argument as above applies, using that the equivalence of $\infty$-categories
\[
\on{Lin}_R(\C,\C)\simeq \on{Lin}_R(\on{RMod}_R,\on{Lin}_R(\C,\C))\simeq \on{Lin}_R(\on{RMod}_R,\C\otimes \C^\vee)
\]
maps a functor $X\colon \C\to \C$ to $(X\otimes \on{id}_{\C^\vee})\circ \on{coev}_{\C}$.
\end{proof}

\begin{notation}
We denote by $\tau$ the $R$-linear equivalence $\C\otimes \C^\vee\simeq \C^\vee\otimes \C$ which permutes the factors.
\end{notation}

We can use the equivalences $\Theta_\C$ and $\Xi_\C$ to define the dual of an $R$-linear endofunctor $\C\to \C$, considered as a functor $\C^\vee\otimes \C\to \on{RMod}_R$ or $\on{RMod}_R \to \C\otimes \C^\vee$.

\begin{definition}\label{dual!*def}
Let $X\in \on{Lin}_R(\C,\C)$ be an $R$-linear endofunctor.
\begin{enumerate}
\item We call $X$ left dualizable if $\Theta_{\C}^{-1}(X)$ admits an $R$-linear left adjoint. In this case, we call
\[X^!\coloneqq \Xi_\C(\tau \circ \on{ladj}(\Theta_\C^{-1}(X)))\in \on{Lin}_R(\C,\C)\] the left dual of $X$. 
\item We call $X$ right dualizable if $\Theta_\C^{-1}(X)$ admits an $R$-linear right adjoint. In this case, we call 
\[X^\ast\coloneqq \Xi_\C(\tau \circ \on{radj}(\Theta_\C^{-1}(X)))\in \on{Lin}_R(\C,\C)\]
the right dual of $X$.
\end{enumerate}
\end{definition}

\begin{proposition}\label{prop:leftdualasradj}
Let $X\in \on{Lin}_R(\C,\C)$.
\begin{enumerate}[(1)]
\item $X$ is left dualizable if and only if $\Xi_\C^{-1}(X)$ admits a right adjoint, and in this case 
\[ X^!\simeq \Theta_\C(\tau \circ \on{radj}(\Xi_\C^{-1}(X)))\,.\]
\item $X$ is right dualizable if and only if $\Xi_\C^{-1}(X)$ admits a left adjoint, and in this case 
\[ X^*\simeq \Theta_\C(\tau \circ \on{ladj}(\Xi_\C^{-1}(X)))\,.\]
\end{enumerate}
\end{proposition}

\begin{proof}
The lemma follows from \Cref{lem:coev=dualev} below and the observations that for any dualizable $R$-linear functor $F\colon \A\to \B$, we have
\[
\on{Lin}_R(\on{radj}(F),\on{RMod}_R) \dashv \on{Lin}_R(F,\on{RMod}_R)\,.
\]
\end{proof}

\begin{lemma}\label{lem:coev=dualev}
Let $X\in \on{Lin}_R(\C,\C)$. There are commutative diagrams,
\[
\begin{tikzcd}[column sep=60pt]
\on{RMod}_R \arrow[r, "\Xi_\C^{-1}(X)"] \arrow[d, "\mathcal{Y}"']                               & \C\otimes \C^\vee \arrow[d, "\mathcal{Y}"'] \\
{\on{Lin}_R(\on{RMod}_R,\on{RMod}_R)} \arrow[r, "{\on{Lin}_R(\Theta_\C^{-1}(X),\on{RMod}_R)}"'] & {\on{Lin}_R(\C^\vee\otimes \C,\on{RMod}_R)}
\end{tikzcd}
\]
and 
\[
\begin{tikzcd}[column sep=60pt]
\C^\vee\otimes \C \arrow[d, "\mathcal{Y}"'] \arrow[r, "\Theta_\C^{-1}(X)"]                           & \on{RMod}_R \arrow[d, "\mathcal{Y}"]  \\
{\on{Lin}_R(\C\otimes \C^\vee,\on{RMod}_R)} \arrow[r, "{\on{Lin}_R(\Xi_{\C}^{-1}(X),\on{RMod}_R)}"'] & {\on{Lin}_R(\on{RMod}_R,\on{RMod}_R)}
\end{tikzcd}
\]
with $\mathcal{Y}$ the $R$-linear Yoneda embedding, see \Cref{def:RYoneda}.
\end{lemma}

\begin{proof}
The evaluation functor
\[ \on{ev}_{\C\otimes \C^\vee}\colon \C^\vee \otimes \C\otimes \C\otimes \C^\vee \simeq  (\C\otimes \C^\vee)^\vee\otimes \C\otimes \C^\vee\longrightarrow \on{RMod}_R\]
is, after reordering the factors of the tensor product, given by the tensor product of the evaluation functors of $\C$ and $\C^\vee$. Note that these two evaluation functors are themselves equivalent, up to composition with $\tau\colon \C\otimes \C^\vee\simeq \C^\vee\otimes \C$. The Yoneda embedding 
\[\on{Lin}_R(\on{RMod}_R,\C\otimes \C^\vee)\simeq \C\otimes \C^\vee \xlongrightarrow{\mathcal{Y}} \on{Lin}_R(\C^\vee\otimes \C,\on{RMod}_R)\] 
is thus given by the functor 
\begin{equation}\label{eq:yonedaCCvee}\left(\on{ev}_{\C}\otimes \on{ev}_{\C}\right)\circ \left(\on{id}_{\C^\vee}\otimes (\mhyphen) \otimes \on{id}_{\C}\right)\,.\end{equation}
Using this, the commutativity directly follows from the triangle identities for the evaluation and coevaluation functors when inserting the descriptions of $\Theta^{-1}_\C(X),\Xi^{-1}_\C(X)$ in \Cref{lem:ThetaXi}. 
\end{proof}

\begin{remark}\label{dualfunrem}
We denote by $\on{Lin}_R^{\on{ld}}(\C,\C)\subset \on{Lin}_R(\C,\C)$ the stable subcategories of left dualizable functors. We similarly denote by $\on{Lin}_R^{\on{rd}}(\C,\C)\subset \on{Lin}_R(\C,\C)$ the stable subcategory of right dualizable functors. Since passing to adjoints is functorial, see \cite[5.2.6.2]{HTT}, there are exact functors 
\[ (\mhyphen)^!\colon \on{Lin}_R^{\on{ld}}(\C,\C)\rightarrow \on{Lin}_R(\C,\C)^{\on{op}}\]
and
\[ (\mhyphen)^*\colon \on{Lin}_R^{\on{rd}}(\C,\C)\rightarrow \on{Lin}_R(\C,\C)^{\on{op}}\,.\]
More concretely, we find that $\on{Lin}_R^{\on{ld}}(\C,\C)=\on{Lin}_R(\C,\C)^{\on{c}}$ is given by the subcategory of compact objects. This follows from the observation that a compact object in 
\[ \C\otimes \C^\vee \overset{\mathcal{Y}}{\simeq} \on{Lin}_R(\C^\vee\otimes \C,\on{RMod}_R)\simeq \on{Lin}_R(\C,\C)\] gives via the Yoneda embedding rise to an exact functor $\C^\vee\otimes \C\to \on{RMod}_R$ which also preserves filtered limits, and hence all limits, and thus admits a left adjoint by the adjoint functor theorem. If $\C$ is compactly generated, then an endofunctor is right dualizable if and only if its image under $\Theta_\C^{-1}$ in $\on{Lin}_R(\C^\vee\otimes \C,\on{RMod}_R)$ preserves compact objects, since in this case the right adjoint preserves colimits and is thus $R$-linear.
\end{remark}

\begin{lemma}\label{lem:pullingbackendofunctors}
Let $F\colon \C\to \D$ be a morphism in $\on{LinCat}_R^{\on{dual}}$ with $R$-linear right adjoint $G$.
\begin{enumerate}[(1)]
\item There exists a commutative diagram 
\[
\begin{tikzcd}
{\on{Lin}_R(\on{RMod}_R,\C\otimes \C^\vee)} \arrow[d, "(F\otimes F^\vee)\circ (\mhyphen)"'] \arrow[r, "\Xi_\C"] & {\on{Lin}_R(\C,\C)} \arrow[d, "F\circ (\mhyphen)\circ G"] \\
{\on{Lin}_R(\on{RMod}_R,\D\otimes \D^\vee)} \arrow[r, "\Xi_\D"]                                                 & {\on{Lin}_R(\D,\D)}                 
\end{tikzcd}
\]
In particular, it follows that the $R$-linear functor
\[ F_!\coloneqq F\circ(\mhyphen)\circ G\colon \on{Lin}_R(\C,\C)\rightarrow \on{Lin}_R(\D,\D)\]
preserves compact objects, meaning left dualizable functors.
\item There exists a commutative diagram 
\[
\begin{tikzcd}
{\on{Lin}_R(\D^\vee\otimes \D,\on{RMod}_R)} \arrow[d, "(\mhyphen)\circ (F^\vee\otimes F)"'] \arrow[r, "\Theta_\D"] & {\on{Lin}_R(\D,\D)} \arrow[d, "G\circ (\mhyphen)\circ F"] \\
{\on{Lin}_R(\C^\vee\otimes \C,\on{RMod}_R)} \arrow[r, "\Theta_\C"]                                                 & {\on{Lin}_R(\C,\C)}                                      
\end{tikzcd}
\]
In particular, it follows that the $R$-linear functor
\[ F^*\coloneqq G\circ(\mhyphen)\circ F\colon \on{Lin}_R(\D,\D)\rightarrow \on{Lin}_R(\C,\C)\]
preserves right dualizable functors.
\end{enumerate}
\end{lemma}

\begin{proof}
We only prove part (1), part (2) is analogous. It follows from
\[ \on{ev}_\D \circ (F^\vee\otimes \on{id}_\D)\simeq \on{ev}_\C\circ (\on{id}_{\C^\vee}\otimes\ G)\]
that 
\begin{align*} 
\Xi_{\D}((F\otimes F^\vee)\circ \alpha)&\simeq ( \on{id}_\D\otimes\on{ev}_{\D})\circ (F\otimes F^\vee\otimes \on{id}_\D)\circ (\alpha\otimes \on{id}_\D)\\
& \simeq (\on{id}_\D\otimes \on{ev}_{\C})\circ (F\otimes \on{id}_{\D^\vee}\otimes \ G)\circ (\alpha \otimes \on{id}_\D)\\
&\simeq F\circ \Xi_\C(\alpha)\circ G\,,
\end{align*}
functorial in $\alpha\colon \on{RMod}_R\to \C\otimes \C^\vee$.
\end{proof}

\subsection{Smooth and proper linear \texorpdfstring{$\infty$}{infinity}-categories}

We fix an $\mathbb{E}_\infty$-ring spectrum $R$ and a dualizable $R$-linear $\infty$-category $\C$.

\begin{definition}~
\begin{enumerate}[(1)]
\item The $\infty$-category $\C$ is called smooth if $\on{id}_{\C}\in \on{Lin}_R(\C,\C)$ is left dualizable. In this case, the left dual $\on{id}_{\C}^!$ is also called the inverse dualizing bimodule.
\item The $\infty$-category $\C$ is called proper if the functor $\on{id}_\C$ is right dualizable.
\end{enumerate}
Note that we do not require smooth or proper $\infty$-categories to be compactly generated. However, if $\C$ is compactly generated, then $\C$ being proper is equivalent to the assertion that for any two compact objects $X,Y\in \C^c$ the $R$-linear morphism object $\on{Mor}_{\C}(X,Y)\in \on{RMod}_R$ is compact.
\end{definition}

We denote $(\mhyphen)^*=\on{Mor}_{\on{RMod}_R}(\mhyphen,R)\colon (\on{RMod}_R^{\on{c}})^{\on{op}}\to \on{RMod}_R^{\on{c}}$

\begin{definition}\label{def:serre}
Suppose that $\C$ is compactly generated and proper. We call an $R$-linear endofunctor $U\in \on{Lin}_R(\C,\C)$ a Serre functor of $\C$ if there exists a natural equivalence
\[ \on{Mor}_{\C}(\mhyphen_1,\mhyphen_2)\simeq \on{Mor}_{\C}(\mhyphen_2,U(\mhyphen_1))^\ast\colon \C^{\on{c},\on{op}}\times \C^{\on{c}}\rightarrow \on{RMod}_R^{\on{c}}\,.\]
\end{definition}

\begin{lemma}\label{lem:Serreunique}
Suppose that $\C$ is compactly generated and proper and let $U,U'$ be two Serre functors of $\C$. Then $U\simeq U'\in \on{Lin}_R(\C,\C)$. 
\end{lemma}
\begin{proof}
Since $U$ and $U'$ are both Serre functors, there exist natural equivalences 
\[ \on{Mor}_{\C}(\mhyphen_1,U(\mhyphen_2))\simeq \on{Mor}_{\C}(\mhyphen_2,\mhyphen_1)^\ast \simeq \on{Mor}_{\C}(\mhyphen_1,U'(\mhyphen_2))\,.\]
Applying $\on{Map}_{\on{RMod}_R}(R,\mhyphen)$ to this equivalence yields 
\[\on{Map}_{\C^{\on{c}}}(\mhyphen_1,U(\mhyphen_2))\simeq \on{Map}_{\C^{\on{c}}}(\mhyphen_1,U'(\mhyphen_2))\,.\] 
It follows that $U\simeq U'$ on $\C^{\on{c}}$ by (a corollary of) the Yoneda lemma, see for instance \cite[Cor.~5.8.14]{Cis}. Passing to $\on{Ind}$-completions shows $U\simeq U'$.
\end{proof}

\begin{lemma}\label{lem:serreinv}~
\begin{enumerate}[(1)]
\item If $\C$ is proper and compactly generated, the right dual $\on{id}_{\C}^\ast$ is a Serre functor of $\C$.
\item If $\C$ is smooth and proper, the functors $\on{id}_{\C}^\ast$ and $\on{id}_{\C}^!$ are inverse equivalences.
\end{enumerate}
\end{lemma}

\begin{remark}
Part (1) of \Cref{lem:serreinv} is stated without proof in \cite[11.1.5.2]{SAG}.
\end{remark}

\begin{proof}[Proof of \Cref{lem:serreinv}]
We begin by proving part (1). We denote by 
\[ \widehat{\on{Mor}}_{\C}(\mhyphen,\mhyphen)\colon \C^\vee\times \C \to \on{RMod}_R\]
the functor obtained by passing to $\on{Ind}$-completions from the restriction of $\on{Mor}_{\C}(\mhyphen,\mhyphen)$ to $\C^{\on{c},\on{op}}\times \C^{\on{c}}$. Let $X\in \C^{\on{c}}$ and consider the adjunction
\[ (\mhyphen)\otimes \widehat{\on{Mor}}_{\C}(X,\mhyphen)\colon \on{RMod}_R\longleftrightarrow \on{Lin}_R(\C,\on{RMod}_R)\noloc \on{Mor}_{\on{Lin}_R(\C,\on{RMod}_R)}( \widehat{\on{Mor}}_{\C}(X,\mhyphen),\mhyphen)\,.\]
Using that $\widehat{\on{Mor}}_{\C}(X,\mhyphen)\simeq \widehat{\on{Mor}}_{\C^\vee}(\mhyphen,X)$ and the fully faithfulness of the $R$-linear Yoneda embedding of $\C^\vee$, we find that the right adjoint is given by evaluation at $X$, i.e.
\[ \on{Mor}_{\on{Lin}_R(\C,\on{RMod}_R)}( \widehat{\on{Mor}}_{\C}(X,\mhyphen),\mhyphen) \simeq \on{ev}_X\,.\]
Using the identification 
\[\on{Lin}_R(\C,\on{RMod}_R)\otimes \on{Lin}_R(\C^\vee,\on{RMod}_R)\simeq \on{Lin}_R(\C\otimes \C^\vee,\on{RMod}_R)\]
we define the functor 
\[ \on{ev}_X'\colon \on{Lin}_R(\C\otimes \C^\vee,\on{RMod}_R)\xlongrightarrow{\on{ev}_X\otimes \on{id}_{ \on{Lin}_R(\C^\vee,\on{RMod}_R)}} \on{Lin}_R(\C^\vee,\on{RMod}_R)\]
with left adjoint
\[ \on{ladj}(\on{ev}_X')\colon \on{Lin}_R(\C^\vee,\on{RMod}_R)\xlongrightarrow{\left((\mhyphen)\otimes \widehat{\on{Mor}}_{\C}(X,\mhyphen)\right)\otimes \on{id}} \on{Lin}_R(\C\otimes \C^\vee,\on{RMod}_R)\,.\]
Informally, the left adjoint $\on{ladj}(\on{ev}_X')$ is given by 
\[ (c'\mapsto F(c'))\mapsto \left((c\otimes c')\mapsto \widehat{\on{Mor}}_{\C}(X,c)\otimes F(c')\right)\,.\]

The right adjoint of the evaluation functor $\on{ev}_\C\colon \C^\vee\otimes \C\to \on{RMod}_R$ is equivalent to $(\on{id}_{\C^\vee}\otimes \on{id}_{\C}^*)\circ \on{coev}_{\C^\vee}$. Using the description the the Yoneda embedding in \eqref{eq:yonedaCCvee}, it follows that the right adjoint of the functor 
\[ \widetilde{\on{ev}}\colon\on{Lin}_R(\C\otimes \C^\vee,\on{RMod}_R)\simeq \C^\vee\otimes \C\xlongrightarrow{\on{ev}_\C}\on{RMod}_R\]
is equivalent to $(\mhyphen)\otimes \on{ev}_{\C^\vee} \circ (\on{id}^*_\C\otimes \on{id}_{\C^\vee})$. 

In total, we obtain that the right adjoint of 
\[ \on{Lin}_R(\C^\vee,\on{RMod}_R)\xlongrightarrow{\on{ladj}(\on{ev}_X')} \on{Lin}_R(\C\otimes \C^\vee,\on{RMod}_R)\xlongrightarrow{\widetilde{\on{ev}}}\on{RMod}_R\]
is given by $(\mhyphen)\otimes \widehat{\on{Mor}}_{\C}(\mhyphen,\on{id}_{\C}^*(X))$. Using the above adjunctions, the fully faithfulness of the $R$-linear Yoenda embedding and \Cref{lem:trifunctorial}, we find the following equivalences in $\on{RMod}_R$, functorial in $(X,Y)\in \C^{\on{c}}\times \C^{\on{c},\on{op}}$:
\begin{align*}
\on{Mor}_{\C}(X,Y)^* =&\on{Mor}_{\on{RMod}_R}(\on{Mor}_{\C}(X,Y),R)\\
\simeq & \on{Mor}_{\on{RMod}_R}(\widetilde{\on{ev}}\circ \on{ladj}(\on{ev}_X')(\widehat{\on{Mor}}_{\C}(\mhyphen,Y)),R)\\
\simeq &\on{Mor}_{\on{Lin}(\C^\vee,\on{RMod}_R)}(\widehat{\on{Mor}}_{\C}(\mhyphen,Y),\widehat{\on{Mor}}_{\C}(\mhyphen,\on{id}_{\C}^*(X)))\\
\simeq & \on{Mor}_{\C}(Y,\on{id}_{\C}^*(X))\,.
\end{align*}
This shows that that $\on{id}_{\C}^*$ is indeed a Serre functor.  

We proceed with proving part (2). We have 
\begin{equation}\label{eq:triangleA} \on{id}_\C\simeq (\on{id}_\C\otimes \on{ev}_\C)\circ (\on{coev}_\C\otimes \on{id}_\C)\end{equation}
and passing to the right adjoint yields
\[ \on{id}_\C\simeq (\on{radj}(\on{coev}_\C)\otimes \on{id}_\C)\circ (\on{id}_\C\otimes\on{radj}(\on{ev}_\C))\,.\]
We have
\[ \on{radj}(\on{ev}_\C)\simeq  (\on{id}_{\C^\vee}\otimes\on{id}_\C^*)\circ \on{coev}_{\C^\vee}\] 
and by \Cref{prop:leftdualasradj} further
\[ \on{radj}(\on{coev}_\C)\simeq \on{ev}_{\C^\vee}\circ (\on{id}_\C^!\otimes \on{id}_{\C^\vee})\,.\]
Combining the above equivalences yields $\on{id}_\C^*\circ \on{id}_\C^!\simeq \on{id}_\C$. The identity $\on{id}_\C^!\circ \on{id}_\C^*\simeq \on{id}_\C$ arises from a similar argument by passing to the left adjoint of \eqref{eq:triangleA}.
\end{proof}

\begin{definition}\label{def:Cfin}
Given a compactly generated $R$-linear $\infty$-category $\C$, we denote by $\C^{\on{fin}}\subset \C$ the full subcategory of objects $Y$, satisfying that $\on{Mor}_{\C}(X,Y)\in \on{RMod}_R$ is compact for all $X\in \C^{\on{c}}$. We also refer to the objects of $\C^{\on{fin}}$ as finite.
\end{definition}

The following lemma provides the analog of part (1) of \Cref{lem:serreinv} for smooth, but not necessarily proper, $R$-linear $\infty$-categories.

\begin{lemma}\label{lem:inverseSerreforsmooth}
Let $\C$ be a compactly generated and smooth $R$-linear $\infty$-category. Then 
\[ \on{Mor}_{\C}(X,Y)^*\simeq \on{Mor}_{\C}(\on{id}^!(Y),X)\,,\]
functorial in $X\in \C^{\on{c}}$ and $Y\in (\C^{\on{fin}})^{\on{op}}$. 

In particular, this shows that if $\on{id}_\C^!$ is an equivalence, then $\C^{\on{fin}}\subset \C^{\on{c}}$.
\end{lemma}

\begin{proof}[Proof of \Cref{lem:inverseSerreforsmooth}.]
The exact inclusion $\C^{\on{fin}}\subset \C$ gives rise to an $R$-linear functor $\on{Ind}\C^{\on{fin}}\to \C$. The $R$-linear functor
\[ \on{ev}_\C^{\on{fin}}\colon \C^\vee \otimes \on{Ind}\C^{\on{fin}}\longrightarrow  \C^\vee\otimes \C\xlongrightarrow{\on{ev}_\C} \on{RMod}_R \]
preserves compact objects by the definition of $\C^{\on{fin}}$ and thus admits an $R$-linear right adjoint $\on{radj}(\on{ev}_\C^{\on{fin}})$. We define the $R$-linear functor $U\colon \C\to \on{Ind}\C^{\on{fin}}$ as the composite 
\[ \C\xlongrightarrow{\on{id}_\C\otimes \on{radj}(\on{ev}_\C^{\on{fin}})}\C\otimes \C^\vee \otimes \on{Ind}\C^{\on{fin}}\xlongrightarrow{\on{ev}_\C\otimes \on{id}_{\on{Ind}\C^{\on{fin}}}} \on{Ind}\C^{\on{fin}}\,.\]
The functor $U$ admits a left adjoint, given by the composite 
\[ \on{Ind}\C^{\on{fin}}\xlongrightarrow{\on{ladj}(\on{ev}_\C)\otimes \on{id}_{\on{Ind}\C^{\on{fin}}}} \C\otimes \C^\vee \otimes \on{Ind}\C^{\on{fin}} \xlongrightarrow{\on{id}_{\C}\otimes \on{ev}_\C^{\on{fin}}} \C\,,\]
which describes the composite of $\on{id}_{\C}^!$ with the $R$-linear functor $\on{Ind}\C^{\on{fin}}\to \C$.

The proof of part (1) of \Cref{lem:serreinv} adapts with minimal changes to this setting and shows that 
\begin{equation}\label{eq:MorYUX} \on{Mor}_{\C}(X,Y)^*\simeq \on{Mor}_{\C}(Y,U(X))\,,\end{equation}
functorial in $Y\in \C^{\on{fin},\on{op}}$ and $X\in \C^{\on{c}}$. By the above adjunction, we have 
\[ \on{Mor}_{\C}(Y,U(X))\simeq \on{Mor}_{\C}(\on{id}_{\C}^!(Y),X),\]
which combined with \eqref{eq:MorYUX} yields the desired equivalence.
\end{proof}

\subsection{Traces}

We fix an $\mathbb{E}_\infty$-ring spectrum $R$.

\begin{definition}
Let $\C$ be a dualizable $R$-linear $\infty$-category. Let $E\colon \C\to\C$ be an $R$-linear endomorphism. The trace $\on{Tr}(E)$ of $E$ is defined as the $R$-linear endomorphism
\[ \on{RMod}_R \xrightarrow{\on{coev}_\C} \C\otimes \C^\vee \xrightarrow{E\otimes \on{id}_{\C^\vee}} \C\otimes \C^\vee \xrightarrow{\tau} \C^\vee \otimes \C\xrightarrow{\on{ev}_\C} \on{RMod}_R\,.\]
\end{definition}

In the following, we recall the construction of Hoyois-Scherotzke-Sibilla \cite{HSS17} of the $S^1$-equivariant $(\infty,1)$-categorical trace functor. This functor will also give rise to the Hochschild homology functor.  

As a model for $(\infty,2)$-categories, we use Barwick's complete $2$-fold Segal spaces. We let $\bm{C}$ be a symmetric monoidal $(\infty,2)$-category. We will primarily be interested in the case where $\bm{C}={\bf LinCat}_R$ is the $(\infty,2)$-category of $R$-linear stable and presentable $\infty$-categories, see \cite[Section 4.4]{HSS17}, where the notation ${\bf Pr}^L(\on{RMod}_R)$ is used. We have an associated complete ($1$-fold) Segal space (which models an $(\infty,1)$-category) of symmetric
monoidal oplax transfors \cite{JFS17} 
\[ \on{End}(\bm{C})=\on{Fun}^{\on{oplax}}_{\otimes}(\on{Fr}^{\on{rig}}(B\mathbb{N}),\bm{C})\,,\]
where $\on{Fr}^{\on{rig}}(B\mathbb{N})$ denotes the free rigid\footnote{By rigid, we mean that all objects are dualizable.} $(\infty,2)$-category generated by $B\mathbb{N}$, see also \cite[Def.~2.2]{HSS17}. The objects and morphisms in $\on{End}(\bm{C})$ can be concretely described as follows:
\begin{itemize}
\item Objects are given by pairs $(\C,E)$ with $\C\in \bm{C}$ dualizable and $E\colon \C\to \C$ an endomorphism in $\bm{C}$.
\item A morphism $(F,\alpha)\colon (\C,E)\to (\D,E')$ corresponds to a right dualizable morphism $F\colon \C\to \D$ in $\bm{C}$ together with an oplax-commutative square
\[
\begin{tikzcd}
\C \arrow[r, "E"] \arrow[d, "F"'] & \C \arrow[d, "F"] \arrow[ld, "\alpha"', Rightarrow] \\
\D \arrow[r, "E'"']               & \D                                                 
\end{tikzcd}
\]
meaning a $2$-morphism $\alpha\colon FE\Rightarrow E'F$ in $\bm{C}$.
\end{itemize}
We also consider the complete Segal space $\Omega \bm{C}$ of endomorphisms of the monoidal unit of $\bm{C}$. The trace defines by \cite[Def.~2.9,2.11]{HSS17} a symmetric monoidal functor between complete Segal spaces
\[ \Tr\colon \on{End}(\bm{C})\to \Omega \bm{C}\,,\]
which is natural in $\bm{C}$. Given a morphism $(F,\alpha)\colon (\C,E)\to (\D,E')$ in $\on{End}({\bf LinCat}_R)$, the morphism 
\[ \on{Tr}(F,\alpha)\colon \on{Tr}(E)\to \on{Tr}(E')\]  
in $\on{Map}_{\on{LinCat}_R}(\on{RMod}_R,\on{RMod}_R)$ can be identified with the composition of the following natural transformations.
\[
\begin{tikzcd}[column sep=large, row sep=large]
\on{RMod}_R \arrow[r, "\on{coev}_{\C}"] \arrow[d, "\on{id}"'] & \C\otimes \C^\vee \arrow[d, "F\otimes F^\vee" description] \arrow[r, "E\otimes \on{id}_{\C^\vee}"] \arrow[ld, Rightarrow, "\nu"] & \C\otimes \C^\vee \arrow[d, "F\otimes F^\vee" description] \arrow[ld, "\alpha\otimes \on{id}" description, Rightarrow] \arrow[r, "\tau"] & \C^\vee\otimes \C \arrow[d, "F^\vee\otimes F" description] \arrow[ld, "\simeq"', Rightarrow] \arrow[r, "\on{ev}_\C"] & \on{RMod}_R \arrow[d, "\on{id}"] \arrow[ld, Rightarrow, "\epsilon"] \\
\on{RMod}_R \arrow[r, "\on{coev}_{\D}"']                      & \D\otimes \D^\vee \arrow[r, "E'\otimes \on{id}_{\D^\vee}"']                                                               & \D\otimes \D^\vee \arrow[r, "\tau"']                                                                                                     & \D^\vee\otimes \D \arrow[r, "\on{ev}_\D"']                                                                           & \on{RMod}_R                                            
\end{tikzcd}
\]
Using the triangle identities, the natural transformation $\nu$ above is defined as the composite of
\begin{align*}
&(F\otimes F^\vee)\circ \on{coev}_\C\\
&\simeq(F\otimes \on{id}_{\D^\vee})\circ (\on{id}_{\C}\otimes\on{ev}_{\C^\vee}\otimes \on{id}_{\D^\vee})\circ (\on{id}_{\C}\otimes\on{id}_{\C^\vee}\otimes G\otimes \on{id}_{\D^\vee})\circ (\on{id}_{\C}\otimes \on{id}_{\C^\vee} \otimes \on{coev}_{\D})\circ \on{coev}_\C\\
&\simeq (FG\otimes \on{id}_{\D^\vee})\circ \on{coev}_\D
\end{align*}
and
\[
(FG\otimes \on{id}_{\D^\vee})\circ \on{coev}_\D\xlongrightarrow{(\counit \otimes \on{id}_{\D^\vee})\circ \on{coev}_\D} \on{coev}_\D\,,
\]
where $G$ denotes the right adjoint of $F$ and $\counit$ the counit. The natural transformation $\epsilon$ is defined similarly. Further, $\on{Tr}(F,\alpha)$ is also equivalent to the composite 
\begin{equation}\label{eq:trace} \on{Tr}(E)\xrightarrow{\on{Tr}(E\unit)} \on{Tr}(EGF)\xrightarrow{\beta} \on{Tr}(FEG) \xrightarrow{\on{Tr}(\on{id}_\D,\alpha G)}\on{Tr}(E'FG)\xrightarrow{\on{Tr}(E'\counit)}\on{Tr}(E')\,,\end{equation}
where $\unit$ denotes the unit and $\beta$ is the apparent morphism, see \cite[Lemma 4.1]{BD21}.

Finally, we turn to the $S^1$-functoriality of the trace functor. For this, consider the complete Segal space
$\on{Aut}(\bm{C})=\on{Fun}^{\on{oplax}}_{\otimes}(\on{Fr}^{\on{rig}}(B\mathbb{Z}),\bm{C})$, which comes with a symmetric monoidal inclusion $\on{Aut}(\bm{C})\subset \on{End}(\bm{C})$. The objects of $\on{Aut}(\bm{C})$ are given by pairs $(\C,E)$ with $\C\in \bm{C}$ dualizable and $E\colon \C\to \C$ an equivalence. The self-action of the circle group $S^1=B\mathbb{Z}$ induces an action on $\on{Aut}(\bm{C})$, which in turn induces an action on the space $\on{Map}(\on{Aut}(\bm{C}),\Omega(\bm{C}))$ of functors, natural in $\bm{C}$. The trace can be exhibited as a homotopy fixed point of the $S^1$-action on $\on{Map}(\on{Aut}(\bm{C}),\Omega(\bm{C}))$, see \cite[Thm.~2.14]{HSS17}. This $S^1$-invariance datum results in an $S^1$-action on the trace of any pair $(\C,\on{id}_\C)\in \on{End}(\bm{C})$ as well as an $S^1$-equivariant map $\on{Tr}(F,\on{id}_F)$ for any morphism $(F,\on{id}_F)\colon (\C,\on{id}_{\C})\to (\D,\on{id}_\D)$ in $\on{End}(\bm{C})$. 

\begin{remark}\label{rem:traceofunit}
The datum of the $S^1$-invariance of the trace functor is natural in the symmetric monoidal $(\infty,2)$-category $\bm{C}$. Thus, given a symmetric monoidal functor $F\colon \bm{C}\to \bm{D}$ and a dualizable object $\mathcal{C}\in \bm{C}$, we have an $S^1$-equivariant equivalence $F(\on{Tr}(\mathcal{C},\on{id}_\mathcal{C}))\simeq \on{Tr}(F(\mathcal{C}),\on{id}_{F(\mathcal{C})})$. 

We can specialize this to the symmetric monoidal functor $\ast \to \bm{C}$ from the terminal symmetric monoidal $(\infty,2)$-category, mapping the unique object of $\ast$ to the unit object $1_C$ of $\bm{C}$. Since the $S^1$-action on the trace $\on{Tr}(\ast,\on{id}_\ast)$ is necessarily trivial, this shows that the $S^1$-action on $\on{Tr}(1_C,\on{id}_{1_C})$ is also trivial. In particular, this holds for $\bm{C}={\bf LinCat}_R$ with $1_C=\on{RMod}_R$ and $\on{Tr}(\on{RMod}_R,\on{id}_{\on{RMod}_R})\simeq R$.
\end{remark}

\subsection{Hochschild homology}

We fix an $\mathbb{E}_\infty$-ring spectrum $R$. Given a dualizable  $R$-linear $\infty$-category $\C$, its $R$-linear Hochschild homology is defined as the value of the trace at $R$:
\[
	\HH(\C) \coloneqq \on{ev}_{\C} \circ \ \tau \circ \on{coev}_{\C}(R)=\on{Tr}(\on{id}_\C)(R)\in \on{RMod}_R\,.
\]
Note that $\on{Tr}(\on{id}_\C)$ is an $R$-linear functor and thus fully determined by its value $\HH(\C)$.  When $R$ is the sphere spectrum, $\HH(\C)$ is also called topological Hochschild homology. When $R=k$ is a commutative ring, $\HH(\C)$ describes the usual $k$-linear Hochschild homology. 

The dual Hochschild homology is defined as 
\[ \HH(\C)^\ast\coloneqq \on{Mor}_{\on{RMod}_R}(\HH(\C),R)\,.\]

The fixed points (limit over $BS^1$) of the $S^1$-action on $\HH(\C)$ are denoted by $\HH(\C)^{S^1}$. We will refer to $\HH(\C)^{S^1}$ as the negative cyclic homology of $\C$, as $\HH(\C)^{S^1}$ recovers the usual negative cyclic homology when $R=k$ is a commutative ring. The orbits (colimit over $BS^1$) of the $S^1$-action on $\HH(\C)$ are denoted $\HH(\C)_{S^1}$, and we similarly call $\HH(\C)_{S^1}$ the cyclic homology. The dual cyclic homology is given by $\HH(\C)_{S^1}^*=\on{Mor}_{\on{RMod}_R}(\HH(\C)_{S^1},R)$. There are natural maps $\HH(\C)^{S^1}\to \HH(\C)$ and $\HH(\C)_{S^1}^*\to \HH(\C)^*$.

\begin{notation}
Let $F\colon \C\to \D$ be a morphism in $\on{LinCat}_R^{\on{dual}}$. We denote by  
\begin{enumerate}[(1)]
\item $\HH(F)\colon \HH(\C)\to \HH(\D)$ the evaluation at $R$ of $\on{Tr}(F,\on{id}_F)$. We further define $\HH(\D,\C)=\cof\HH(F)$.
\item $\HH(F)^{S^1}\colon \HH(\C)^{S^1}\to \HH(\D)^{S^1}$ the induced map. We similarly define $\HH(\D,\C)^{S^1}=\on{cof}(\HH(F)^{S^1})$.
\end{enumerate}
Let $F\colon \D\to \C$ be a morphism in $\on{LinCat}_R^{\on{dual}}$. We denote by  
\begin{enumerate}
\item[(3)] $\HH(F)^\ast\colon \HH(\C)^\ast\to \HH(\D)^\ast$ the dual map obtained by precomposition with $\HH(F)$. We further define $\HH(\D,\C)^*=\on{cof}\HH(F)^*$.
\item[(4)] $\HH(F)_{S^1}^*\colon \HH(\C)_{S^1}^*\to \HH(\D)_{S^1}^*$ the induced map. We similarly define $\HH(\D,\C)^*_{S^1}=\on{cof}(\HH(F)_{S^1}^*)$.
\end{enumerate}
The notation of (1) and (2) will usually be used when discussing left Calabi--Yau structures on $F$, whereas the notation of (3) and (4) will usually be used when discussing right Calabi--Yau structures on $F$. This is why we swap the roles of $\C$ (the 'Calabi--Yau-boundary') and $\D$ (the relative Calabi--Yau category) as domain and target. 
\end{notation}

\begin{lemma}\label{lem:HHvsHom}
Let $\C$ be a dualizable $R$-linear $\infty$-category and let $E\colon \C\to \C$ be $R$-linear.
\begin{enumerate}[(1)]
\item If $\C$ is smooth, then $\on{Tr}(E)(R)$ is canonically equivalent to 
\[ \on{Mor}_{\on{Lin}_R(\C,\C)}(\on{id}_\C^!,E)\,.\] 
In particular, we have $\on{HH}(\C)\simeq \on{Mor}_{\on{Lin}_R(\C,\C)}(\on{id}_\C^!,\on{id}_\C)$.
\item If $\C$ is proper, then 
\[ \on{Tr}(E)(R)^*\coloneqq \on{Mor}_{\on{RMod}_R}(\on{Tr}(E)(R),R)\] 
is canonically equivalent to 
\[ \on{Mor}_{\on{Lin}_R(\C,\C)}(E,\on{id}_\C^*)\,.\] 
In particular, we have $\on{HH}(\C)^*\simeq \on{Mor}_{\on{Lin}_R(\C,\C)}(\on{id}_\C,\on{id}_\C^*)$.
\end{enumerate}
\end{lemma}

\begin{proof}
Suppose that $\C$ is smooth. Then we have an adjunction 
\[ \on{ladj}(\on{ev}_\C)\circ (\mhyphen)\colon \on{Lin}_R(\on{RMod}_R,\C^\vee\otimes \C) \longleftrightarrow  \on{Lin}_R(\on{RMod}_R,\on{RMod}_R) \noloc \on{ev}_\C\circ (\mhyphen)\,,\]
whose unit is given by precomposition with the unit of $\on{ladj}(\on{ev}_\C)\dashv \on{ev}_\C$. It follows that 
\begin{align*} \on{Mor}_{\on{Lin}_R(\C,\C)}(\on{id}_\C^!,E)\simeq & \on{Mor}_{\on{Lin}_R(\on{RMod}_R,\C\otimes \C^\vee)}(\tau\circ\on{ladj}(\on{ev}_\C),(E\otimes \on{id}_{\C^\vee})\circ \on{coev}_\C) \\
\simeq & \on{Mor}_{\on{Lin}_R(\on{RMod}_R,\on{RMod}_R)}(\on{id}_{\on{RMod}_R},\on{ev}_\C\circ\tau\circ (E\otimes \on{id}_{\C^\vee}) \circ \on{coev}_\C)\\
\simeq & \on{Tr}(E)(R)\,.
\end{align*}
If $\C$ is proper, a similar argument applies.
\end{proof}

\begin{remark}
Suppose the dualizable $R$-linear $\infty$-category $\C$ is smooth. If we make two different choices of left duals/adjoints 
\[ \on{id}_\C^!=\Xi_\C(\on{ladj}(\on{ev}_\C))\,,\quad(\on{id}_\C^!)'=\Xi_\C(\on{ladj}(\on{ev}_\C)')\] 
and two choices of units, there is a contractible space of equivalences $\alpha\colon (\on{id}_\C^!)'\simeq \on{id}_\C^!$, compatible with the unit, see \cite[Prop.~6.1.9]{Cis}. Any such equivalence $\alpha$ assembles with the equivalences from \Cref{lem:HHvsHom} into to a commutative diagram as follows:
\[
\begin{tikzcd}
{\on{Mor}_{\on{Lin}_R(\C,\C)}(\on{id}_\C^!,E)} \arrow[rd, "\simeq"'] \arrow[rr, "{\on{Mor}_{\on{Lin}_R(\C,\C)}(\alpha,E)}"] &             & {\on{Mor}_{\on{Lin}_R(\C,\C)}((\on{id}_\C^!)',E)} \arrow[ld, "\simeq"] \\
                                                                                                  & \on{Tr}(E)(R) &                                                                                
\end{tikzcd}
\] 
Stated differently, this means that the equivalence in part (1) of \Cref{lem:HHvsHom} is independent of the choice of left dual. A similar statement holds for the equivalence in part (2).
\end{remark}

\begin{construction}\label{constr:unitcounit}
Let $\C,\D$ be dualizable $R$-linear $\infty$-categories.

\noindent {\bf Case 1:} Suppose that $\C,\D$ are smooth. 

Let $F\colon \C\to \D$ be a morphism in $\on{LinCat}_R^{\on{dual}}$ and $G$ the $R$-linear right adjoint of $F$. We denote by
\[ F_!(\mhyphen)=F\circ (\mhyphen)\circ G\colon \on{Lin}_R(\C,\C)\longrightarrow \on{Lin}_R(\D,\D) \]
the functor from \Cref{lem:pullingbackendofunctors}, and by $\counit\colon F_!(\on{id}_\C)\to \on{id}_\D$ the counit transformation of $F\dashv G$. We define the \textit{unit} $\tilde{\unit}\colon \on{id}_\D^!\to F_!(\on{id}_\C^!)$ as the image under $\Xi_\D$ of the natural transformation
\begin{align*} 
\on{ladj}(\on{ev}_{\D})&\rightarrow \on{ladj}(\on{ev}_{\D})\circ \on{ev}_{\C}\circ \on{ladj}(\on{ev}_{\C})\\ 
&\rightarrow \on{ladj}(\on{ev}_{\D})\circ \on{ev}_{\D}\circ (F^\vee\otimes F)\circ \on{ladj}(\on{ev}_{\C})\\
&\rightarrow (F^\vee\otimes F)\circ \on{ladj}(\on{ev}_{\C})
\end{align*}
composed with the equivalence 
\[ \Xi_\D((F^\vee\otimes F)\circ \on{ladj}(\on{ev}_{\C}))\simeq F_!(\on{id}_\C^!)\]
from \Cref{lem:pullingbackendofunctors}. The transformation $\tilde{\unit}$ is indeed a unit if $F$ admits a left adjoint, see \Cref{lem:unitisunit}.

\noindent {\bf Case 2:} Suppose that $\C,\D$ are proper.

Let $F\colon \D\to \C$ be a morphism in $\on{LinCat}_R^{\on{dual}}$ and $G$ the $R$-linear right adjoint of $F$. Consider the functor 
\[ F^*(\mhyphen)=G\circ (\mhyphen)\circ F\colon \on{Lin}_R(\C,\C)\longrightarrow \on{Lin}_R(\D,\D)\]
from \Cref{lem:pullingbackendofunctors} and denote by $\unit\colon \on{id}_\D\to F^*(\on{id}_\C)$ the unit of $F\dashv G$. Let $E^\vee$ denote the right adjoint of $F^\vee$. Applying $\Xi_{\C^\vee}^{-1}$ to the counit $F^\vee E^\vee\to \on{id}_\C$ defines a natural transformation 
\[ (F^\vee\otimes F)\circ \on{coev}_{\D^\vee}\longrightarrow \on{coev}_{\C^\vee}\,.\]  
We use this to define the \textit{counit} $\tilde{\counit}\colon F^*(\on{id}_\C^*)\to \on{id}_\D^*$ as the image under $\Theta_\D$ of the natural transformation
\begin{align*}
\on{ladj}(\on{coev}_{\C^\vee})\circ (F^\vee\otimes F) & \to \on{ladj}(\on{coev}_{\C^\vee})\circ (F^\vee\otimes F)\circ \on{coev}_{\D^\vee} \circ \on{ladj}(\on{coev}_{\D^\vee})\\
& \to \on{ladj}(\on{coev}_{\C^\vee})\circ \on{coev}_{\C^\vee} \circ \on{ladj}(\on{coev}_{\D^\vee})\\
& \to \on{ladj}(\on{coev}_{\D^\vee})
\end{align*}
composed with the identification $\Theta_{\D}(\on{ladj}(\on{coev}_{\C^\vee})\circ (F^\vee\otimes F))\simeq F^*(\on{id}_\C^*)$ from \Cref{lem:pullingbackendofunctors}.  
\end{construction}

The following proposition describes the Hochschild homology functor in terms of the corresponding morphisms between bimodules.

\begin{proposition}[$\!\!${\cite[Prop.~4.4]{BD21}}]\label{prop:smoothHHfun}
Let $F\colon\C\to \D$ be a map in $\on{LinCat}_R^{\on{dual}}$ with $\C$ smooth. Consider a morphism $\xi\colon R[i]\to \on{Tr}(\on{id}_\C)$, corresponding via \Cref{lem:HHvsHom} to a natural transformation $\alpha\colon \on{id}_\C^![i]\to \on{id}_\C$. 
\begin{enumerate}[(1)]
\item The morphism
\begin{equation}\label{eq:Tr} R[i]\xlongrightarrow{\xi} \on{Tr}(\on{id}_\C)\xlongrightarrow{\on{Tr}(F,\on{id}_F)}\on{Tr}(\on{id}_\D)\end{equation}
is equivalent to the composite 
\begin{align*} 
R[i]\longrightarrow \on{Tr}(\on{id}_\C^!)[i]\xrightarrow{\on{Tr}(\on{id}_\C,\on{id}_\C^!\unit)[i]} &\on{Tr}(\on{id}_\C^!GF)[i]
 \longrightarrow \on{Tr}(F\on{id}_\C^!G)[i]\\ \xrightarrow{\on{Tr}(\on{id}_\D,F\alpha G)}& \on{Tr}(FG)\xrightarrow{\on{Tr}(\on{id}_\D,\counit)}\on{Tr}(\on{id}_\D)\,.
\end{align*}
\item Suppose that $\D$ is smooth. Then the morphism \eqref{eq:Tr} can be identified with
\[ \counit\circ F_!(\alpha)\circ \tilde{\unit}\in \on{Ext}^{-i}_{\on{Lin}_R(\D,\D)}(\on{id}_\D^!,\on{id}_\D)\,.\] 
\end{enumerate}
\end{proposition}

\begin{proof}
The proof of part (1) can be found in \cite[Prop.~4.4]{BD21}, we also spell out its dual version in the proof of \Cref{prop:properHHfun} below. Part (2) is stated in loc.~cit., we spell the proof out for convenience of the reader. Using part (1), it suffices to show that the composite
\[R[i]\longrightarrow \on{Tr}(\on{id}_\C^!)[i]\xrightarrow{\on{Tr}(\on{id}_\C,\on{id}_\C^!\unit)[i]} \on{Tr}(\on{id}_\C^!GF)[i]
 \longrightarrow \on{Tr}(F\on{id}_\C^!G)[i]\]
maps $1\in R[i]$ to $\tilde{\unit}$ after evaluating the trace at $R$ and using the identification of \Cref{lem:HHvsHom}. 

The image under $\Xi_\C^{-1}$ of the unit $\tilde{\unit}\colon \on{id}_\D^!\to F_!(\on{id}_\D^!)$ corresponds via the adjunction 
\[ \on{ladj}(\on{ev}_\D)\circ (\mhyphen)\colon  \on{Lin}_R(\on{RMod}_R,\on{RMod}_R) \longleftrightarrow\on{Lin}_R(\on{RMod}_R,\D^\vee \otimes \D)\noloc \on{ev}_\D\circ (\mhyphen)\] 
to the morphism
\begin{align*}
\on{id}_{\on{RMod}_R} & \to \on{ev}_\D\circ \on{ladj}(\on{ev}_\D)\circ \on{ev}_\C\circ \on{ladj}(\on{ev}_\C)\\
& \to \on{ev}_\D\circ \on{ladj}(\on{ev}_\D)\circ \on{ev}_\D\circ (F^\vee\otimes F)\circ \on{ladj}(\on{ev}_\C)\\
& \to \on{ev}_\D\circ (F^\vee\otimes F)\circ \on{ladj}(\on{ev}_\C)\,,
\end{align*}
which is by the triangle identity the same as 
\[ \on{id}_{\on{RMod}_R} \to \on{ev}_\C\circ \on{ladj}(\on{ev}_\C) \to \on{ev}_\D\circ (F^\vee\otimes F)\circ \on{ladj}(\on{ev}_\C)\,.\]

The map coming from the trace is given by
\begin{align*}
\on{id}_{\on{RMod}_R} & \to \on{ev}_\C\circ \on{ladj}(\on{ev}_\C)\\
& \to \on{ev}_\C\circ ( \on{id}_{\C^\vee}\otimes GF)\circ \on{ladj}(\on{ev}_\C)\\
& \to \on{ev}_\D\circ (F^\vee\otimes FGF)\circ \on{ladj}(\on{ev}_\C)\\
& \to \on{ev}_\D\circ (F^\vee \otimes F)\circ \on{ladj}(\on{ev}_\C)\,.
\end{align*}
Applying the triangle identity of the adjunction $F\dashv G$, we find that this agrees with the above.
\end{proof}

Given $(\C,E)\in {\bf End}({\bf LinCat}_R)$, we denote by $\on{Tr}(\C,E)^*$ the composite of $\on{Tr}(\C,E)$ and $\on{Mor}_{\on{RMod}_R}(\mhyphen,R)\colon \on{RMod}_R\to\on{RMod}_R$. The analogue of \Cref{prop:smoothHHfun} for the dual Hochschild homology is as follows.

\begin{proposition}\label{prop:properHHfun}
Let $F\colon \D \to \C$ be a map in $\on{LinCat}_R^{\on{dual}}$ with $\D$ proper. Consider a morphism $\xi\colon R[i]\to \on{Tr}(\on{id}_\C)^*$, corresponding via \Cref{lem:HHvsHom} to a natural transformation $\alpha\colon \on{id}_\C\to \on{id}_\C^*[-i]$. 
\begin{enumerate}[(1)]
\item The morphism
\begin{equation}\label{eq:dualTr} 
 R[i] \xlongrightarrow{\xi} \on{Tr}(\on{id}_\C)^*\xlongrightarrow{\on{Tr}(F,\on{id}_F)^*}\on{Tr}(\on{id}_\D)^*\end{equation}
is equivalent to the composite 
\begin{align*} 
R[i]\longrightarrow \on{Tr}(\on{id}_\C^*)^*[i]\xrightarrow{\on{Tr}(\on{id}_\C,\on{id}_\C^*\counit)^*[i]} &\on{Tr}(\on{id}_\C^*FG)^*[i]
 \longrightarrow \on{Tr}(G\on{id}_\C^*F)^*[i]\\ \xrightarrow{\on{Tr}(\on{id}_\D,G\alpha F)^*}& \on{Tr}(GF)^*\xrightarrow{\on{Tr}(\on{id}_\D,\unit)^*}\on{Tr}(\on{id}_\D)^*\,.
\end{align*}
\item Suppose that $\C$ is proper. Then the morphism \eqref{eq:dualTr} can be identified with
\[ \tilde{\counit}\circ F^*(\alpha)\circ \unit\in \on{Ext}^{-i}_{\on{Lin}_R(\D,\D)}(\on{id}_\D,\on{id}_\D^*)\,.\]
\end{enumerate}
\end{proposition}

\begin{proof}
The proof of part (1) is dual to the proof of part (1) of \Cref{prop:smoothHHfun}. Using \Cref{lem:HHvsHom}, the identity on $\on{id}_\C^*$ gives rise to a morphism $ \mu \colon R[i]\to \on{Tr}(\on{id}_\C^*)[i]$. The statement now follows from the following commutative diagram.
\[
\begin{tikzcd}[column sep=46]
{R[i]} \arrow[r, "\mu"] \arrow[rr, "\xi", bend left=15] & {\on{Tr}(\on{id}_\C^*)^*[i]} \arrow[r, "{\on{Tr}(\on{id}_\C,\alpha)^*}"'] \arrow[d, "{\on{Tr}(\on{id}_\C,\on{id}_\C^*\counit)^*[i]}"'] & \on{Tr}(\on{id}_\C)^* \arrow[d, "{\on{Tr}(\on{id}_\C,\counit)^*}"'] \arrow[rdd, "{\on{Tr}(F,\on{id}_F)^*}"] &                       \\
                                              & {\on{Tr}(\on{id}_\C^*FG)^*[i]} \arrow[d] \arrow[r, "{\on{Tr}(\on{id}_\C,\alpha FG)^*}"]                                                 & \on{Tr}(FG)^* \arrow[d]                                                                                     &                       \\
                                              & {\on{Tr}(G\on{id}_\C^*F)^*[i]} \arrow[r, "{\on{Tr}(\on{id}_\D,G\alpha F)^*}"]                                                           & \on{Tr}(GF)^* \arrow[r, "{\on{Tr}(\on{id}_\D,\unit)^*}"]                                                    & \on{Tr}(\on{id}_\D)^*
\end{tikzcd}
\]
Note that the commutativity of the rightmost triangle follows from dualizing \Cref{eq:trace}.

Part (2) can be shown as in the proof of \Cref{prop:smoothHHfun}.
\end{proof}

\begin{lemma}\label{lem:unitisunit}
Let $\C,\D\in \on{LinCat}_R^{\on{dual}}$.
\begin{enumerate}[(1)]
\item Suppose that $\C,\D$ are smooth. Let $F\colon \C\to \D$ be a map in $\on{LinCat}_R^{\on{dual}}$ which admits a left adjoint $E$. Then there exists a natural equivalence $F_!(\on{id}_\C^!)\simeq FE\on{id}_{\D}^!$, such that the composite of 
\[ \tilde{\unit}\colon \on{id}_{\D}^!\to F_!(\on{id}_\C^!)\] 
with this equivalence describes a unit of $E\dashv F$ composed with $\on{id}_\D^!$.
\item Suppose that $\C,\D$ are proper. Let $F\colon \D\to \C$ be a map in $\on{LinCat}_R^{\on{dual}}$ which admits a left adjoint $E$. Then there exists a natural equivalence $F^*(\on{id}_\C^*)\simeq \on{id}_{\D}^*EF$, such that the composite of 
\[ \tilde{\counit}\colon F^*(\on{id}_\C^*)\to \on{id}_\D^*\] 
with this equivalence describes a counit of $E\dashv F$ composed with $\on{id}_\D^*$.
\end{enumerate}
\end{lemma}

\begin{proof}
We only prove part (1), part (2) is similar. We have
\[\on{ev}_\C\circ (E^\vee\otimes \on{id}_\C)\simeq \on{ev}_\D\circ (\on{id}_{\D^\vee}\otimes F)\,.\]
The units of the adjunctions  
\[(F^\vee\otimes \on{id}_\C)\circ \on{ladj}(\on{ev}_\C)\dashv \on{ev}_\C\circ ( E^\vee\otimes\on{id}_\C)\] 
and 
\[(\on{id}_{\D^\vee}\otimes E)\circ \on{ladj}(\on{ev}_\D) \dashv \on{ev}_\D\circ (\on{id}_{\D^\vee}\otimes F)\]
are therefore equivalent. This gives rise to the following commutative diagram:
\[
\begin{tikzcd}[column sep=small]
\on{id}_{\on{RMod}_R} \arrow[d] \arrow[r] & \on{ev}_\C\circ \on{ladj}(\on{ev}_\C) \arrow[r]                                   & \on{ev}_\C \circ (E^\vee F^\vee\otimes\on{id}_\C)\circ\on{ladj}(\on{ev}_\C) \arrow[d, "\simeq"] \\
 \on{ev}_\D\circ \on{ladj}(\on{ev}_\D) \arrow[r]           &  \on{ev}_\D\circ (\on{id}_{\D^\vee}\otimes FE)\circ \on{ladj}(\on{ev}_\D) \arrow[r, "\simeq"] & \on{ev}_\D \circ(F^\vee\otimes F)\circ\on{ladj}(\on{ev}_\C)                                        
\end{tikzcd}
\]
The naturality of the unit $\on{id}_{\D^\vee}\otimes \on{id}_{\D}\to \on{id}_{\D^\vee}\otimes FE $ and the counit $\on{id}_{\D^\vee}\otimes\on{ladj}(\on{ev}_\D)\on{ev}_\D\to \on{id}_{\D^\vee}\otimes \on{id}_{\D}$ gives rise to the following commutative diagram:
\[
\begin{tikzcd}
\on{ladj}(\on{ev}_\D) \on{ev}_\D \on{ladj}(\on{ev}_\D) \arrow[d] \arrow[r]               & \on{ladj}(\on{ev}_\D) \on{ev}_\D (\on{id}_{\D^\vee}\otimes FE)\on{ladj}(\on{ev}_\D)\arrow[d]               \\
 \on{ladj}(\on{ev}_\D) \arrow[r] & (\on{id}_{\D^\vee}\otimes FE)\on{ladj}(\on{ev}_\D)
\end{tikzcd}
\]
Postcomposing the upper diagram with $\on{ladj}(\on{ev}_\D)$ and combining it with the lower diagram, we see that the definition of the natural transformation $\tilde{\unit}\colon \on{id}_\D^!\to F_!(\on{id}_\C^!)$ in \Cref{constr:unitcounit}
is equivalent to the image under $\Xi_\D$ of the natural transformation
\begin{align*}
\on{ladj}(\on{ev}_{\D})&\to \on{ladj}(\on{ev}_{\D})\on{ev}_{\D}\on{ladj}(\on{ev}_{\D})\\
&\to \on{ladj}(\on{ev}_{\D})\\
&\to  (\on{id}_{\D^\vee}\otimes FE)\on{ladj}(\on{ev}_{\D})\\
&\simeq (F^\vee\otimes F)\on{ladj}(\on{ev}_{\C})\,.
\end{align*}
The desired description of $\tilde{\unit}$ now follows via the triangle identity for the adjunction $\on{ladj}(\on{ev}_{\D})\dashv \on{ev}_{\D}$.
\end{proof}

The following lemma shows that the natural transformations $\tilde{\unit}$ and $\tilde{\counit}$ can also be seen as the adjoints of the counit and unit of $F\dashv G$.

\begin{lemma}\label{lem:dualcounit}~
\begin{enumerate}[(1)]
\item Let $F\colon \C\to\D$ be a morphism in $\on{LinCat}_R^{\on{dual}}$ with $\C,\D$ smooth. Denote the right adjoint of $F$ by $G$. The image under $\Xi_\C^{-1}$ of the unit $\tilde{\unit}\colon \on{id}_\D^!\to F\on{id}_\C^! G$ is left adjoint to the counit $\on{ev}_\D\circ (\on{id}_{\D^\vee} \otimes FG)\to \on{ev}_\D$. 
\item Let $F\colon \D\to\C$ be a morphism in $\on{LinCat}_R^{\on{dual}}$ with $\C,\D$ proper. Denote the right adjoint of $F$ by $G$. The image under $\Theta_\C^{-1}$ of the counit $\tilde{\counit}\colon G\on{id}_\C^* F\to \on{id}_\C^*$ is left adjoint to the unit $\on{ev}_\D\to \on{ev}_\D\circ (\on{id}_{\D^\vee}\otimes GF)$.
\end{enumerate}
\end{lemma}

\begin{proof}
We only prove part (1), part (2) can be proven similarly. Denote the right adjoint of $F^\vee$ by $E^\vee$. Upon passing to right adjoints, the natural transformation $\Xi_\C^{-1}(\tilde{\unit})\colon \on{ladj}(\on{ev}_\D)\to (F^\vee \otimes F)\circ \on{ladj}(\on{ev}_\C)$ induces the following natural transformation:
\begin{align*}
\on{ev}_{\C} \circ (E^\vee\otimes G) & \to  \on{ev}_{\C} \circ(E^\vee\otimes G)\circ\on{radj}(\on{ev}_{\D})\circ \on{ev}_{\D}\\
&\to \on{ev}_{\C}\circ \on{radj}(\on{ev}_{\C}) \circ \on{ev}_{\D}\\
&\to \on{ev}_{\D}
\end{align*}
Unraveling the definition of the above natural transformation, one sees that it is equivalent to the composite\footnote{Hint: given a natural transformation $F\to F'$, the induced natural transformation on the right adjoints $G'\to G$ is obtained as the composite $G'\to GFG'\to GF'G'\to G$.}
\begin{align*}
\on{ev}_{\C} \circ (E^\vee\otimes G) & \to  \on{ev}_{\C} \circ(E^\vee\otimes G)\circ\on{radj}(\on{ev}_{\D})\circ \on{ev}_{\D}\\
&\to \on{ev}_{\D} \circ(F^\vee E^\vee\otimes FG)\circ\on{radj}(\on{ev}_{\D})\circ \on{ev}_{\D}\\
&\to  \on{ev}_{\D}\circ \on{radj}(\on{ev}_{\D}) \circ \on{ev}_{\D}\\
&\to \on{ev}_{\D}\,,
\end{align*}
where the third natural transformation uses the counits of $F\dashv G$ and $F^\vee\dashv E^\vee$. This natural transformation fits into the following commutative diagram,
\[
\begin{tikzcd}[column sep=small]
\on{ev}_\D\circ (\on{id}_{\D^\vee}\otimes FG) \arrow[d] \arrow[r, "\simeq"] & \on{ev}_{\C}\circ (E^\vee\otimes G) \arrow[r] \arrow[d]       & \on{ev}_{\C}\circ (E^\vee\otimes G)\circ \on{radj}(\on{ev}_\D)\circ \on{ev}_\D \arrow[d]         \\
\on{ev}_\D\circ (\on{id}_{\D^\vee}\otimes FGFG) \arrow[r, "\simeq"]         & \on{ev}_\D\circ (F^\vee E^\vee\otimes FG) \arrow[r] \arrow[d] & \on{ev}_{\D}\circ (F^\vee E^\vee\otimes FG)\circ \on{radj}(\on{ev}_\D)\circ \on{ev}_\D \arrow[d] \\
                                                                            & \on{ev}_\D \arrow[r, leftrightarrow]                                          & \on{ev}_{\D}\circ \on{radj}(\on{ev}_\D)\circ \on{ev}_\D                               
\end{tikzcd}
\]
exhibiting it as the counit $\on{ev}_\D\circ (\on{id}_{\D^\vee}\otimes FG)\to \on{ev}_\D$ by further triangle identities.
\end{proof}

\section{Relative Calabi--Yau structures}\label{sec:relCYstr}

The goal of this section is to introduce $R$-linear relative Calabi--Yau structures and describe essential features of their theory. We begin in \Cref{subsec:leftCYstructures,subsec:rightCYstructures} with their definitions. After the short \Cref{subsec:tensorCY} on the behavior of Calabi--Yau structures under tensor products, we generalize the gluing properties of relative Calabi--Yau structures of \cite{BD19} to the $R$-linear setting in \Cref{subsec:CYglue}. For the entire section, we fix an $\mathbb{E}_\infty$-ring spectrum $R$. 

\subsection{Left Calabi--Yau structures}\label{subsec:leftCYstructures}

Let $F\colon\C\rightarrow \D$ be a dualizable $R$-linear functor between dualizable and smooth $R$-linear $\infty$-categories. Part (2) of \Cref{prop:smoothHHfun} shows that an $R$-linear relative Hochschild homology class $\sigma\colon R[n]\to \HH(\D,\C)$ amounts to a diagram in $\on{Lin}_R(\D,\D)$
\[
\begin{tikzcd}
{\on{id}_{\D}^!} \arrow[r, "\tilde{\unit}"] & {F_!(\on{id}_{\C}^!)} \arrow[d]               &                       \\
                                         & {F_!(\on{id}_{\C})[1-n]}\arrow[r, "{\counit[1-n]}"] & {\on{id}_{\D}[1-n]}
\end{tikzcd}
\]
together with a choice of null-homotopy of the composite $\on{id}_\D^!\to \on{id}_\D[1-n]$. The null-homotopy allows us to extend the diagram to a diagram with horizontal fiber and cofiber sequences as follows:
\begin{equation}\label{eq:leftCYfibcof}
\begin{tikzcd}
{\on{id}_{\D}^!} \arrow[r, "\tilde{\unit}"] \arrow[d] & {F_!(\on{id}_{\C}^!)} \arrow[d] \arrow[r] & \cof \arrow[d] \\
\fib \arrow[r]                       & {F_!(\on{id}_{\C})[1-n]} \arrow[r, "{\counit[1-n]}"]                    & {\on{id}_{\D}[1-n]}       
\end{tikzcd}
\end{equation}
We call the Hochschild homology class $\sigma$ {\em non-degenerate} if all the vertical maps in the diagram \eqref{eq:leftCYfibcof} are equivalences. 

\begin{definition}
\begin{enumerate}[(1)]
\item A weak left $n$-Calabi--Yau structure on the functor $F$ consists of a non-degenerate Hochschild homology class $\sigma\colon R[n]\to \HH(\D,\C)$.  If $F=0$, we also say that $\mathcal{D}$ carries a weak left $n$-Calabi--Yau structure. 
\item A left $n$-Calabi--Yau structure on the functor $F$ consists of a negative cyclic homology class $\eta\colon R[n]\to \HH^{S^1}(\D,\C)$, whose composition with $\HH^{S^1}(\D,\C)\to \HH(\D,\C)$ defines a non-degenerate Hochschild homology class. If $F=0$, we also say that $\mathcal{D}$ carries a left $n$-Calabi--Yau structure. 
\end{enumerate}
\end{definition}

Weak left $n$-Calabi--Yau structures are also sometimes called bimodule $n$-Calabi--Yau structures.

\begin{remark}\label{rem:independenceCY}
The notion of weak left Calabi--Yau structure on a functor $F$ only depends on the functor and the relative Hochschild class and not on any further choices made in its definition. This includes choices of adjoints and (co)units. For example, we make a choice of right adjoint of $F$ together with the counit, the space of such choices is however contractible. Inspecting the definition one finds that making a different choice yields an equivalent diagram in \eqref{eq:leftCYfibcof} and thus the same condition of the Hochschild class being non-degenerate.
\end{remark}

\subsection{Right Calabi--Yau structures}\label{subsec:rightCYstructures}

Let $F\colon\D\rightarrow \C$ be a dualizable $R$-linear functor between dualizable and proper $R$-linear $\infty$-categories. Part (2) of \Cref{prop:properHHfun} shows that an $R$-linear dual relative Hochschild homology class $\sigma\colon R[n]\to \HH(\D,\C)^*$ amounts to a diagram in $\on{Lin}_R(\D,\D)$
\[
\begin{tikzcd}
{\on{id}_{\mathcal{D}}} \arrow[r, "{\unit}"] & {F^*(\on{id}_\C)} \arrow[d]       &                         \\
                                                       & F^*(\on{id}_\C^*)[1-n] \arrow[r, "{\tilde{\counit}[1-n]}"] & \on{id}_{\mathcal{D}}^*[1-n]
\end{tikzcd}
\]
together with a choice of null-homotopy of the composite $\on{id}_\D\to \on{id}_\D^*[1-n]$. We extend the diagram to a diagram with horizontal fiber and cofiber sequences as follows:
\begin{equation}\label{eq:defrightCY}
\begin{tikzcd}
{\on{id}_{\mathcal{D}}} \arrow[r, "{\unit}"] \arrow[d] & {F^*(\on{id}_\C)} \arrow[d] \arrow[r] & \on{cof} \arrow[d]      \\
\on{fib} \arrow[r]                                               & F^*(\on{id}_\C^*)[1-n] \arrow[r, "{\tilde{\counit}[1-n]}"]     & \on{id}_{\mathcal{D}}^*[1-n]
\end{tikzcd}
\end{equation}
As in the smooth case, we call the dual Hochschild homology class $\sigma$ {\em non-degenerate} if all the vertical maps in the above diagram are equivalences. 

\begin{definition}~
\begin{enumerate}[(1)]
\item A weak right $n$-Calabi--Yau structure on the functor $F$ consists of a non-degenerate dual Hochschild homology class $\sigma\colon R[n]\to \HH(\D,\C)^*$. If $\C=0$, we also say that $\D$ carries a weak right $n$-Calabi--Yau structure. 
\item A right $n$-Calabi--Yau structure on the functor $F$ consists of a dual cyclic homology class $\eta\colon R[n]\to \HH(\D,\C)^*_{S^1}$, whose composition with $\HH(\D,\C)^*_{S^1}\to \HH(\D,\C)^*$ defines a non-degenerate dual Hochschild homology class. If $\C=0$, we also say that $\D$ carries a right $n$-Calabi--Yau structure. 
\end{enumerate}
\end{definition}

\begin{remark}
Assume that $\D$ is compactly generated. A weak right $n$-Calabi--Yau structure on $\D$ equivalently consists of an equivalence in $\on{RMod}_R$ 
\begin{equation}\label{eq:CYduality} \on{Mor}_{\mathcal{D}}(X,Y)\simeq \on{Mor}_{\mathcal{D}}(Y,X[n])^*\,,\end{equation}
bifunctorial in $X,Y\in \mathcal{D}^{\on{c}}$. 

In good situations, a relative right $n$-Calabi--Yau structure on $\D$ gives rise to a duality of a subfunctor of $\on{Mor}_\D(\mhyphen,\mhyphen)$, generalizing the equivalence \eqref{eq:CYduality}, see \cite{Chr22b}.
\end{remark}

\begin{remark}
It is also possible to make sense of relative right Calabi--Yau structures on some non-proper $k$-linear $\infty$-categories, namely those arising as the derived $\infty$-category of a dg algebra whose homology is finite dimensional in each degree. We refer to the recent work \cite{KL23} for this notion and the construction of many examples related to cluster categories.
\end{remark}

Finally, we comment on the relation with dg categorical left and right Calabi--Yau structures.

Recall that the passage to derived $\infty$-categories defines a functor $\D(\mhyphen)\colon N(\on{dgCat}_k)\to \on{LinCat}_k$, with $\on{dgCat}_k$ the $1$-category of $k$-linear dg categories.

\begin{lemma}\label{lem:dg_CY_to_infty_CY}
Let $f\colon A\to B$ be a dg functor. Then $F\coloneqq \D(f)\colon \D(A)\to \D(B)$ is a colimit preserving functor between compactly generated $k$-linear $\infty$-categories.
\begin{enumerate}[(1)]
\item Suppose that $A$ and $B$ are smooth. Then $\D(A),\D(B)$ are smooth as well. Further, any weak left $n$-Calabi--Yau structure of $f$ induces a weak left $n$-Calabi--Yau structure of $F$ and vice versa. 
\item Suppose that $A$ and $B$ are proper. Then $\D(A),\D(B)$ are proper as well. Further, any weak right $n$-Calabi--Yau structure of $f$ induces a weak right $n$-Calabi--Yau structure of $F$ and vice versa. 
\end{enumerate}
\end{lemma}

\begin{proof}
We may assume that $A$ and $B$ are cofibrant dg categories (with respect to the quasi-equivalence model structure). We first show that $\D(A\otimes A^{\on{op}})\simeq \D(A)\otimes \D(A^{\on{op}})$, where the former tensor product is of dg categories and the latter in $\on{LinCat}_k$. A similar argument shows that $\D(B\otimes B^{\on{op}})\simeq \D(B)\otimes \D(B^{\on{op}})$. We can compute $A\otimes A^{\on{op}}$ as the colimit of the $2$-sided bar construction, which is a diagram $\Delta^{\on{op}}\to \on{dgCat}_k$, mapping $[n]$ to $A\times \on{dgMod}_k^{\times n} \times B$. This diagram is cofibrant with respect to the Reedy model structure on $\on{Fun}(\Delta^{\on{op}},\on{dgCat}_k)$, meaning that all latching morphisms are cofibrations \cite[Def.~15.3.3]{Hir03}. Thus, the colimit of the $2$-sided bar construction computes its homotopy colimit, see also \cite[Thm.~15.10.8]{Hir03}. The image under $\D(\mhyphen)$ of the the homotopy colimit can further be identified with the colimit of the $\infty$-categorical $2$-sided bar construction \cite[Constr.~4.4.2.7]{HA}, computing $\D(A)\otimes \D(A^{\on{op}})\in \on{LinCat}_k$. 

To the authors knowledge, it has not been shown that that there is a symmetric monoidal functor of $(\infty,2)$-categories lifting $\D(\mhyphen)$. However, the above shows that the $\infty$-functor $\D(\mhyphen)\colon N(\on{dgCat}_k)\to \on{LinCat}_k$ preserves the tensor product of two cofibrant dg categories, and thus maps the evaluation and coevaluation bimodules to the evaluation and coevaluation functors. The functor thus (non-coherently) identifies the two traces, and hence (non-coherently) the functoriality of Hochschild homology. To identify the $S^1$-actions on the traces, a coherent identifications of the traces would be necessary. 

To obtain the statement from the Lemma, all that remains to note is that the non-degeneracy conditions on relative (possibly dual) Hochschild classes of $f$ and $F$ coincide, which follows from \Cref{prop:smoothHHfun} and \cite[Prop.~4.4]{BD21} in the smooth case and a similar argument in the proper case.
\end{proof}

\subsection{Behavior under tensor products}\label{subsec:tensorCY}

\begin{lemma}\label{lem:tensorHH}
Let $\mathcal{C},\mathcal{D}$ be dualizable $R$-linear $\infty$-categories. There is a canonical equivalence
\[ \on{HH}(\mathcal{C})\otimes \on{HH}(\mathcal{D})\simeq \on{HH}(\mathcal{C}\otimes \mathcal{D})\,.\]
\end{lemma}

\begin{proof}
Using that $\on{ev}_{\C\otimes \D}\simeq \on{ev}_{\C}\otimes \on{ev}_{\D}$ and $\on{coev}_{\C\otimes \D}\simeq \on{coev}_{\C}\otimes \on{coev}_{\D}$, we find 
\begin{align*}
 \HH(\C\otimes \D)& =\on{ev}_{\C\otimes \D}\circ \tau \circ \on{coev}_{\C\otimes \D}(R)\\
 &\simeq (\on{ev}_{\C}\otimes \on{ev}_\D)\circ \tau \circ (\on{coev}_{\C}\otimes \on{coev}_\D)(R)\\
 &\simeq (\on{ev}_{\C}\circ \tau \circ \on{coev}_{\C})(R)\otimes (\on{ev}_\D \circ \tau \circ \on{coev}_\D)(R)\\
 & = \HH(\C)\otimes \HH(D)\,.
\end{align*}
\end{proof}

We note that the equivalence in \Cref{lem:tensorHH} is $S^1$-equivariant.

\begin{remark}
If $\C,\D$ are smooth, then $\C\otimes \D$ is also smooth and we have $\on{id}_{\C\otimes \D}^!\simeq \on{id}_{\C}^!\otimes \on{id}_\D^!$. Similarly, if $\C,\D$ are proper then $\C\otimes \D$ is proper and $\on{id}_{\C\otimes \D}^*\simeq \on{id}_{\C}^*\otimes \on{id}_{\D}^*$. 

In the smooth case, a pair of morphisms  $\alpha\colon \on{id}_\C^!\to \on{id}_\C[-n]$, $\beta\colon \on{id}_\D^!\to \on{id}_\D[-m]$ gives under the identifications from \Cref{lem:HHvsHom,lem:tensorHH} rise to 
the morphism
\[ \on{id}_{\C\otimes \D}^!\simeq \on{id}_\C^!\otimes \on{id}_\D^!\xrightarrow{\alpha\otimes \beta} \on{id}_{\C}[-n]\otimes \on{id}_\D[-m]\simeq \on{id}_{\C\otimes \D}[-n-m]\,.\]
A similar assertion holds in the proper case.
\end{remark}

The following proposition shows that the tensor product of a Calabi--Yau functor with a Calabi--Yau category is again Calabi--Yau. A similar statement is proven in \cite[Prop.~6.4]{BD21}.

\begin{proposition}\label{lem:cytp}
Let $\mathcal{C},\mathcal{D},\mathcal{E}$ be dualizable $R$-linear $\infty$-categories. 
\begin{enumerate}[(1)]
\item Assume that $\C,\D,\E$ are smooth. Let $F\colon \mathcal{C}\rightarrow \mathcal{D}$ be a dualizable $R$-linear functor and let $\eta\colon R[n]\to \on{HH}(\mathcal{D},\mathcal{C})^{S^1}$ and $\eta'\colon R[m]\to \on{HH}(\mathcal{E})^{S^1}$ be left Calabi--Yau-structures on $F$ and $\mathcal{E}$, respectively. The class 
\[
R[n+m]\simeq R[n]\otimes R[m]\xlongrightarrow{\eta\otimes\eta'}\HH(\D,\C)^{S^1}\otimes \HH(\mathcal{E})^{S^1}\simeq \HH(\D\otimes \mathcal{E},\C\otimes \mathcal{E})^{S^1} 
\]
defines a left $(n+m)$-Calabi--Yau structure on 
\[F\otimes \mathcal{E}\colon \C\otimes \mathcal{E}\rightarrow \D\otimes \mathcal{E}\,.\]
\item Assume that $\C,\D,\E$ are proper. Let $F\colon \D\rightarrow \C$ be a dualizable $R$-linear functor and let $\eta\colon R[n]\to \on{HH}(\mathcal{D},\mathcal{C})^*_{S^1}$ and $\eta'\colon R[m]\to \on{HH}(\mathcal{E})^*_{S^1}$ be right Calabi--Yau-structures on $F$ and $\mathcal{E}$, respectively. The class 
\[
R[n+m]\simeq R[n]\otimes R[m]\xlongrightarrow{\eta\otimes\eta'}\HH(\D,\C)_{S^1}^*\otimes \HH(\mathcal{E})^*_{S^1}\simeq \HH(\D\otimes \mathcal{E},\C\otimes \mathcal{E})^*_{S^1}
\]
defines a right $(n+m)$-Calabi--Yau structure on 
\[F\otimes \mathcal{E}\colon \mathcal{D}\otimes \mathcal{E}\rightarrow \mathcal{C}\otimes \mathcal{E}\,.\]
\end{enumerate}
\end{proposition}

\begin{proof}
We only prove part (1), part (2) is analogous.  
The Hochschild homology class $R[n+m]\to \on{HH}(\mathcal{D}\otimes \mathcal{E},\mathcal{C}\otimes \mathcal{E})^{S^1}\to  \on{HH}(\mathcal{D}\otimes \mathcal{E},\mathcal{C}\otimes \mathcal{E})$ gives rise to the following diagram in $\on{Lin}_R(\mathcal{D}\otimes \mathcal{E},\mathcal{C}\otimes \mathcal{E})$, up to equivalence.
\[
\begin{tikzcd}
\on{id}_{\mathcal{D}}^!\otimes \on{id}_{\mathcal{E}}^! \arrow[d] \arrow[r] & F_!(\on{id}_\C^!)\otimes \on{id}_{\mathcal{E}}^! \arrow[d] \arrow[r] & \on{cof}\otimes \on{id}_{\mathcal{E}}^! \arrow[d]             \\
{\on{fib}\otimes \on{id}_{\mathcal{E}}[-m]} \arrow[r]                       & {F_!(\on{id}_\C)[1-n]\otimes \on{id}_{\mathcal{E}}[-m]} \arrow[r]       & {\on{id}_{\mathcal{D}}[1-n]\otimes \on{id}_{\mathcal{E}}[-m]}
\end{tikzcd}
\]
The horizontal sequences in the above diagram are fiber and cofiber sequences as tensor products of such with $\on{id}_{\mathcal{E}}^!$ or $\on{id}_{\mathcal{E}}[m]$. The vertical maps are equivalences as tensor products of equivalences, showing the non-degeneracy of the Hochschild homology class.
\end{proof}

\subsection{Gluing Calabi--Yau structures}\label{subsec:CYglue}

In this section, we discuss a generalization of the gluing theorem for left Calabi--Yau structures on $k$-linear dg categories, see \cite[Theorem 6.1]{BD19}, to left and right Calabi--Yau structures on dualizable $R$-linear $\infty$-categories, see \Cref{thm:leftCYglue,thm:rightCYglue}. The gluing theorems boil down to a simple description of objects in pullbacks/pushouts of stable, presentable $\infty$-categories in terms of their restrictions given in \Cref{lem:obj}. 

Consider the simplicial set $Z=\Delta^{\{0,1\}}\amalg_{\Delta^{\{0\}}}\Delta^{\{0,1'\}}$ describing a span with objects $0,1,1'$ and two non-degenerate $1$-simplices $0\to 1,1'$. We fix a diagram $D\colon Z\rightarrow \on{LinCat}_R$ with colimit $\mathcal{C}$, satisfying that $D$ maps each morphism to a functor admitting an $R$-linear right adjoint. For $z\in Z$, we denote $\mathcal{C}_z=D(z)$ and can depict the colimit diagram of $D$ as follows:
\[
\begin{tikzcd}
\C_0 \arrow[r] \arrow[d] \arrow[rd, "i_0"] & \C_1 \arrow[d, "i_1"] \\
\C_{1'} \arrow[r, "i_{1'}"]                & \C                   
\end{tikzcd}
\] 
Above $i_z\colon \mathcal{C}_z\rightarrow \mathcal{C}$ denotes the functor from the colimit cone. We further denote by $j_z=\on{radj}(i_z)$ the right adjoint, and $\counit_z$ the counit of $i_z\dashv j_z$, for $z\in Z$.

The fact that counits compose to counits provides us with a commutative square \[ \phi_D\colon Z^\triangleright\rightarrow  \on{Lin}_R(\mathcal{C},\mathcal{C})\] 
which can be depicted as follows:
\[
\begin{tikzcd}
i_0j_0 \arrow[d] \arrow[r] \arrow[rd, "\counit_0"] & i_1j_1 \arrow[d, "\counit_1"] \\
i_{1'}j_{1'} \arrow[r, "\counit_{1'}"]             & \on{id}_{\C}                 
\end{tikzcd}
\]

\begin{proposition}\label{prop:phiD1}
The square $\phi_D\colon Z^\triangleright \rightarrow \on{Lin}_R(\mathcal{C},\mathcal{C})$ is biCartesian. 
\end{proposition}

\begin{proof}
Using that the forgetful functor $\on{Lin}_R(\mathcal{C},\mathcal{C})\rightarrow \on{Fun}(\mathcal{C},\mathcal{C})$ reflects finite colimits, see \Cref{lem:internalHomexact}, and that colimits in functor categories are computed pointwise, see \cite[5.1.2.3]{HTT}, the statement reduces to \Cref{lem:obj}.
\end{proof}

\begin{lemma}\label{lem:obj}
Let $c\in \mathcal{C}$ and consider the diagram
\[ \phi_c\colon Z^\triangleright\rightarrow \mathcal{C}\,,\]
given by evaluating $\phi_D$ at $c$. Then $\phi_c$ describes a biCartesian square in $\C$.
\end{lemma}

\begin{proof}
We can identify the colimit $\mathcal{C}$ of $D$ with the $\infty$-category of coCartesian sections of the Grothendieck construction $\pi\colon \Gamma(\on{radj}(D))\rightarrow Z^{\on{op}}$ of the diagram obtained from $D$ by passing to right adjoint functors, i.e.~the coCartesian fibration classified by $\on{radj}(D)$, see 
\cite[\href{https://kerodon.net/tag/05RX}{Prop.~05RX}]{Ker}. Denote by $\mathcal{L}$ the $\infty$-category of all sections of $\pi$ and by $\kappa\colon \mathcal{C}\hookrightarrow \mathcal{L}$ the fully faithful inclusion with left adjoint $\zeta$. For $z\in Z$, we denote by $\tilde{j}_z\colon \mathcal{L}\to \C_z$ the evaluation functor at $z$, satisfying that $\tilde{j}_z\circ \kappa \simeq j_z$. The left adjoint of $\tilde{j}_z$ is denoted $\tilde{i}_z$, it satisfies $\zeta\circ \tilde{i}_z\simeq i_z$. An object $c\in \mathcal{C}$ corresponding to a coCartesian section of $\pi$ is of the form
\[
\begin{tikzcd}
c_1 \arrow[rd] &   & c_{1'} \arrow[ld] \\
             & c_{0} &             
\end{tikzcd}
\]
with $c_i\in \C_i$. By \cite[4.3.2.17]{HTT}, $\tilde{i}_z$ is given by the $\pi$-relative left Kan extension functor and the objects $\tilde{i}_zj_z(c)$, $z\in Z$ are hence given as follows.
\[\tilde{i}_0{j}_0(c)\simeq \begin{tikzcd}
0 \arrow[rd] &     & 0 \arrow[ld] \\
               & c_0 &             
\end{tikzcd}\]
\[ \tilde{i}_1{j}_1(c)\simeq
\begin{tikzcd}
c_1 \arrow[rd] &     & 0 \arrow[ld] \\
               & c_0 &             
\end{tikzcd}
\quad\quad \tilde{i}_{1'}{j}_{1'}(c)\simeq 
\begin{tikzcd}
0 \arrow[rd] &     & c_{1'} \arrow[ld] \\
               & c_0 &             
\end{tikzcd}
\]

These assemble into a square $\tilde{\phi}_c$ in $\mathcal{L}$ of the form
\[
\begin{tikzcd}
\tilde{i}_0j_0(c) \arrow[d] \arrow[r] & \tilde{i}_1j_1(c) \arrow[d] \\
\tilde{i}_{1'}j_{1'}(c) \arrow[r]                                    & c                          
\end{tikzcd}
\]
which restricts at $0\in Z$ to the constant diagram with value $c_0$, up to equivalence, and at $i=1,1'\in Z$ to the fiber sequence of the map $c_i\rightarrow 0$ in $\C_i$. Using that limits in the $\infty$-category $\mathcal{L}$ of sections of the Grothendieck construction are computed componentwise in $Z$, it follows that $\tilde{\phi}_c$ is a limit diagram in $\mathcal{L}$. Using that $\zeta\colon \mathcal{L}\rightarrow \mathcal{C}$ is exact, we conclude that $\phi_c\simeq \zeta\circ \tilde{\phi}_c$ is a limit diagram as well.
\end{proof}

\Cref{prop:phiD1} implies that $R$-linear smoothness is preserved under finite colimits along dualizable functors. A variant of this observation for $k$-linear $\infty$-categories appears in \cite[Lemma 8.21]{ST16}.

\begin{corollary}\label{cor:colimissmooth}
Let $W$ be a finite simplicial set and $D\colon W\rightarrow \on{LinCat}_R^{\on{dual}}$ a functor taking values in smooth $R$-linear $\infty$-categories. Then the colimit $\C$ of $D$ in $\on{LinCat}_R^{\on{dual}}$ is also smooth.
\end{corollary}

\begin{proof}
Any finite colimit can be computed in terms of pushouts and finite coproducts. Smoothness is clearly preserved under finite coproducts. It thus suffices to check that the pushout of a span of compactly generated, smooth $R$-linear $\infty$-categories along compact objects preserving functors is again smooth. This follows from combining the the fact that the forgetful functor $\on{LinCat}_R^{\on{dual}}\to \on{LinCat}_R$ preserves colimits, see \cite[Prop.~1.65]{Efi24}, with \Cref{lem:obj}, part (1) of \Cref{lem:pullingbackendofunctors} and the fact that pushouts of compact objects are again compact.
\end{proof}

We next discuss a dual version of \Cref{prop:phiD1} for limit diagrams of dualizable
$\infty$-categories. Fix a cospan $D\colon Z^{\on{op}}\to \on{LinCat}_R^{\on{dual}}$. For $z\in Z^{\on{op}}$, we denote $\C_z=D(z)$. Let $\tilde{D}\colon (Z^{\on{op}})^\triangleleft\to \on{LinCat}_R^{\on{dual}}$ be a cone over $D$, with tip denoted $\C$. We depict this cone as follows:
\begin{equation}\label{eq:Cpbdiag}
\begin{tikzcd}
\C \arrow[r, "j_1"] \arrow[d, "j_{1'}"] \arrow[rd, "j_0"] & \C_1 \arrow[d] \\
\C_{1'} \arrow[r]                                         & \C_0          
\end{tikzcd}
\end{equation}
We denote by $k_z$ the right adjoint of $j_z$ for $z\in Z^{\on{op}}$. The units of these adjunctions assemble into the diagram $\phi^D\colon (Z^{\on{op}})^{\triangleleft}\to \on{Lin}_R(\C,\C)$, depicted as follows.
\[\begin{tikzcd}
\on{id}_{\C} \arrow[d, "\unit_{1'}"] \arrow[r, "\unit_1"] \arrow[rd, "\unit_0"] & k_1j_1 \arrow[d] \\
k_{1'}j_{1'} \arrow[r]                                                 & k_0j_0          
\end{tikzcd}
\]

\begin{proposition}\label{prop:phiD2}
If $\tilde{D}$ is a limit cone in $\on{LinCat}_R^{\on{dual}}$, then $\phi^D$ is a biCartesian square. 
\end{proposition}

\begin{proof}
We prove below that if the composition of $\tilde{D}$ with the forgetful functor $\on{LinCat}_R^{\on{dual}}\to \on{LinCat}_R$ defines a limit cone, then $\phi^D$ is a biCartesian square. Note that the pullback $\C_1$ of $D$ in $\on{LinCat}_R^{\on{dual}}$ embeds fully faithfully into the pullback $\C_2$ of $D$ in $\on{LinCat}_R$ via an $R$-linear functor $\iota\colon \C_1\hookrightarrow \C_2$ admitting an $R$-linear right adjoint, see \cite[Prop.~1.87]{Efi24}. Hence the square $\phi^D$ in $\on{Lin}_R(\C_1,\C_1)$ arises from applying to the the biCartesian square $\on{Lin}_R(\C_2,\C_2)$ the exact functor $\on{radj}(\iota)\circ (\mhyphen)\circ \iota$, which shows that $\phi^D$ is biCartesian. 

We thus suppose that $\tilde{D}$ expresses $\C$ as the pullback of $D$ in $\on{LinCat}_R$. Using \Cref{lem:internalHomexact}, and that colimits in functor categories are computed pointwise, see \cite[5.1.2.3]{HTT}, it suffices to show that for any $c\in \C$ the diagram $\phi^c\coloneqq \phi^D(c)$ is biCartesian in $\D$. By passing to opposite $\infty$-categories (which exchanges left and right adjoints, as well as unit and counit maps), the argument from the proof of \Cref{lem:obj} directly applies to show that $\phi^c$ is biCartesian.
\end{proof}

\begin{corollary}\label{cor:limitisproper}
Let $W$ be a finite simplicial set and $D\colon W\rightarrow \on{LinCat}_R^{\on{dual}}$ a functor taking values in proper $R$-linear $\infty$-categories. Then the limit $\C$ of $D$ in $\on{LinCat}_R^{\on{dual}}$ is also proper.
\end{corollary}

\begin{proof}
Since finite limits are generated by products (for which the statement is clear) and pullbacks, the statement follows from combining \Cref{prop:phiD2} and part (2) of \Cref{lem:pullingbackendofunctors}.
\end{proof}

The above discussion provides us with the tools needed for proving the gluing results for Calabi--Yau structures. We begin with the gluing of left Calabi--Yau structures. For this, fix a colimit diagram in $\on{LinCat}_R^{\on{dual}}$, valued in smooth $\infty$-categories of the following form:
\begin{equation}\label{eq:composeCYcospans}
\begin{tikzcd}
                        &                                                                    & \B_3 \arrow[d, "F_{3,2}"] \\
                        & \B_2 \arrow[rd, "\ulcorner", phantom] \arrow[r, "F_{2,2}"] \arrow[d, "F_{2,1}"'] & \C_2 \arrow[d, "F_2"] \\
\B_1 \arrow[r, "F_{1,1}"] & \C_1 \arrow[r, "F_1"]                                            & \D         
\end{tikzcd}
\end{equation}
We form the following diagram in $\on{RMod}_R$: 
\[
\begin{tikzcd}[column sep=60]
X \arrow[d] \arrow[r] \arrow[rd, "\square", phantom]                                                                                    & \HH(\B_2)^{S^1} \arrow[d, "{(\on{id},-\on{id})}"] \arrow[r] \arrow[rd, "\square", phantom] & 0 \arrow[d]                      \\
{{\bigoplus_{i=1,2} \HH(\C_{i},\B_i\times \B_{i+1})^{S^1}[-1]}} \arrow[d] \arrow[r] \arrow[rd, "\square", phantom] & (\on{HH}(B_2)^{S^1})^{\oplus 2} \arrow[d, "{\HH(F_{2,1})^{S^1}\oplus \HH(F_{2,2})^{S^1}}"] \arrow[r, "{(\on{id},\on{id})}"]                     & \HH(\B_2)^{S^1} \arrow[d] \\
\HH(\B_1)^{S^1}\oplus \on{HH}(\B_3)^{S^1} \arrow[r, "{\HH(F_{1,1})^{S^1}\oplus \HH(F_{3,2})^{S^1}}"]                                                                           & \HH(\C_1)^{S^1}\oplus \HH(\C_2)^{S^1} \arrow[r, "{\left(\HH(F_1)^{S^1}, \HH(F_2)^{S^1}\right)}"']                              & \HH(\D)^{S^1}            
\end{tikzcd}
\]
The outer square of the above diagram, though not necessarily biCartesian, induces a morphism $X\to \HH(\D,\B_1\oplus \B_3)^{S^1}[-1]$. A class $R[n]\to X[1]$ corresponds to two classes $R[n]\to \HH(\C_1,\B_1\times \B_2)^{S^1},\HH(\C_2,\B_2\times \B_3)^{S^1}$, whose restrictions to $\HH(B_2)^{S^1}[1]$ are \textit{not} identical, but differ exactly by a reversal of the sign, i.e.~composition with $-\HH(\on{id}_{\B_2})^{S^1}$. In this case, we say that the restrictions of the classes to $\HH(B_2)^{S^1}[1]$ are compatible.

\begin{theorem}\label{thm:leftCYglue}
Consider two classes $\eta_i\colon R[n]\to \on{HH}(\C_i,\B_i\times \B_{i+1})^{S^1}$, with $i=1,2$, whose restrictions to $\HH(\B_2)^{S^1}[1]$ are compatible and let $\eta\colon R[n]\to \HH(\D,\B_1\times \B_3)^{S^1}$ be the arising class. If $\eta_1$ and $\eta_2$ define left $n$-Calabi--Yau structures on the functors 
\[ \B_i\times \B_{i+1}\longrightarrow \C_i\,,~i=1,2\,,\] 
then $\eta$ defines a left $n$-Calabi--Yau structure on the functor 
\[ \B_1\times \B_3\longrightarrow \D\,.\]
\end{theorem}

\begin{proof}
For $\mathcal{X}=\B_1,\B_2,\B_3,\C_1,\C_2$, denote by $i_{\mathcal{X}}\colon \mathcal{X}\to \D$ the functor from \eqref{eq:composeCYcospans}. Let $j_{\mathcal{X}}$ be the right adjoint of $i_{\mathcal{X}}$. Since the restriction of \eqref{eq:composeCYcospans} to $\B_2,\C_1,\C_2$ and $\D$ is a pushout diagram, we find by \Cref{prop:phiD1} a biCartesian square in $\on{Lin}_R(\D,\D)$, which is depicted as follows.
\[ 
\begin{tikzcd}
i_{\B_2}j_{\B_2} \arrow[r, "\alpha_1"] \arrow[d, "\alpha_2"'] \arrow[rd, "\square", phantom] & i_{\C_1}j_{\C_1} \arrow[d, "\beta_1"] \\
i_{\C_2}j_{\C_2} \arrow[r, "\beta_2"]                                                        & \on{id}_{\D}                                  
\end{tikzcd}
\]
The sequence
\[i_{\B_2}j_{\B_2}\xlongrightarrow{(\alpha_1,-\alpha_2)}  i_{\C_1}j_{\C_1}\oplus  i_{\C_2}j_{\C_2} \xlongrightarrow{(\beta_1,\beta_2)}\on{id}_{\D}\]
is hence a fiber and cofiber sequence.

Using the pasting law for biCartesian squares, this gives rise to the following  commutative diagram in $\on{Lin}_R(\D,\D)$, where all squares are biCartesian and all objects are compact.
\begin{equation}\label{eq:CYgluethm}
\begin{tikzcd}
V \arrow[r] \arrow[d] \arrow[rd, "\square", phantom]            & i_{\B_2}j_{\B_2} \arrow[d, "{(\on{id},-\on{id})}"] \arrow[r] \arrow[rd, "\square", phantom]                                          & 0 \arrow[d]                                  \\
U \arrow[d] \arrow[r] \arrow[rd, "\square", phantom]            & i_{\B_2}j_{\B_2}\oplus i_{\B_2}j_{\B_2} \arrow[r, "{(\on{id},\on{id})}"] \arrow[d, "\alpha_1\oplus \alpha_2"] \arrow[rd, "\square", phantom] & i_{\B_2}j_{\B_2} \arrow[d] \\
{i_{\B_1}j_{\B_1}\oplus i_{\B_3}j_{\B_3}} \arrow[r] & {i_{\C_1}j_{\C_1}\oplus i_{\C_2}j_{\C_2}} \arrow[r, "\beta_1\oplus \beta_2"]                                                                                 & \on{id}_{\D}             
\end{tikzcd}
\end{equation}
The image under $(\mhyphen)^!$ of the counit $\counit_{\mathcal{X}}$ is given by the unit $\tilde{\unit}\colon \on{id}_{\mathcal{D}}^!\to i_{\mathcal{X}}\on{id}_{\mathcal{X}}^!j_{\mathcal{X}}$, see \Cref{lem:dualcounit}. Applying the exact contravariant functor $(\mhyphen)^!$ to \eqref{eq:CYgluethm} yields the following diagram, up to equivalence.
\[
\begin{tikzcd}
\on{id}_{\D}^! \arrow[r] \arrow[d] \arrow[rd, "\square", phantom]                                                                                              & i_{\B_2}\on{id}_{\B_2}^!j_{\B_2} \arrow[d, "{(\on{id},-\on{id})}"] \arrow[r] \arrow[rd, "\square", phantom]                                                             & 0 \arrow[d]                                              \\
i_{\C_1}\on{id}_{\C_1}^!j_{\C_1}\oplus i_{\C_2}\on{id}_{\C_2}^!j_{\C_2} \arrow[d] \arrow[r] \arrow[rd, "\square", phantom] & i_{\B_2}\on{id}_{\B_2}^!j_{\B_2}\oplus i_{\B_2}\on{id}_{\B_2}^!j_{\B_2} \arrow[r, "{(\on{id},\on{id})}"] \arrow[d] \arrow[rd, "\square", phantom] & i_{\B_2}\on{id}_{\B_2}^!j_{\B_2} \arrow[d] \\
i_{\B_1}\on{id}_{\B_1}^!j_{\B_1}\oplus i_{\B_3}\on{id}_{\B_3}^!j_{\B_3} \arrow[r]                                          & U^! \arrow[r]                                                                                                                                                                         & V^!                                                     
\end{tikzcd}
\]
The classes $\sigma_1,\sigma_2$ define an equivalence between the lower left squares and upper right squares of the lower diagram and the $(1-n)$-th suspension of the upper diagram. These equivalences extend to an equivalence of the entire diagrams by using that the lower right and upper left squares are biCartesian. Restricting the equivalence to the outer biCartesian squares provides us with a diagram in $\on{Lin}_R(\D,\D)$
\[
\begin{tikzcd}
\id_{\D}^! \arrow[r, "\tilde{\unit}"] \arrow[d, "\simeq"] & (i_{\B_1}j_{\B_1})^!\oplus (i_{\B_3}j_{\B_3})^! \arrow[r] \arrow[d, "\simeq"] & V^! \arrow[d, "\simeq"] \\
{V[n-1]} \arrow[r]                                & {i_{\B_1}j_{\B_1}[1-n]\oplus i_{\B_3}j_{\B_3}[1-n]} \arrow[r, "\counit"]      & {\id_{\D}[1-n]}        
\end{tikzcd}
\]
with horizontal fiber and cofiber sequences. This diagram arises from the class 
\[R[n]\xrightarrow{\eta}\HH(\D,\B_1\times \B_3)^{S^1}\to \HH(\D,\B_1\times \B_3)\,,\]
thus showing that $\eta$ defines a left $n$-Calabi--Yau structure.
\end{proof}

We next describe the gluing properties of right Calabi--Yau structures along pullbacks. For this, we consider a limit diagram in $\on{LinCat}_R^{\on{dual}}$ valued in proper $R$-linear $\infty$-categories of the following form:
\[
\begin{tikzcd}
\D \arrow[d, "F_1"] \arrow[r, "F_2"] \arrow[rd, "\lrcorner", phantom] & \C_2 \arrow[d, "F_{2,2}"] \arrow[r, "F_{2,3}"] & \B_3 \\
\C_1 \arrow[d, "F_{1,1}"] \arrow[r, "F_{1,2}"]                                & \B_2                     &               \\
\B_1                                                    &                                   &              
\end{tikzcd}
\]

We form the following diagram in $\on{RMod}_R$:
\begin{equation}\label{eq:XHH}
\begin{tikzcd}[column sep=huge]
X \arrow[d] \arrow[r] \arrow[rd, "\square", phantom]                                                                                    & \on{HH}(\B_2)_{S^1}^* \arrow[d, "{(\on{id},-\on{id})}"] \arrow[r] \arrow[rd, "\square", phantom] & 0 \arrow[d]                      \\
{{\bigoplus_{i=1,2} \HH(\C_{i},\B_i\times \B_{i+1})_{S^1}^*[-1]}} \arrow[d] \arrow[r] \arrow[rd, "\square", phantom] & (\HH(\B_2)_{S^1}^*)^{\oplus 2} \arrow[d, "{\HH(F_{2,1})_{S^1}^*\oplus \HH(F_{2,2})_{S^1}^*}"] \arrow[r, "{(\on{id},\on{id})}"]                     & \HH(\B_2)_{S^1}^* \arrow[d] \\
\HH(\B_1)^*\oplus \HH(\B_3)_{S^1}^* \arrow[r, "{\HH(F_{1,1})_{S^1}^*\oplus \HH(F_{3,2})_{S^1}^*}"]                                                                           & \HH(\C_1)^*\oplus \HH(\C_2)^* \arrow[r, "{\left(\HH(F_1)_{S^1}^*, \HH(F_2)_{S^1}^*\right)}"']                              & \HH(\D)_{S^1}^*          
\end{tikzcd}
\end{equation}

Similar to the smooth case, a class in $X$ consists of classes in $\on{HH}(\C_{i},\B_i\times \B_{i+1})^*_{S^1}[-1]$, with $i=1,2$, whose restrictions to $\on{HH}(\B_2)_{S^1}^*$ differ by sign, and we again call such classes compatible. 

\begin{theorem}\label{thm:rightCYglue} 
Consider two classes $\eta_i\colon R[n]\to\on{HH}(\C_{i},\B_i\times \B_{i+1})_{S^1}^*$, with $i=1,2$, whose restrictions to $\HH(\B_2)_{S^1}^*[1]$ are compatible and let $\eta \colon R[n]\to \HH(\D,\B_1\times \B_3)_{S^1}^*$ be the arising class. If $\eta_1$ and $\eta_2$ define right $n$-Calabi--Yau structures on the functors 
\[ \C_i\longrightarrow \B_i\times \B_{i+1}\,,~i=1,2\,,\] 
then $\eta$ also defines a right $n$-Calabi--Yau structure on the functor 
\[ \D\longrightarrow \B_1\times \B_3\,.\]
\end{theorem}

\begin{proof}
Using \Cref{prop:phiD2}, this can be proven in the same way as \Cref{thm:leftCYglue}.
\end{proof}

\begin{remark}
The analogues of \Cref{thm:leftCYglue,thm:rightCYglue} for weak relative Calabi--Yau structures clearly hold as well.
\end{remark}

\section{Perverse schobers on surfaces with boundary}

\subsection{Surfaces, ribbon graphs and line fields}

\begin{definition}\label{def:surf}
By a surface, we will mean a smooth, connected, oriented surface ${\bf S}$ with non-empty boundary $\partial {\bf S}$ and interior ${\bf S}^\circ$. We will also assume that ${\bf S}$ is compact, unless stated otherwise. Note that if ${\bf S}$ is compact, the boundary $\partial {\bf S}$ consists of a disjoint union of circles.

A marked surface $({\bf S},M)$ consists of a surface and a non-empty finite set $M\subset \partial {\bf S}$ of marked points, lying on the boundary of ${\bf S}$. We do not require that each connected component of $\partial {\bf S}$ contains at least one marked point. 
\end{definition}

\begin{definition}
\begin{itemize}
\item A graph ${\rgraph}$ consists of two finite sets ${\rgraph}_0$ of vertices and $\on{H}_{\rgraph}$ of halfedges (mostly simply denoted by $\on{H}$) together with an involution $\tau\colon \on{H}\rightarrow \on{H}$ and a map $\sigma\colon \on{H}\rightarrow \rgraph_0$.
\item Let $\rgraph$ be a graph. An edge of $\rgraph$ is defined to be an orbit of $\tau$. The set of edges is denoted by $\rgraph_1$. An edge is called internal if the orbit contains two elements and called external if the orbit contains a single element. An internal edge is called a loop at a vertex $v\in \rgraph_0$ if it consists of two halfedges both being mapped under $\sigma$ to $v$. We denote the set of external edges of $\rgraph$ by $\rgraph^\partial_1$.
\item A ribbon graph consists of a graph $\rgraph$ together with a choice of a cyclic order on the set $\on{H}(v)$ of halfedges incident to each vertex $v$.
\end{itemize}
We will always assume graphs to be connected.
\end{definition}

\begin{definition}
Let $\rgraph$ be a graph. We define the exit path category $\on{Exit}(\rgraph)$ of $\rgraph$ to be the nerve of the $1$-category with
\begin{itemize}
\item objects the vertices and edges of $\rgraph$ and
\item a non-identity morphism of the form $v\rightarrow e$ for every vertex $v$ and incident edge $e$. If $e$ is a loop at $v$, then there are two morphisms $v\rightarrow e$.
\end{itemize}
The geometric realization $|\rgraph|$ of $\rgraph$ is defined as the geometric realization $|\on{Exit}(\rgraph)|$ of $\on{Exit}(\rgraph)$. 
\end{definition}

\begin{remark}
A graph $\rgraph$ whose geometric realization $|\rgraph|$ is embedded into an oriented surface ${\bf S}$ inherits a canonical ribbon graph structure by requiring the halfedges at any vertex to be ordered in the counterclockwise direction.
\end{remark}

\begin{definition}\label{def:sg}
Let ${\bf S}$ be a marked surface. A spanning graph for ${\bf S}$ consists of a graph $\rgraph$ together with an embedding $i:|\rgraph|\subset {\bf S}\backslash M$ satisfying that 
\begin{itemize}
\item $i$ is a homotopy equivalence,
\item $i$ maps $\partial |\rgraph|$ to $\partial {\bf S}$, and
\item the restriction $\partial |\rgraph|\rightarrow \partial {\bf S}\backslash M$ gives a homotopy equivalence with the boundary components which do not contain marked points.
\end{itemize}
We consider a spanning graph of ${\bf S}$ as endowed with the canonical ribbon graph structure arising from the embedding into ${\bf S}$.
\end{definition}

We now turn to line fields and framings on surfaces.

\begin{definition}
Let $\Sigma$ be a possibly non-compact surface.
\begin{enumerate}[(1)]
\item A line field $\nu$ on $\Sigma$ is a section of the projectivized tangent bundle $\mathbb{P}T\Sigma$.
\item  Assume that $\Sigma$ is equipped with a line field $\nu$ and let $\gamma\colon S^1\simeq [0,1]/(0\sim 1)\to \Sigma$ be a smooth, immersed loop. We denote by $W(\gamma)\in \mathbb{Z}\simeq \pi_1\mathbb{P}T_{\gamma(0)}\Sigma$ the winding number of $\gamma$ with respect to $\nu$.

For the equivalence $\mathbb{Z}\simeq \pi_1\mathbb{P}T_{\gamma(0)}\Sigma$, we use the convention that $1$ corresponds to a counterclockwise half-turn.
\end{enumerate}
We let $\on{Fr}(\Sigma)\to \Sigma$ denote the principal $\on{Gl}(2,\mathbb{R})$-bundle of frames (i.e.~ordered bases) of the tangent bundle $T\Sigma$.
\begin{enumerate}
\item[(3)] A framing $\xi$ on $\Sigma$ is a section of $(T\Sigma\backslash \{0\})/\mathbb{R}_+$.
\end{enumerate}
\end{definition}

The winding number $W(\gamma)$ of an immersed loop $\gamma$ can be obtained as follows: one chooses any homotopy in $\mathbb{P}T_{\gamma(0)}\Sigma$ from the tangent vector $\dot{\gamma}(0)$ to $\nu(\gamma(0))$, extends this homotopy to a homotopy of sections $\Gamma([0,1];\gamma^*\mathbb{P}T\Sigma)$ from $\dot{\gamma}$ to $\gamma^*\nu$, and then composes the two homotopies between $\dot{\gamma}(1)=\dot{\gamma}(0)$ and $\nu(\gamma(0))=\nu(\gamma(1))$ to obtain a loop in $\mathbb{P}T_{\gamma(0)}\Sigma$. Informally, this counts the number of half-rotations of the tangent field along $\gamma$ with respect to the line field. 

\begin{remark} 
By projecting onto the first element of the ordered basis, we obtain a map of fiber bundles $\on{Fr}(\Sigma)\to T\Sigma\to \mathbb{P}T\Sigma$. By composing with this map, any framing gives rise to a line field, all of whose winding numbers are even. Conversely, any line field with even winding numbers arises from a framing, see for instance \cite[Lem.~1.1.4]{LP20}.

The set of homotopy classes of line fields is a $H^1(\Sigma,\mathbb{Z})$-torsor, see \cite{LP20}.
\end{remark}

Finally, we describe how a choice of spanning ribbon graph induces a line field on the complement of its vertices.

\begin{example}\label{ex:linefield}
Let ${\bf S}$ be a marked surface with a spanning graph $\rgraph$. Then ${\bf S}\backslash \rgraph_0$ inherits a canonical (homotopy class of a) line field $\nu_\rgraph$, which we can depict locally near a vertex (of valency $5$ for concreteness) of $\rgraph$ as follows:
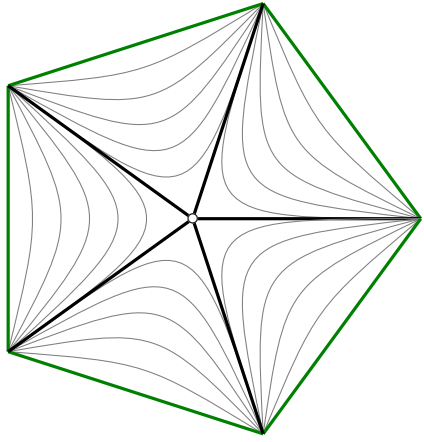
\begin{center}
    \begin{tikzpicture}
        \foreach \n in {0,...,4}
        {
            \coordinate (v) at (360/5*\n:3);
            \coordinate (w) at (360/5*\n+360/5:3);
            \foreach \m in {0,.5,1,1.5,2}
            {
                \coordinate (u) at (360/5*\n+36:\m);
                \draw[gray] (v) .. controls (u) and (u) .. (w);
            }
            \coordinate (u) at (360/5*\n+36:3.3);
            \draw[gray] (0,0) -- (v);
            
            \draw[ao , very thick] (v) -- (w);

            \draw[very thick, black] (0,0) -- (360/5*\n:3);
		 }
        
        \draw[black] (0,0) \ww;
    \end{tikzpicture}
\end{center}
The winding number with respect to $\nu_{\rgraph}$ of an embedded loop wrapping clockwise around a vertex of valency $m$ is thus given by $-m$.

Suppose that $\rgraph\to \rgraph'$ is a contraction between spanning graphs of ${\bf S}$, contracting a set $E$ of edges of $\rgraph$. Then the corresponding line fields are closely related: the winding numbers of immersed loops not intersecting the edges in $E$ are not affected by the contraction.  
\end{example}

\subsection{Perverse schobers}\label{sec3.2}\label{subsec:schobers}

In this section, we recall the definition of perverse schober parametrized by a ribbon graph. For more background on perverse schobers on surfaces, see \cite[Section 3]{CHQ23}, which refines the treatment in \cite[Sections 3,4]{Chr22}.

Given $n\in\mathbb{N}_{\geq 1}$, we denote by $\rgraph_{n}$ the ribbon graph with a single vertex $v$ and $n$ incident external edges. We call $\rgraph_n$ the $n$-spider. 

\begin{definition}\label{def:schberngon}
Let $n\geq 1$. An $R$-linear perverse schober parametrized by the $n$-spider, or on the $n$-spider for short, consists of the following data: 
\begin{enumerate}
\item[(1)] If $n=1$, an $R$-linear spherical adjunction
\[ F\colon \mathcal{V}\longleftrightarrow \mathcal{N}\cocolon G\,,\]
i.e.~an adjunction whose twist functor\footnote{There is a different much used convention for the definition of the twist and cotwist functors in the literature, see for instance \cite{ST01,AL17}. This convention differs from the one considered in this paper in two ways: Firstly, instead of forming the cone of the unit, the other convention considers the cocone of the unit and similarly the cone of the counit. Furthermore, the labels of the twist and cotwist are swapped. 

Let us illustrate our convention in the case of a spherical object: Thinking of $\mathcal{V}$ as the categorified vanishing cycles, this category is then given by $\mathcal{V}=\D(k)$. The twist functor acts on $\D(k)$ as a shift functor. The cotwist is the arising autoequivalence $T_\mathcal{N}$ of the category $\mathcal{N}$ containing the spherical object.} $T_{\mathcal{V}}=\on{cof}(\on{id}_{\mathcal{V}}\xrightarrow{\on{unit}}GF)\in \on{Fun}(\mathcal{V},\mathcal{V})$ and cotwist functor $T_{\mathcal{N}}=\on{fib}(FG\xrightarrow{\on{counit}}\on{id}_{\mathcal{N}})\in \on{Fun}(\mathcal{N},\mathcal{N})$ are equivalences. Such functors $F$ and $G$ are also called spherical functors, see \cite{AL17}. We also point out, that such spherical functors $F,G$ automatically admit all left and right adjoints, given by compositions of $F$ and $G$ with powers of the (co)twist functors, see \cite{DKSS21}.
\item[(2)] If $n\geq 2$, a collection of $R$-linear adjunctions
\[ (F_i\colon \mathcal{V}^n\longleftrightarrow \mathcal{N}_i\cocolon G_i)_{i\in \mathbb{Z}/n\mathbb{Z}}\]
satisfying that
\begin{enumerate}
    \item $G_i$ is fully faithful, i.e.~$F_iG_i\simeq \on{id}_{\mathcal{N}_i}$ via the counit,
    \item $F_{i}\circ G_{i+1}$ is an equivalence of $\infty$-categories,
    \item $F_i\circ G_j\simeq 0$ if $j\neq i,i+1$,
    \item $G_i$ admits a right adjoint $\on{radj}(G_i)$, and
    \item $\on{fib}(\on{radj}(G_{i+1}))=\on{fib}(F_{i})$ as full subcategories of $\mathcal{V}^n$.
\end{enumerate}
\end{enumerate}
We will also consider a collection of functors $(F_i\colon \mathcal{V}^n\rightarrow \mathcal{N}_i)_{i\in \mathbb{Z}/n\mathbb{Z}}$ as determining a perverse schober on the $n$-spider, or as a perverse schober on the $n$-spider for short, if there exist adjunctions $(F_i\dashv \on{radj}(F_i))_{i\in \mathbb{Z}/n\mathbb{Z}}$ which define a perverse schober on the $n$-spider. Such a collection of functors $(F_i)_{i\in \mathbb{Z}/n\mathbb{Z}}$ can be equivalently encoded as a functor $\on{Fun}(\on{Exit}(\rgraph_n),\on{LinCat}_R)$.
\end{definition}

One can show that the datum of a perverse schober on the $n$-spider is for any $n\geq 1$ equivalent to the datum of a perverse schober on the $1$-spider. This gives rise to an explicit model of perverse schobers on the $n$-spider, which we describe in the following. This construction is related to Dyckerhoff's categorified Dold-Kan correspondence \cite{Dyc17} and the Waldhausen $S_\bullet$-construction, see also the discussion in \cite{Chr22}. Let $F\colon \mathcal{V}\to \mathcal{N}$ be an $R$-linear spherical functor. We denote by $\mathcal{V}^n_F$ the pullback in $\on{LinCat}_R$ of the following diagram:
\[
\begin{tikzcd}
                  & {\on{Fun}(\Delta^{n-1},\N)} \arrow[d, "\on{ev}_0"] \\
\V \arrow[r, "F"] & \N                                                
\end{tikzcd}
\]
Explicitly, the $\infty$-category $\mathcal{V}^n_F$ consists of diagram
\begin{equation}\label{eq:objVnF}a\rightarrow b_1\rightarrow \dots \rightarrow b_{n-1}\end{equation}
where $a\in \mathcal{V}$, $b_i\in \mathcal{N}$ and the morphism $a\rightarrow b_1$ lies in the Grothendieck construction of $F$. We define 
\[ \varrho_1\colon \mathcal{V}^n_F\to \mathcal{N}\] as the projection functor to $b_{n-1}$. For $1<i<n$, we define recursively define $\varrho_i$ as the doubly left adjoint of  $\varrho_{i-1}$, which is the functor 
\[ \varrho_i= \on{cof}_{n-i+1,n-i+2}[i-2]\colon \mathcal{V}^n_F\to \mathcal{N}\] 
given by a suspension of the composite of the projection functor to $b_{n-i}\to b_{n-i+1}$ with the cofiber functor. We define $\varrho_{n}$ as the doubly left adjoint of $\varrho_{n-1}$, which maps an object \eqref{eq:objVnF} to the cofiber $\on{cof}(F(a)\xrightarrow{\alpha} b_1)[m-2]$. We further define for $1\leq i\leq n$ the functor $\varsigma_i$ as the right adjoint of $\varrho_i$. We thus have a sequence of adjunctions 
\begin{equation}\label{eq:rhosigmaadj} \varrho_n\dashv \varsigma_n\dashv \varrho_{n-1}\dashv \dots \dashv \varsigma_2\dashv \varrho_1 \dashv \varsigma_1\,.\end{equation}
There is a further adjunction
\begin{equation}\label{eq:rhosigmaadj2} 
\varsigma_1T_{\mathcal{N}}^{-1}[1-n]\dashv \varrho_n\,,
\end{equation}
where $T_{\mathcal{N}}$ denotes the cotwist functor of the adjunction $F\dashv G$, see \cite[Lem.~3.8]{Chr22}.

\begin{proposition}[$\!\!$\cite{Chr22,CHQ23}]\label{prop:localmodel}
Let $F\colon \mathcal{V}\to \mathcal{N}$ be a spherical functor. The collection of adjunctions
\[(\varrho_i\colon \mathcal{V}^n_F\longleftrightarrow\mathcal{N}\noloc \varsigma_i)_{i\in \mathbb{Z}/n\mathbb{Z}} \]
define an $R$-linear perverse schober on the $n$-spider, denoted $\mathcal{F}_n(F)$. Furthermore, for every $R$-linear perverse schober on the $n$-spider $\mathcal{F}$, there exists such a spherical functor $F$ and an equivalence 
\[ \mathcal{F}\simeq \mathcal{F}_n(F)\in \on{Fun}(\on{Exit}(\rgraph_n),\on{LinCat}_R)\,.\] 
We call any choice of such $F$ the spherical functor underlying $\mathcal{F}$.
\end{proposition}

Let $v$ be a vertex of valency $n$ of a ribbon graph $\rgraph$. Let $\on{Exit}(\rgraph)_{v/}$ be the undercategory, which has $n+1$ objects, which can be identified with $v$ and its $n$ incident halfedges and non-identify morphisms going from $v$ to these halfedges. There is a functor $\on{Exit}(\rgraph)_{v/}\to \on{Exit}(\rgraph)$, which is fully faithful if $\rgraph$ has no loops incident to $v$.

\begin{definition}\label{def:schober}
Let $\rgraph$ be a ribbon graph. A functor $\mathcal{F}\colon \on{Exit}(\rgraph)\rightarrow \on{LinCat}_R$ is called an $R$-linear $\rgraph$-parametrized perverse schober if for each vertex $v$ of $\rgraph$, the restriction of $\mathcal{F}$ to $\on{Exit}(\rgraph)_{v/}$ determines a perverse schober parametrized by the $n$-spider in the sense of \Cref{def:schberngon}.

We also call a functor $\on{Exit}(\rgraph)\rightarrow \on{LinCat}_R^{\on{dual}}$ or $\on{Exit}(\rgraph)\rightarrow\on{LinCat}_R^{\on{cpt-gen}}$ a $\rgraph$-parametrized perverse schober if its composite with $\on{LinCat}_R^{\on{dual}},\on{LinCat}_R^{\on{cpt-gen}}\to \on{LinCat}_R$ is a $\rgraph$-parametrized perverse schober.
\end{definition}

\begin{remark}
A $\rgraph$-parametrized perverse schober $\mathcal{F}$ assigns to each edge of $\rgraph$ an equivalent stable $\infty$-category, referred to as the generic stalk of $\mathcal{F}$, and usually denoted by $\N$.
\end{remark}

\begin{definition}\label{def:schobersingularity}
Let $\mathcal{F}$ be an $R$-linear $\rgraph$-parametrized perverse schober. For a vertex $v$ of $\rgraph$, consider a choice of spherical functor $F_v\colon \mathcal{V}_v\to \mathcal{N}$ underlying the restriction of $\mathcal{F}$ to $\on{Exit}(\rgraph)_{v/}$ in the sense of \Cref{prop:localmodel}. We call $v$ a singularity of $\mathcal{F}$ if $\mathcal{V}_v\not\simeq 0$. 
\end{definition}

\begin{definition}\label{def:sections}
Let $\mathcal{F}$ be an $R$-linear $\rgraph$-parametrized perverse schober.
\begin{enumerate}[(1)]
\item  We denote by $\glsec(\rgraph,\mathcal{F})=\on{lim}_{\on{LinCat}_R}(\mathcal{F})$ the limit of $\mathcal{F}$ in the $\infty$-category $\on{LinCat}_R$ of $R$-linear $\infty$-categories. We call $\glsec(\rgraph,\mathcal{F})$ the $\infty$-\textit{category of global sections} of $\mathcal{F}$. 
\item Suppose that $\mathcal{F}$ takes values in compactly generated $R$-linear $\infty$-categories. We denote by $\cptglsec(\rgraph,\mathcal{F})= \on{lim}_{\on{LinCat}_R^{\on{cpt-gen}}}(\mathcal{F})$ the limit of $\mathcal{F}$ in the $\infty$-category $\on{LinCat}_R^{\on{cpt-gen}}$. We call $\cptglsec(\rgraph,\mathcal{F})$ the $\infty$-\textit{category of locally compact global sections}. 
\end{enumerate}
\end{definition}

\begin{remark}
Recall that $\glsec(\rgraph,\mathcal{F})$ agrees with the limit of $\mathcal{F}$ in $\on{Cat}_\infty$ and can thus be identified with the $\infty$-category of coCartesian sections of the Grothendieck construction of $\mathcal{F}$, i.e.~the coCartesian fibration classified by $\mathcal{F}$, see \cite[\href{https://kerodon.net/tag/05RX}{Prop.~05RX}]{Ker}.

Similarly, the $\infty$-category of locally compact global sections $\cptglsec(\rgraph,\mathcal{F})$ can be identified with the $\on{Ind}$-completion of the limit in $\on{Cat}_\infty$ of the pointwise restriction of $\mathcal{F}$ to the subcategories of compact objects. Hence a compact object $Y\in \cptglsec(\rgraph,\mathcal{F})^{\on{c}}$ consists of a pointwise compact coCartesian section of the Grothendieck construction.
\end{remark}

\begin{definition}\label{def:restrfunctor}
Given an $R$-linear $\rgraph$-parametrized perverse schober $\mathcal{F}$ and an edge $e$ of $\rgraph$, we denote by $\on{ev}_e\colon \glsec(\rgraph,\mathcal{F})\rightarrow \mathcal{F}(e)$ the evaluation functor, which maps a coCartesian section of the Grothendieck construction to its value at $e$.
\end{definition}

The functor $\on{ev}_e$ is $R$-linear and preserves limits and colimits, which can be proven using that limits and colimits in the functor $\infty$-category of sections of the Grothendieck construction are computed pointwise.

We will often consider the product of the evaluation functors at the external edges
\[\prod_{e\in \rgraph_1^\partial} \on{ev}_e\colon \glsec(\rgraph,\mathcal{F})\rightarrow  \prod_{e\in \rgraph_1^\partial} \mathcal{F}(e)\,,\]
By the $\infty$-categorical adjoint functor theorem, $\prod_{e\in \rgraph_1^\partial}\on{ev}$ admits a right adjoint, which we denoted by
\begin{equation}\label{eq:cupfunctor} 
\partial \mathcal{F}\colon \prod_{e\in \rgraph_1^\partial} \mathcal{F}(e)\rightarrow \glsec(\rgraph,\mathcal{F})\,.
\end{equation}

The adjunction $\prod_{e\in \rgraph_1^\partial} \on{ev}_e\dashv \partial \mathcal{F}$ is spherical, see \cite[Thm.~5.2.5]{CDW23} for a (sketch of) proof.

Next, we briefly discuss how to relate perverse schober parametrized by different ribbon graphs. Given a ribbon graph $\rgraph$ and an edge $e$ of $\rgraph$ which is not a loop, we can contract $e$ to create a new ribbon graph $\rgraph'$. The two vertices incident to $e$ are identified in $\rgraph'$, the edge $e$ removed, but otherwise $\rgraph'$ is the same as $\rgraph$. More generally, we say that a ribbon graph $\rgraph'$ is a contraction of $\rgraph$ if $\rgraph'$ can be obtained by contracting (automatically finitely many) edges of $\rgraph$. In this case, we write $c\colon \rgraph\to \rgraph'$. 

\begin{proposition}[$\!\!${\cite[Prop.~4.28]{Chr21b}}]\label{prop:constraction}
Let $\mathcal{F}$ be an $R$-linear $\rgraph$-parametrized perverse schober and $c\colon \rgraph\to \rgraph'$ a contraction which collapses no edge of $\rgraph$ which is incident to two singularities of $\mathcal{F}$. Then there exists a canonical $\rgraph'$-parametrized perverse schober $c_*(\mathcal{F})$ together with an $R$-linear equivalence of $\infty$-categories 
\[ \glsec(\rgraph,\mathcal{F})\simeq \glsec(\rgraph',c_*(\mathcal{F}))\,.\]
\end{proposition}

We note that if $\mathcal{F}$ above takes values in $\on{LinCat}_R^{\on{cpt-gen}}$, then
\[\cptglsec(\rgraph,\mathcal{F})\simeq \cptglsec(\rgraph',c_*(\mathcal{F}))\]
holds as well.

We end this section by describing the relationship between global sections and locally compact global sections of perverse schobers in typical situations.

\begin{lemma}\label{lem:loccpt=Indfin}
Let $\rgraph$ be a ribbon graph and $\mathcal{F}\colon \on{Exit}(\rgraph)\to \on{LinCat}_R^{\on{cpt-gen}}$ be an $R$-linear $\rgraph$-parametrized perverse schober. Suppose that for every vertex $v$ of $\rgraph$, the spherical functor $F_v\colon \V_v\to \N$ underlying $\mathcal{F}$ near $v$ is conservative. 
\begin{enumerate}[(1)]
\item A global section $Y\in\glsec(\rgraph,\mathcal{F})$ is finite in the sense of \Cref{def:Cfin} if and only if $\on{ev}_e(Y)\in \mathcal{F}(e)$ is finite for all edges $e\in \rgraph_1$. 
\item Suppose the generic stalk $\N$ is proper and that $F_v$ reflects compact objects for all vertices $v$ of $\rgraph$. Then the locally compact global sections of $\mathcal{F}$ coincide with the $\on{Ind}$-completion of the finite global sections, i.e.
\[ \cptglsec(\rgraph,\mathcal{F})\simeq \on{Ind}\glsec(\rgraph,\mathcal{F})^{\on{fin}}\,.\]
\end{enumerate}
\end{lemma}

\begin{proof}
We begin by proving part (1). For an edge $e\in \rgraph_1$, denote the left adjoint of the evaluation functor
$\on{ev}_e\colon \glsec(\rgraph,\mathcal{F})\rightarrow \mathcal{F}(e)\simeq \N$ by $\on{ev}_e^*$. A straightforward computation shows that the assumption that $F_v\colon \V_v\to \N$ is conservative for all vertices $v$ with incident edges $e_1,\dots,e_m$ is equivalent to the assertion that 
\[ \prod_{i=1}^m\mathcal{F}(v\to e_i)\colon \mathcal{F}(v)\longrightarrow \prod_{i=1}^m\mathcal{F}(e_i)\]
is conservative. Thus, a global section $Y$ of $\mathcal{F}$ vanishes if and only if $\on{ev}_e(Y)\simeq 0$ for all $e\in \rgraph_1$. 

Let $\{X_i\}_{i\in I}$ be a collection of compact generators of the generic stalk $\N$ of $\mathcal{F}$, meaning that an object $N\in \N$ vanishes if and only if $\on{Mor}_\N(X_i,N)\simeq 0$ for all $i\in I$. 
The collection of objects $\{\on{ev}_e^*(X_i)\}_{i\in I, e\in \rgraph_1}$ compactly generates $\glsec(\rgraph,\mathcal{F})$, as follows from the equivalence
\[ \on{Mor}_{\glsec(\rgraph,\mathcal{F})}(\on{ev}_e^*(X_i),Y)\simeq \on{Mor}_\N(X_i,\on{ev}_e(Y))\,.\]
Hence, the objects $\on{ev}_e^*(X_i)$ suffice to test the finiteness of global sections. A global section $Y\in \glsec(\rgraph,\mathcal{F})$ is thus finite if and only if $\on{ev}_e(Y)\in \mathcal{F}(e)\simeq \N$ is finite for all $e\in \rgraph_1$.

We proceed with proving part (2). Let $Y\in \glsec(\rgraph,\mathcal{F})$ with $\on{ev}_e(Y)\in \mathcal{F}(e)^{\on{c}}$ for all $e\in \rgraph_1$. The assertion that $F_v$ reflects compact objects for a vertex $v$, with incident edges $e_1,\dots,e_m$, is equivalent to the assertion that the functor $\prod_{i=1}^m\mathcal{F}(v\to e_i)$ reflects compact objects. It follows that $Y$ is a pointwise compact coCartesian section of the Grothendieck construction, i.e.~lies in $\cptglsec(\rgraph,\mathcal{F})^{\on{c}}$.

Since $\N$ is proper, we have $\N^{\on{fin}}\simeq \N^{\on{c}}$. Combining part (1) with the above thus shows that the finite global sections are the locally compact global sections, i.e.~$\glsec(\rgraph,\mathcal{F})^{\on{fin}}\simeq \cptglsec(\rgraph,\mathcal{F})^{\on{c}}$. Passing to $\on{Ind}$-completions shows part (2).
\end{proof}

\begin{remark}\label{rem:cptglsec=glsec}
Let $\rgraph$ be a spanning graph of a marked surface ${\bf S}$ and let $\mathcal{F}\colon \on{Exit}(\rgraph)\to \on{LinCat}_R^{\on{cpt-gen}}$ be a $\rgraph$-parametrized perverse schober with non-vanishing generic stalk. One can show that $\cptglsec(\rgraph,\mathcal{F})\simeq \glsec(\rgraph,\mathcal{F})$ if and only if each boundary component of the marked surface contains at least one marked point.
\end{remark}

\subsection{Monodromy of perverse schobers}\label{subsec:monodromy}

Let $\rgraph$ be a spanning graph of a marked surface ${\bf S}$ and $\mathcal{F}$ an $R$-linear $\rgraph$-parametrized perverse schober. We denote by $P$ the set of vertices of $\rgraph$ which are singularities of $\mathcal{F}$, see \Cref{def:schobersingularity}. In analogy with perverse sheaves on ${\bf S}$, which restrict to a local system of abelian groups or vector spaces on ${\bf S}\backslash P$, one may wish to associate a monodromy local system of $R$-linear $\infty$-categories to $\mathcal{F}$. It turns out that there is indeed a reasonable notion of monodromy of a perverse schober, but it does not \textit{canonically} assemble into a local system on ${\bf S}\backslash P$. Instead, to define the local system ${\bf S}\backslash P$, one needs to input a choice of framing of ${\bf S}\backslash P$.
Applying $K_0$, we obtain the usual local system of abelian groups of the underlying perverse sheaf; the choice of framing does not affect this local system of abelian groups.

To define the monodromy local system, we will in a first step define an auxiliary local system of transports on ${\bf S}\backslash \rgraph_0$. From this, we will obtain a local system on the projectivized tangent bundle $\mathbb{P}T({\bf S}\backslash \rgraph_0)$ with monodromy $[1]$ on the circle fiber. We then pull back the local system to the frame bundle $\on{Fr}_{{\bf S}\backslash \rgraph_0}$. Any choice of framing then allows to further pull back this local system to a local system on ${\bf S}\backslash \rgraph_0$ and this local system  extends to ${\bf S}\backslash P$, as desired.

To obtain the transport along a loop, we compose local transports. We define these in \Cref{constr:transport}. For technical convenience, we will replace in the following the surface ${\bf S}$ with spanning graph $\rgraph$ by the homotopic non-compact surface $\Sigma_{\rgraph}$ described in \Cref{ssurfrem}. 

\begin{remark}\label{ssurfrem}
Let $\rgraph$ be a ribbon graph. To each vertex $v$ of $\rgraph$ of valency $n$ we associate a (non-compact) surface, denoted $\Sigma_v$, or also $\Sigma_n$, with an embedding of $v$ and its $n$ incident halfedges. We depict $\Sigma_v$ as follows (in green). The dotted lines correspond to open ends, whereas the solid lines indicate the boundary. 
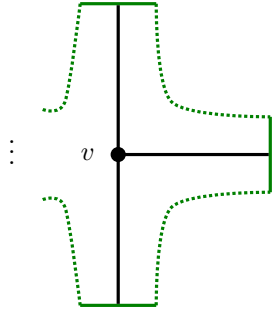
\begin{center}
\begin{tikzpicture}
 \draw[color=ao][very thick][densely dotted] plot [smooth] coordinates {(0.5,2) (0.7,0.7) (2,0.5)};
 \draw[color=ao][very thick][densely dotted] plot [smooth] coordinates {(0.5,-2) (0.7,-0.7) (2,-0.5)};
 \draw[color=ao][very thick][densely dotted] plot [smooth] coordinates {(-0.5,-2) (-0.7,-0.7) (-1,-0.6)};
 \draw[color=ao][very thick][densely dotted] plot [smooth] coordinates {(-0.5,2) (-0.7,0.7) (-1,0.6)};
 \draw[color=ao][very thick] plot [smooth] coordinates {(2,-0.5) (2,0.5)};
 \draw[color=ao][very thick] plot [smooth] coordinates {(-0.5,2) (0.5,2)};
 \draw[color=ao][very thick] plot [smooth] coordinates {(-0.5,-2) (0.5,-2)}; 
  \node (0) at (0,0){};  
  \node (1) at (-1.4,0.14){\vdots};
  \node (2) at (-0.4,0){\small $v$};
  \fill (0) circle (0.1);
  \draw[very thick]
  (0,1.98) -- (0,0)
  (0,-1.98) -- (0,0)
  (0,0) -- (1.98,0); 
\end{tikzpicture}
\end{center}
We define the thickening of $\rgraph$ to be the surface $\Sigma_\rgraph$, obtained from gluing the surfaces $\Sigma_v$, whenever two vertices are incident to the same edge, at their boundary components corresponding to the edge. The surface $\Sigma_{\rgraph}$ comes with an embedding of $|\rgraph|$, which is also a homotopy equivalence. We define the subset ${\bf B}\subset \Sigma_{\rgraph}$ as the union of the images of the boundaries $\partial \Sigma_v$ for all vertices $v$. Note that each edge $e\in \rgraph_1$ of $\rgraph$ intersects exactly one connected component of ${\bf B}$ exactly once, we denote this component by $B_e\subset {\bf B}$.
\end{remark}

\begin{construction}\label{constr:transport}
Let $n\geq 2$, the case $n=1$ is addressed at the end. The $n$-spider $\rgraph_n$ is embedded in $\Sigma_n=\Sigma_v$, we denote its central vertex by $v$. Consider an embedded curve $\delta\colon[0,1]\to \Sigma_v\backslash \{v\}$ satisfying that $\delta(0),\delta(1)\in \partial \Sigma_v$ and that the boundary component $B_{e_1}$ of $\rgraph_n$ containing $\delta(1)$ lies one step in the counterclockwise direction of the boundary component $B_{e_0}$ containing $\delta(0)$. We can depict this setup as follows:
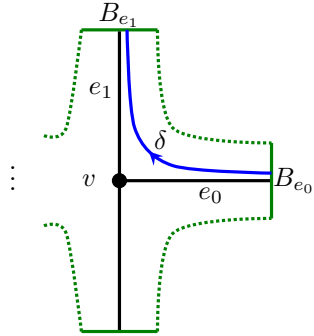
\begin{center}
\begin{tikzpicture}[decoration={markings, 
	mark= at position 0.5 with {\arrow{stealth}}}]
 \draw[color=ao][very thick][densely dotted] plot [smooth] coordinates {(0.5,2) (0.7,0.7) (2,0.5)};
 \draw[color=ao][very thick][densely dotted] plot [smooth] coordinates {(0.5,-2) (0.7,-0.7) (2,-0.5)};
 \draw[color=ao][very thick][densely dotted] plot [smooth] coordinates {(-0.5,-2) (-0.7,-0.7) (-1,-0.6)};
 \draw[color=ao][very thick][densely dotted] plot [smooth] coordinates {(-0.5,2) (-0.7,0.7) (-1,0.6)};
 \draw[color=ao][very thick] plot [smooth] coordinates {(2,-0.5) (2,0.5)};
 \draw[color=ao][very thick] plot [smooth] coordinates {(-0.5,2) (0.5,2)};
 \draw[color=ao][very thick] plot [smooth] coordinates {(-0.5,-2) (0.5,-2)}; 
  \node (0) at (0,0){};  
  \node (1) at (-1.4,0.14){\vdots};
  \node (2) at (-0.4,0){\small $v$};
  \node (3) at (0.54,0.54){\small $\delta$};
  \node (4) at (0,2.2){\small $B_{e_1}$};
  \node (4) at (-0.24,1.2){\small $e_1$};
  \node (5) at (2.3,0){\small $B_{e_0}$};
  \node (5) at (1.2,-0.2){\small $e_0$};
  \fill (0) circle (0.1);
  \draw[very thick]
  (0,1.98) -- (0,0)
  (0,-1.98) -- (0,0)
  (0,0) -- (1.98,0); 
   \draw[color=blue][very thick][postaction={decorate}] plot [smooth] coordinates {(2, 0.1) (0.7, 0.2) (0.2, 0.7) (0.1,2)};
\end{tikzpicture}
\end{center}
Given an $R$-linear $\rgraph_n$-parametrized perverse schober $\mathcal{F}$, we define the transport $\mathcal{F}^{\rightarrow}(\delta)$ of $\mathcal{F}$ along $\delta$ as the $R$-linear equivalence
\[ \mathcal{F}(e_0)\xrightarrow{\on{ladj}(\mathcal{F}(v\to e_0))}\mathcal{F}(v)\xrightarrow{\mathcal{F}(v\to e_1)}\mathcal{F}(e_1)\,.\] 

Reversing the orientation of $\delta$ yields a path $\delta^{\on{rev}}$ going one step in the clockwise direction around $v$ and the transport of $\mathcal{F}$ along $\delta^{\on{rev}}$ is defined as 
\[ \mathcal{F}^{\rightarrow}(\delta^{\on{rev}})\coloneqq \mathcal{F}(v\to e_0)\circ \on{radj}(\mathcal{F}(v\to e_1))\colon \mathcal{F}(e_1)\longrightarrow \mathcal{F}(e_0)\,.
\]
We thus have
$\mathcal{F}^{\rightarrow}(\delta^{\on{rev}})\simeq \left(\mathcal{F}^{\rightarrow}(\delta)\right)^{-1}$.

Consider now an arbitrary curve $\delta\colon[0,1]\to \Sigma_v\backslash \{v\}$ satisfying that $\delta(0),{\delta}(1)\in \partial \Sigma_v$. Then $\delta$ is homotopic relative to its endpoints either to a curve contained in the boundary $\partial \Sigma_v$, in which case we set $i=0$, or there exists $i\in \mathbb{Z}\backslash \{0\}$ such that $\delta$ is homotopic relative to its endpoints to the composite $\delta_{|i|}\ast\dots\ast \delta_1$ of $|i|$ embedded paths $\delta_1,\dots,\delta_{|i|}$ as above, each wrapping one step counterclockwise if $i>0$ and one step clockwise if $i<0$. Thus $\delta$ goes $i\in \mathbb{Z}$ steps counterclockwise. We define the transport along $\delta$ as 
\begin{equation} \mathcal{F}^{\rightarrow}(\delta)\coloneqq \begin{cases} \on{id}_{\mathcal{F}(e_0)} & i=0\,,\\
\mathcal{F}^{\rightarrow}(\delta_{|i|})\circ \dots \circ \mathcal{F}^{\rightarrow}(\delta_1)& i\neq 0\,. \end{cases}
\end{equation}
If $i=n$, i.e.~$\delta$ is a loop wrapping once counterclockwise around $v$, it follows from the adjunctions \eqref{eq:rhosigmaadj},\eqref{eq:rhosigmaadj2} that $\mathcal{F}^{\rightarrow}(\delta)\simeq T_{\N}^{-1}[1-n]$, with $T_{\mathcal{N}}^{-1}$ the inverse cotwist of the spherical adjunction underlying $\mathcal{F}$ at $v$.

We conclude with the case $n=1$. Consider a perverse schober $\mathcal{F}$ parametrized by the $1$-spider, with vertex $v$ and edge $e$, i.e.~a spherical functor $F\colon \mathcal{V}=\mathcal{F}(v)\to \mathcal{N}=\mathcal{F}(e)$. Let $\delta\colon [0,1]\to \Sigma_1\backslash \{v\}$ be a curve with $\delta(0),\delta(1)\in \partial\Sigma_1$. If $\delta$ wraps $i\in \mathbb{Z}$ times counterclockwise around $v$, we define the transport as
\[ \mathcal{F}^{\rightarrow}(\delta)\coloneqq \begin{cases} \on{id}_{\mathcal{N}} & i=0\,,\\ T_{\mathcal{N}}^i & i<0\,,\\
T_{\mathcal{N}}^{-i} & i>0\,,\end{cases}\]
where $T_{\mathcal{N}}$ denotes the cotwist functor of the adjunction $F\dashv \on{radj}(F)$.
\end{construction} 

\begin{example}\label{ex:transport}
Consider an $R$-linear perverse schober $\mathcal{F}$ on the $m$-spider $\rgraph_m$ for some $m\geq 2$, with central vertex $v$, generic stalk $\mathcal{N}$ and no singularity at $v$. Such a perverse schober categorifies a perverse sheaf without singularities (i.e.~a local system) on the disc which automatically has trivial monodromy. 

The perverse schober $\mathcal{F}$ is up to equivalence described by the adjunctions
\[
\left( \varrho_i[-i]\colon \on{Fun}(\Delta^{m-1},\N) \longleftrightarrow \N \noloc \varsigma_i[i]\right)_{1\leq i\leq m}\,.
\] 
The transport functors along paths wrapping one step clockwise around $v$ from the $(i+1)$-th edge of $\rgraph_m$ to the $i$-th edge are given by 
\[ \varrho_{i}[-i]\circ \varsigma_{i+1}[i+1]\simeq [1]\,,\]
for $1\leq i\leq m-1$, as follows from the adjunction \eqref{eq:rhosigmaadj}. For $i=m$, the transport is given by
\[ \varrho_m[-m]\circ \varsigma_1[1]\simeq [-1]\,,\]
as follows from the adjunction \eqref{eq:rhosigmaadj2} using that the cotwist functor of the adjunction $0\leftrightarrow \N$ is given by $T_\mathcal{N}\simeq [-1]$.  The transport of an embedded full loop enclosing $v$ going in the clockwise direction is thus given by $[m-2]$. 
\end{example}

\begin{construction}\label{constr:transport2}
Let $\rgraph$ be a ribbon graph and consider a curve $\eta\colon[0,1]\to \Sigma_{\rgraph}\backslash \rgraph_0$ going from $B_{e_0}$ to $B_{e_1}$ for some $e_0,e_1\in \rgraph_1$. 
We can write $\eta$ uniquely as the composite of a minimal number of curves $\delta_1,\dots,\delta_m$ with endpoints in ${\bf B}$, such that each $\delta_i$ is contained in $\Sigma_{v_i}\subset \Sigma_{\rgraph}$ for some vertex $v_i\in \rgraph_0$. Concretely, the paths $\delta_i$ near a given vertex $v$ are obtained as the components of the intersection of $\gamma$ with $\Sigma_v\subset \Sigma_{\rgraph}$.

Let $\mathcal{F}$ be an $R$-linear $\rgraph$-parametrized perverse schober. We define the transport $\mathcal{F}^{\rightarrow}(\eta)$ of $\mathcal{F}$ along $\eta$ as the $R$-linear equivalence
\[ 
\mathcal{F}^{\rightarrow}(\delta_m)\circ \dots \circ \mathcal{F}^{\rightarrow}(\delta_1)\colon \mathcal{F}(e_0)\longrightarrow \mathcal{F}(e_1) \,.
\]
\end{construction}

The following lemma collects some properties of the transport functors.

\begin{lemma}\label{lem:transport}
Let $\mathcal{F}$ be an $R$-linear $\rgraph$-parametrized perverse schober.
\begin{enumerate}[(1)]
\item Consider two curves $\eta,\eta'\colon [0,1]\to \Sigma_{\rgraph}\backslash \rgraph_0$ with $\eta_1(0)=\eta_2(0)\in B_{e_0},\eta_1(1)=\eta_2(1)\in B_{e_1}$, for $e_0,e_1\in \rgraph_1$. If $\eta,\eta'$ are homotopic in $\Sigma_{\rgraph}\backslash \rgraph_0$ with the homotopy fixing endpoints, then 
\[ \mathcal{F}^\rightarrow(\eta_1)\simeq \mathcal{F}^\rightarrow(\eta_2)\,.\]
\item Let $\eta,\eta'\colon [0,1]\to \Sigma_{\rgraph}\backslash \rgraph_0$ be curves with endpoints in ${\bf B}$ and such that $\eta(1)=\eta'(0)$. Denote their composite by $\eta'\ast \eta$. Then 
\begin{equation}\label{eq:comptransport} \mathcal{F}^\rightarrow(\eta'\ast \eta)\simeq \mathcal{F}^\rightarrow(\eta')\circ \mathcal{F}^\rightarrow(\eta)\,.\end{equation}
\item Let $c\colon \rgraph\to \rgraph'$ be a contraction of ribbon graphs contracting a set $E\subset \rgraph_1$ of edges, each not connecting two singularities of $\mathcal{F}$. Let $\eta\colon [0,1]\to \Sigma_{\rgraph}\backslash \rgraph_0$ be a curve with endpoints in ${\bf B}\backslash \cup_{e\in E} B_e$ not intersecting any edges in $E$. Choose a smooth map $C\colon \Sigma_{\rgraph}\to \Sigma_{\rgraph'}$ that realizes the contraction $|\rgraph|\to |\rgraph'|$ by contracting small neighborhoods of the edges in $E$ and that restricts on the complement of these neighborhoods to a diffeomorphism. Then $C\circ \eta$ is a curve in $\Sigma_{\rgraph'}\backslash \rgraph_0'$ and
\[ \mathcal{F}^\rightarrow(\eta)\simeq c_*(\mathcal{F})^\rightarrow(C\circ \eta)\,.\] 
\end{enumerate}
\end{lemma}

\begin{proof}
Parts (1) and (2) are immediate from the definition of transport. For part (3), it suffices to consider the case that $c$ contracts a single edge $e$, since any contraction is a finite composition of such contractions. As $\mathcal{F}$ and $c_*(\mathcal{F})$ are identical away from a neighborhood of $e$, it suffices to show that the transports near $e$ of $\mathcal{F}$ and $c_*(\mathcal{F})$ coincide. Using the local model for $\mathcal{F}$ from \Cref{prop:localmodel} and the local model for $c_*(\mathcal{F})$ described in \cite[Lemma 4.26]{Chr21b}, this is straightforward to verify.
\end{proof} 

For the following, we fix a ribbon graph $\rgraph$ and we choose once and for all an edge $e\in \rgraph_1$. Let $x\in e\cap B_e$ be the intersection point.

\begin{lemma}\label{lem:locsystem}
Let $\mathcal{F}$ be an $R$-linear $\rgraph$-parametrized perverse schober. The assignment $\gamma\mapsto \mathcal{F}^\rightarrow(\gamma)$ on based loops at $x$ defines a group homomorphism
\[
\mathcal{F}^{\rightarrow} \colon \pi_1(\Sigma_{\rgraph}\backslash \rgraph_0,x)\longrightarrow \pi_0\on{Aut}(\mathcal{F}(e))\,,
\]
i.e.~a local system on $\Sigma_{\rgraph}\backslash \rgraph_0$, with values in the group of equivalence classes of $R$-linear autoequivalences of $\mathcal{F}(e)$.
\end{lemma}

\begin{proof}
This assignment is well defined by part (1) of \Cref{lem:transport} and a group homomorphism by part (2) of \Cref{lem:transport}.
\end{proof}

The main issue with the local system from \Cref{lem:locsystem} is that it does not extend to $\Sigma_{\rgraph}\backslash P$, where $P$ denotes the set of singularities of $\mathcal{F}$, see \Cref{ex:transport}. This is general not even true on $K_0$.

Let $x'\in \mathbb{P}T(\Sigma_{\rgraph}\backslash \rgraph_0)$ be an inverse image of $x$ under the bundle projection $\mathbb{P}T(\Sigma_{\rgraph}\backslash \rgraph_0) \to \Sigma_{\rgraph}\backslash \rgraph_0$. Consider the short exact sequence of groups:
\[
1\to \pi_1(\mathbb{P}T_x\Sigma_{\rgraph}\backslash \rgraph_0,x') \to \pi_1(\mathbb{P}T\Sigma_{\rgraph}\backslash \rgraph_0,x') \to \pi_1(\Sigma_{\rgraph}\backslash \rgraph_0,x)\to 1\,,
\]
where 
$\pi_1(\mathbb{P}T_x\Sigma_{\rgraph}\backslash \rgraph_0)\simeq \mathbb{Z}$ (with $1$ corresponding to the counterclockwise half-turn). The line field $\nu_{\rgraph}$ from \Cref{ex:linefield} defines a splitting
\[ \pi_1(\mathbb{P}T(\Sigma_{\rgraph}\backslash \rgraph_0))\simeq \pi_1(\Sigma_{\rgraph}\backslash \rgraph_0,x)\times \pi_1(\mathbb{P}T_x(\Sigma_{\rgraph}\backslash \rgraph_0))\simeq \pi_1(\Sigma_{\rgraph}\backslash \rgraph_0,x)\times \mathbb{Z}\,.\]

\begin{definition}
Let $\mathcal{F}$ be an $R$-linear $\rgraph$-parametrized perverse schober.
\begin{enumerate}[(1)]
\item We define the local system $\mathcal{L}_{\mathbb{P}T}\mathcal{F}$ on $\mathbb{P}T(\Sigma_{\rgraph}\backslash \rgraph_0)$ by the assignment
\[
\pi_1(\mathbb{P}T(\Sigma_{\rgraph}\backslash \rgraph_0))\simeq \pi_1(\Sigma_{\rgraph}\backslash \rgraph_0,x)\times \mathbb{Z}\longrightarrow \pi_0\on{Aut}(\mathcal{F}(e))\,,~(\gamma,i)\mapsto \mathcal{F}^\rightarrow(\gamma)[i]\,.
\]
\item We define the local system $\mathcal{L}\mathcal{F}$ on $\on{Fr}(\Sigma_{\rgraph}\backslash \rgraph_0)$ as the pullback of $\mathcal{L}_{\mathbb{P}T}\mathcal{F}$ along the map $\on{Fr}(\Sigma_{\rgraph}\backslash \rgraph_0)\to \mathbb{P}T(\Sigma_{\rgraph}\backslash \rgraph_0)$. 
\end{enumerate}
\end{definition}

\begin{remark}
Let $v$ be a vertex of $\rgraph$ of valency $m$. Near $v$, the local system $\mathcal{L}\mathcal{F}$ can be described as follows. Since $\Sigma_v$ is diffeomorphic to a subset of $\mathbb{R}^2$, there exist canonical splittings $\mathbb{P}T\Sigma_v\simeq \Sigma_v\times S^1$ and $\on{Fr}(\Sigma_v)\simeq \Sigma_v\times \on{GL}(2,\mathbb{R})$, which restrict to splittings 
\[ \mathbb{P}T(\Sigma_v\backslash \{v\}) \simeq (\Sigma_v\backslash \{v\})\times S^1\] and 
\[ \on{Fr}(\Sigma_v\backslash \{v\})\simeq \Sigma_v\backslash \{v\}\times \on{GL}(2,\mathbb{R})\,.\]
Let $\gamma\colon S^1\to \Sigma_v\backslash \{v\}$ be an embedded loop mapping the basepoint to $\partial \Sigma_v$ and wrapping one time clockwise around $v$. Then $\gamma$ defines loops $\gamma'$ in $(\Sigma_v\backslash \{v\})\times S^1$ and $\gamma''$ in $\Sigma_v\backslash \{v\}\times \on{GL}(2,\mathbb{R})$ which are both constant in the second component. We have
\[
\mathcal{L}\mathcal{F}(\gamma'')\simeq \mathcal{L}_{\mathbb{P}T}\mathcal{F}(\gamma')\simeq \mathcal{F}^\rightarrow(\gamma)[W(\gamma)]=\mathcal{F}^\rightarrow(\gamma)[-m]\,.
\]
By \Cref{ex:transport}, if $v$ is not a singularity of $\mathcal{F}$, then 
\begin{equation}\label{eq:monodromy-2}
\mathcal{L}\mathcal{F}(\gamma'')\simeq [-2]\,.
\end{equation}
\end{remark}

\begin{proposition}\label{prop:monodromy}
Let $\mathcal{F}$ be an $R$-linear $\rgraph$-parametrized perverse schober with singularities at $P$ and let $\xi$ be a framing on $\Sigma_{\rgraph}\backslash P$. The pullback local system $\xi^*\mathcal{L}\mathcal{F}$ then extends to a local system
\[
\xi^*\mathcal{L}\mathcal{F}\colon \pi_1(\Sigma_{\rgraph}\backslash P,x)\longrightarrow \pi_0\on{Aut}(\mathcal{F}(e))\,.
\]
We call this local system the monodromy of $\mathcal{F}$ with respect to the framing $\xi$.
\end{proposition}

\begin{proof}
Consider the restriction $\xi|_{\Sigma_{\rgraph}\backslash \rgraph_0}$. Then the winding number of a clockwise embedded loop wrapping one time around a non-singular vertex $v\in \rgraph_0\backslash P$ is given by $-2$. Comparing with \eqref{eq:monodromy-2}, we see that the monodromy of $\xi|_{\Sigma_{\rgraph}\backslash \rgraph_0}^*\mathcal{L}\mathcal{F}$ along this loop is trivial. It follows that $\xi|_{\Sigma_{\rgraph}\backslash \rgraph_0}^*\mathcal{L}\mathcal{F}$ extends to the desired local system on $\Sigma_{\rgraph}\backslash P$.
\end{proof}

\begin{remark}
The above constructions can be seen as implementing the following observation: the group $H^1(\Sigma_{\rgraph}\backslash \rgraph_0,\mathbb{Z})$ acts on the $\infty$-category of local systems on $\Sigma_{\rgraph}\backslash \rgraph_0$. The line field $\nu_{\rgraph}$ from \Cref{ex:linefield} gives a base point of the $H^1(\Sigma_{\rgraph}\backslash \rgraph_0,\mathbb{Z})$-torsor of homotopy classes of line fields. Thus the line field arising from a framing $\xi$ acts on the local system $\mathcal{F}^\rightarrow$ to produce a new local system, given by $\xi^*\mathcal{L}\mathcal{F}$. Contrary to $\mathcal{F}^\rightarrow$, the local system $\xi^*\mathcal{L}\mathcal{F}$ extends to $\Sigma_{\rgraph}\backslash P$.
\end{remark}

\begin{example}
Consider a perverse schober $\mathcal{F}$ on the disc $D$ with one boundary marked point, parametrized by the $1$-spider $\rgraph_1$ with vertex $v$. Thus, $\mathcal{F}$ consists of a spherical functor $F\colon \V \to \N$. Let $G$ be the right adjoint of $F$ and $T_{\mathcal{N}}$ the cotwist functor of $F\dashv G$. Choosing any framing on $D$ and restricting it to $D\backslash\{v\}$, the corresponding clockwise monodromy of $\mathcal{F}$ around $v$ is given by $T_{\mathcal{N}}[1]$. Passing to Grothendieck groups, i.e.~applying $K_0(\mhyphen)$, we obtain a perverse sheaf on $D$ with (at most) one single singularity. We have 
\[ K_0(T_{\mathcal{N}})=K_0(F)K_0(G)-K_0(\on{id}_\mathcal{N})\,.\] 
The automorphism of $K_0(\mathcal{N})$ 
\[ K_0(T_{\mathcal{N}}[1])=-K_0(T_{\mathcal{N}})=K_0(\on{id}_\mathcal{N})-K_0(F)K_0(G)\]
describes the usual monodromy of this perverse sheaf. 
\end{example}

\begin{remark}\label{rem:spinstructure}
Let $\mathcal{F}$ be an $R$-linear $\rgraph$-parametrized perverse schober and suppose that $\mathcal{F}(e)$ is $2n$-periodic for some $n\geq 1$, in the sense that $[2n]\simeq \on{id}_{\mathcal{F}(e)}$. Let $x''\in \on{Fr}(\Sigma_{\rgraph}\backslash \rgraph_0)$ be an inverse image of $x\in \Sigma_{\rgraph}\backslash \rgraph_0$. The fiberwise monodromy of $\mathcal{L}\mathcal{F}$, i.e.~the action of  $1\in \mathbb{Z}\simeq \pi_1(\on{Fr}_x(\Sigma_{\rgraph}\backslash \rgraph_0),x'')$ is given by $[2]$. 

Suppose first that $n=1$, i.e.~$\mathcal{F}(e)$ is $2$-periodic. Then the local system 
\[ \mathcal{L}\mathcal{F}\colon \pi_1(\on{Fr}(\Sigma_{\rgraph}\backslash \rgraph_0),x'')\longrightarrow \pi_0\on{Aut}(\mathcal{F}(e))
\] factors through the quotient 
\[
\pi_1(\on{Fr}(\Sigma_{\rgraph}\backslash \rgraph_0),x'')/\pi_1(\on{Fr}_x(\Sigma_{\rgraph}\backslash \rgraph_0),x'')\simeq \pi_1(\Sigma_{\rgraph}\backslash \rgraph_0,x)
\]
and thus canonically defines a local system on $\Sigma_{\rgraph}\backslash \rgraph_0$ (that even extends to $\Sigma_\rgraph\backslash P$), without a choice of framing.

Now suppose that $n>1$. An $n$-spin structure on $\Sigma_{\rgraph}\backslash \rgraph_0$ amounts to an $n$-fold connected covering of $\on{Fr}(\Sigma_{\rgraph}\backslash \rgraph_0)$. Pulling back $\mathcal{L}\mathcal{F}$ along this covering yields a local system with trivial monodromy on the fiber over $\Sigma_\rgraph\backslash \rgraph_0$ that hence restricts as in the case $n=1$ to a local system on $\Sigma_\rgraph\backslash P$.
\end{remark}

\begin{remark}\label{rem:monodromyHH}
Let $\mathcal{F}$ be an $R$-linear $\rgraph$-parametrized perverse schober with singularities at $P\subset \rgraph_0$. Applying any additive invariant $E$, such as Hochschild homology, to the local system $\mathcal{L}\mathcal{F}$ on $\on{Fr}(\Sigma_{\rgraph}\backslash \rgraph_0)$ yields a local system with trivial monodromy on the fiber, since $E([2])\simeq \on{id}$. By the same argument as in \Cref{rem:spinstructure}, $E(\mathcal{L}\mathcal{F})$ thus defines a local system on $\Sigma_{\rgraph}\backslash P$.  
\end{remark}

The monodromy of a perverse schober is independent on the choice of ribbon graph in the appropriate sense:

\begin{lemma}\label{lem:invcontraction}
Let $\mathcal{F}$ be an $R$-linear $\rgraph$-parametrized perverse schober with singularities at $P$. Let $c\colon \rgraph \to \rgraph'$ be a contraction of ribbon graphs not contracting the edge $e$ and no edges connecting any two vertices in $P$ and choose a smooth map $C\colon \Sigma_{\rgraph}\to \Sigma_{\rgraph'}$ as in \Cref{lem:transport}. Let $\xi'$ be a framing on $\Sigma_{\rgraph'}\backslash P$ and let $\xi=C^*(\xi')$ be the corresponding framing  on $\Sigma_{\rgraph}\backslash P$. 
There exists an equivalence between $\xi^*\mathcal{L}\mathcal{F}$ and the local system
\[ 
\pi_1(\Sigma_{\rgraph}\backslash P,x)\xlongrightarrow{\pi_1(C)} \pi_1(\Sigma_{\rgraph'}\backslash P,C(x))\xlongrightarrow{(\xi')^*\mathcal{L}c_*(\mathcal{F})}\pi_0\on{Aut}((c_*\mathcal{F})(e))= \pi_0\on{Aut}(\mathcal{F}(e))\,.
\] 
\end{lemma}

\begin{proof}
This follows from \Cref{lem:transport} and the relation between the line fields $\nu_{\rgraph}$ and $\nu_{\rgraph'}$ described in \Cref{ex:linefield}.
\end{proof}

Next, we prove that a perverse schober without singularities is completely determined by its generic stalk and monodromy.

\begin{proposition}\label{prop:schobersfrommonodromy}
Let $\mathcal{F}_1,\mathcal{F}_2$ be two $R$-linear $\rgraph$-parametrized perverse schobers without singularities. Let $\xi$ be a framing of $\Sigma_{\rgraph}$. Assume that $\mathcal{F}_1(e)=\mathcal{F}_2(e)$. The following two are equivalent:
\begin{enumerate}[1)]
\item There exists an equivalence of $\rgraph$-parametrized perverse schobers $\mathcal{F}_1\simeq \mathcal{F}_2$.
\item The monodromy local systems $\xi^*\mathcal{L}\mathcal{F}_1$ and $\xi^*\mathcal{L}\mathcal{F}_2$ are equivalent.
\end{enumerate}
\end{proposition}

\begin{proof}
It is clear that 1) implies 2). We next show that 2) implies 1). Denote by $\N=\mathcal{F}_1(e)$ the generic stalk. Choose a contraction $c\colon\rgraph\to\rgraph'$, such that $e$ is not contracted and $\rgraph'$ has only a single vertex. Part (3) of \Cref{lem:transport} implies that the local systems of transport of $c_*(\mathcal{F}_1)$ and $c_*(\mathcal{F}_2)$ are equivalent.

We choose a total order of the $m$ halfedges incident to the vertex $v$ of $\rgraph'$, compatible with their given (counterclockwise) cyclic order. We denote the $i$-th halfedge by $a_i$, and its corresponding edge by $e_i$. We can replace $c_*(\mathcal{F}_1)$ and $c_*(\mathcal{F}_2)$ by equivalent $\rgraph'$-parametrized perverse schobers, denoted $\mathcal{G}_1$ and $\mathcal{G}_2$, such that for $j=1,2$ and $1\leq i\leq m$
\[ \mathcal{G}_j(v)=\mathcal{V}^m_{0_{\N}}=\on{Fun}(\Delta^{m-2},\N)\,,\]
\[ \mathcal{G}_j(e_i)=\N\,,\]
and
\[ \mathcal{G}_j(v\xrightarrow{a_i}e_i)= S_{i,j}\circ \varrho_i[-i]\,,\]
where $S_{i,j}$ is some autoequivalence of $\N$. The monodromy relative $\xi$ along a path $\gamma$ starting at $e_i$ and going around a given loop of $\rgraph'$, composed of halfedges $a_i,a_{i'}$ of $\mathcal{G}_j$, is given by $S_{i',j}\circ S_{i,j}^{-1}[i-i'+W(\gamma)]$ if $i'>i$ and by $S_{i',j}\circ S_{i,j}^{-1}[i-i'-m+2+W(\gamma)]$ if $i'<i$. We thus find $S_{i',1}S_{i,1}^{-1}\simeq S_{i',2}S_{i,2}^{-1}$. We can additionally assume that $S_{i',2}=S_{i',1}=\on{id}_{\mathcal{G}(e)}$, by replacing $\mathcal{G}_1,\mathcal{G}_2$ by equivalent perverse schobers once more. We thus conclude $S_{i,1}\simeq S_{i,2}$ as well. Performing this argument for all loops of $\rgraph'$ shows that $\mathcal{G}_1\simeq \mathcal{G}_2$, concluding the proof. 
\end{proof}

The next example illustrates that a perverse schober consists of more data than its singularity data and monodromy data.

\begin{example}\label{ex:monodromy}
Consider the following spanning graph $\rgraph$ of the $1$-gon:
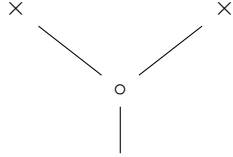
\[
\begin{tikzcd}
\times \arrow[rd, no head] &                                & \times \arrow[ld, no head] \\
                           & \circ                          &                            \\
                           & {} \arrow[u, no head] &                           
\end{tikzcd}
\]
Consider a $\rgraph$-parametrized perverse schober $\mathcal{F}$ with two singularities at the vertices labeled $\times$ and no singularity at the vertices labeled $\circ$. Let $F_1\colon \V_1\to \N$ and $F_2\colon \V_2\to \N$ be the underlying spherical adjunctions at the two vertices. These are determined only up to pre- and postcomposition with equivalences of $\infty$-categories. The perverse schober $\mathcal{F}$ thus corresponds, up to equivalence, to a diagram
\[
\begin{tikzcd}
\mathcal{V}_1 \arrow[d, "S_1\circ F_1"'] &                                                                                                            & \mathcal{V}_2 \arrow[d, "S_2\circ F_2"] \\
\mathcal{N}                              & {\on{Fun}(\Delta^1,\mathcal{N})} \arrow[r, "{\on{ev}_1}"'] \arrow[l, "{\on{ev}_0[1]}"] \arrow[d, "\on{cof}"] & \mathcal{N}                             \\
                                         & \mathcal{N}                                                                                                &                                        
\end{tikzcd}
\]
with $S_1,S_2$ autoequivalences of $\mathcal{N}$. Note that $\on{Fun}(\Delta^1,\mathcal{N})\simeq \mathcal{V}^3_{0_{\mathcal{N}}}$, with $0_{\mathcal{N}}\colon 0 \to \mathcal{N}$ the spherical zero functor. Composing with the autoequivalence $\on{Fun}(\Delta^1,S_2^{-1})$ of $\on{Fun}(\Delta^1,\mathcal{N})$ replaces $S_1$ by $S_2^{-1}S_1$ and $S_2$ by $\on{id}_{\mathcal{N}}$. Up to equivalence of perverse schobers, we may thus assume that $S_2=\on{id}_{\mathcal{N}}$. In total, the equivalence class of the perverse schober $\mathcal{F}$ is determined by the functors $S_1,F_1$ and $F_2$, up to natural equivalence of these functors and precomposition of $F_1,F_2$ with autoequivalences of $\V_1,\V_2$.  

We equip the $1$-gon with any framing (it is unique up to homotopy) and denote its restriction to the complement of the two singular vertices by $\xi$. The monodromy of $\mathcal{F}$ with respect to $\xi$ along a clockwise loop wrapping once around the left singularity is given by $S_1T_{\mathcal{N}}S_1^{-1}[1]$, with $T_{\mathcal{N}}$ the cotwist of $F_1\dashv \on{radj}(F_1)$, which is the suspension of the cotwist functor of the adjunction $S_1 F_1 \dashv \on{radj}(F_1)S_1^{-1}$. The monodromy of the clockwise loop wrapping once around the right singularity is similarly given by the cotwist functor of the adjunction $F_2\dashv \on{radj}(F_2)$. 

Since the choice of $S_1$ affects the monodromy of $\mathcal{F}$ only by conjugation with $S_1$, we find that given the functors $F_1,F_2$ and the monodromy local system, one cannot in general recover the autoequivalence $S_1$, and thus also not recover the equivalence class of the perverse schober. One can make this failure explicit by choosing for instance $F_1=F_2=\phi^*$ the functor from \Cref{subsec:relGinzburg} for $n=3$, where $T_\N=[3]$, and $S_1=\varphi^*$ is the functor from the proof of \Cref{thm:ginzburgcy}.
\end{example}

\section{Calabi--Yau structures and perverse schobers}\label{sec:CYschobers}

We begin with the local picture in \Cref{subsec:locschoberCY} by constructing relative Calabi--Yau structures on the sections of perverse schobers parametrized by the $n$-spider. In \Cref{subsec:schoberCYglobal}, we discuss the construction of relative Calabi--Yau structures on the global sections of perverse schobers.
Finally, in \Cref{subsec:sphericalCY}, we return to the case of perverse schobers parametrized by the $1$-spider, i.e.~spherical adjunctions, and give a novel and easy to verify criterion for the existence of a weak right Calabi--Yau structure.

\subsection{Calabi--Yau structures locally}\label{subsec:locschoberCY}

Consider a perverse schober on the $m$-spider $\mathcal{F}$. \Cref{prop:localmodel} states that $\mathcal{F}$ arises from a spherical adjunction $F\colon \D\leftrightarrow \C\noloc G$, in the sense that $\mathcal{F}$ is given, up to equivalence, by the collection of adjunctions
\[
\left( \varrho_i\colon \mathcal{V}^m_F \longleftrightarrow \mathcal{C} \noloc \varsigma_i\right)_{1\leq i\leq m}\,.
\] 
We can combine these adjunctions into the single adjunction
\[
R=(\varrho_1,\dots,\varrho_m)\colon \mathcal{V}^m_F\longleftrightarrow \mathcal{C}^{\times m}\noloc S=(\varsigma_1,\dots,\varsigma_m)\,.
\]
As it turns out, this adjunction is spherical and its twist functor describes the change in the perverse schober on the $m$-gon when rotating it by an angle of $\frac{2\pi}{m}$, see \Cref{prop:rotationaltwist} below. Note that this is not a genuine symmetry, a full rotation does not return the perverse schober to itself, but instead changes it by the monodromy around the singularity, i.e.~the suspension of the cotwist functor of $F\dashv G$.

\begin{proposition}[{$\!\!$\cite[Propositions~3.8,\ 3.11]{Chr22}}]\label{prop:rotationaltwist}
The adjunction $R\dashv S$ is spherical. Let $T_{\mathcal{V}^m_F}$ be its twist functor and $T_{\mathcal{C}}$ the cotwist functor of the adjunction $F\dashv \on{radj}(F)$. There exist equivalences of functors 
\[
\varrho_i\circ T_{\mathcal{V}^m_F}\simeq \begin{cases} \varrho_{i+1} & \text{for }1\leq i\leq m-1\,, \\
T_{\mathcal{C}}[m-1]\circ\varrho_1& \text{for }i=m\,.
\end{cases}
\]
\end{proposition}

The goal of this section is to prove the following proposition, roughly stating that the functors $R,S$ can be equipped with Calabi--Yau structures, provided that $F,G$ are already equipped with relative Calabi--Yau structures. In this case, the Serre or inverse dualizing functors $\on{id}_{\mathcal{V}^m_F}^*,\on{id}_{\mathcal{V}^m_F}^!$ thus describe the effect of partial rotation of the $m$-gon. 

To simplify the signs, we change $R$ and $S$ in the following proposition by composition with componentwise suspension functors.

\begin{proposition}\label{prop:loccy}
Let $F\colon \D\leftrightarrow \C\noloc G$ be a spherical adjunction of dualizable $R$-linear $\infty$-categories and $m\geq 1$. Consider the spherical adjunction
\begin{equation}\label{eq:varsig} 
\mathrm{R}_F^m\coloneqq (\varrho_1[-1],\varrho_2[-2],\dots,\varrho_m[-m])\colon \mathcal{V}^m_{F}\longleftrightarrow \C^{\times m}\noloc \mathrm{S}_F^m\coloneqq (\varsigma_1[1],\varsigma_2[2],\dots,\varsigma_m[m])\,.
\end{equation}
\begin{enumerate}[(1)]
\item Suppose that $\mathcal{C},\mathcal{D}$ are smooth and that there exists a class $\eta_{G}\colon R[n]\to \on{HH}(\D,\C)^{S^1}$ defining a left $n$-Calabi--Yau structure on $G$, which restricts to a left $(n-1)$--Calabi--Yau structure $\eta_{\D}\colon R[n-1]\to \HH(\C)^{S^1}$ on $\C$. Then the functor $\mathrm{S}_F^m$ admits a canonical left $n$-Calabi--Yau structure, which restricts on $\C^{\times m}$ to
\[ \eta_{\C}^{\times m}\colon R[n-1]\longrightarrow \HH(\C^{\times m})^{S^1}\simeq \bigoplus_m \HH(\C)^{S^1}\,.\]
\item Suppose that $\C,\D$ are proper and that there exists a class $\eta_F\colon R[n]\to \HH(\D,\C)^*_{S^1}$ defining a right $n$-Calabi--Yau structure on $F$, which restricts to a right $(n-1)$--Calabi--Yau structure $\eta_{\C}\colon R[n-1]\to \HH(\C)^*_{S^1}$ on $\C$. Then the functor $\mathrm{R}_F^m$ admits a canonical right $n$-Calabi--Yau structure, which restricts on $\C^{\times m}$ to 
\[ \eta_{\C}^{\times m}\colon R[n-1]\longrightarrow \HH(\C^{\times m})^*_{S^1}\simeq \bigoplus_m \HH(\C)^*_{S^1}\,.\]
\end{enumerate}
\end{proposition}

Using the gluing properties of Calabi--Yau structures, we will reduce the proof of \Cref{prop:loccy} to the case $m=3$ and $\D=0$. This case is then first solved for $\C=\on{RMod}_R$, which admits canonical left and right $0$-Calabi--Yau structures, and then for arbitrary $\C$ by tensoring, see \Cref{lem:cytp}.

\begin{construction}\label{constr:U}
Let $\C=\on{RMod}_R$, considered as an $R$-linear smooth and proper $\infty$-category. We construct two inverse equivalences $U,U^{-1}\colon \on{Fun}(\Delta^1,\C)\rightarrow  \on{Fun}(\Delta^1,\C)$ via Kan extension. 

Consider the $\infty$-category $\mathcal{X}$ of diagrams in $\C$ of the following form, where all squares are biCartesian, i.e.~both pushout and pullback squares.
\[
\begin{tikzcd}
a' \arrow[r] \arrow[d] \arrow[rd, "\square", phantom] & b' \arrow[d] \arrow[r] \arrow[rd, "\square", phantom] & 0 \arrow[d]                                          &             \\
0 \arrow[r]                                           & a \arrow[d] \arrow[r] \arrow[rd, "\square", phantom]  & b \arrow[d] \arrow[r] \arrow[rd, "\square", phantom] & 0 \arrow[d] \\
                                                      & 0 \arrow[r]                                           & a'' \arrow[r]                                        & b''        
\end{tikzcd}
\]
One can formally characterize the $\infty$-category $\mathcal{X}$ as consisting of diagrams which are repeated Kan extensions of their restriction to $a\rightarrow b$. The restriction functor to $a\rightarrow b$ thus defines by \cite[4.3.2.15]{HTT} a trivial fibration $\phi\colon \mathcal{X}\rightarrow \on{Fun}(\Delta^1,\C)$. It hence admits an inverse, unique up to contractible space of choices. We denote by $\psi'\colon \mathcal{X}\rightarrow \on{Fun}(\Delta^1,\C)$ the restriction functor to $a'\rightarrow b'$ and by $\psi''\colon \mathcal{X}\rightarrow \on{Fun}(\Delta^1,\C)$ the restriction functor to $a''\rightarrow b''$. 

The functor $U$ is defined as the suspension of the composite functor 
\[ \on{Fun}(\Delta^1,\C)\xrightarrow{\phi^{-1}}\mathcal{X}\xrightarrow{\psi'}\on{Fun}(\Delta^1,\C)\,,
\]
and the functor $U^{-1}$ is defined as the looping of the composite functor 
\[\on{Fun}(\Delta^1,\C)\xrightarrow{\phi^{-1}}\mathcal{X}\xrightarrow{\psi''}\on{Fun}(\Delta^1,\C)\,.
\]
\end{construction}

\begin{remark}\label{rem:Serreexplicit}
An alternative description of the functor $U$ from \Cref{constr:U} is as follows. Let again $\C=\on{RMod}_R$. We have functors $\on{ev}_1,\on{cof}\colon \on{Fun}(\Delta^1,\C)\to \C$, together with a natural transformation $\on{ev}_1\to \on{cof}$, as well as the further fully faithful functors $\iota_0,\iota_1\colon \C\rightarrow \on{Fun}(\Delta^1,\C)$, given informally by $\iota_0\colon c\mapsto (c\rightarrow 0)$ and $\iota_2\colon c\mapsto (0\rightarrow c)$, together with a natural transformation $\iota_1[-1]\to \iota_2$. Composing these functors, we have an induced natural transformation $\iota_1\on{ev}_1[-1]\to \iota_2\on{cof}$, whose cofiber describes an endofunctor of $\on{Fun}(\Delta^1,\C)$, given by the assignment
\[\big(a\longrightarrow b\big) \quad \mapsto \quad\big(b\longrightarrow \on{cof}(a\rightarrow b)\big)\,.
\]
This functor is equivalent to the endofunctor $U$, as follows from the universal property of the $R$-linear $\infty$-category $\on{Fun}(\Delta^1,\C)$, as the lax limit of the functor $\on{id}_{\C}\colon \C\to \C$, considered as a $\Delta^1$-indexed diagram in the $(\infty,2)$-category of $R$-linear $\infty$-categories.
\end{remark}

\begin{lemma}\label{lem:serreofDelta1C}
Let $\C=\on{RMod}_R$. The $R$-linear $\infty$-category $\on{Fun}(\Delta^1,\C)$ is smooth and proper and the functor $U$ from \Cref{constr:U} is a Serre functor of $\on{Fun}(\Delta^1,\C)$.
\end{lemma}

\begin{proof}
It is straightforward to check that $\on{Fun}(\Delta^1,\C)$ is smooth and proper. 
Any object $a\shortrightarrow b$ in $\on{Fun}(\Delta^1,\C)$ is given as the cofiber of a map $(a[-1]\shortrightarrow 0)\rightarrow (0\shortrightarrow b)$. By the exactness of $\on{Mor}_{\on{Fun}(\Delta^1,\C)}(\mhyphen,\mhyphen)$ in the second entry, it follows that 
\begin{equation}\label{eq:Morviacof} \on{Mor}_{\on{Fun}(\Delta^1,\C)}(a\shortrightarrow b,a'\shortrightarrow b')\simeq \on{fib}\big(\on{Mor}_{\C}(b,b')\rightarrow {\on{Mor}_{\C}(a,\on{cof}(a'\shortrightarrow b'))}\big)\,.\end{equation}
Let $a\rightarrow b,a'\rightarrow b'\in \on{Fun}(\Delta^1,\C)$ be compact objects. Using that $\on{id}_{\C}$ is an $R$-linear Serre functor, we find equivalences in $\on{RMod}_R$
\begin{align*}
\on{Mor}_{\on{Fun}(\Delta^1,\C)}(a\shortrightarrow b,a'\shortrightarrow b')& \simeq \on{fib}\big(\on{Mor}_{\C}(b,b')\rightarrow {\on{Mor}_{\C}(a,\on{cof}(a'\shortrightarrow b'))}\big)\\
&\simeq \on{cof}\big({\on{Mor}_{\C}(a,\on{cof}(a'\shortrightarrow b'))^\ast \rightarrow \on{Mor}_{\C}(b,b')^\ast}\big)^\ast\\
&\simeq \on{cof}\big({\on{Mor}_{\C}(\on{cof}(a'\shortrightarrow b'),a) \rightarrow \on{Mor}_{\C}(b',b)}\big)^\ast\\
&\simeq \on{Mor}_{\on{Fun}(\Delta^1,\C)}(b'\shortrightarrow \on{cof}(a'\shortrightarrow b'),\on{cof}(a\rightarrow b)\shortrightarrow a[1])^\ast\\
&\simeq \on{Mor}_{\on{Fun}(\Delta^1,\C)}(a'\rightarrow b', U(a\rightarrow b))^\ast\,,
\end{align*}
which are functorial in $a\shortrightarrow b \in \on{Fun}(\Delta^1,\C)^{\on{c},\on{op}}$, $a'\shortrightarrow b' \in \on{Fun}(\Delta^1,\C)^{\on{c}}$. 
The second to last equivalence arises in the same way as the equivalence \eqref{eq:Morviacof}. The last equivalence uses that the sequences $a'\rightarrow b'\rightarrow \on{cof}(a'\rightarrow b')$ and $b\rightarrow \on{cof}(a\rightarrow b)\rightarrow a[1]$ are fiber and cofiber sequences and \Cref{rem:Serreexplicit}. This shows that $U$ is a Serre functor, which concludes the proof.
\end{proof}

\begin{lemma}\label{lem:1cy}
Let $\C=\on{RMod}_R$. We let $\eta_{\C}\colon R\rightarrow \HH(\C)^{S^1}$ and $\eta_{\C}^*\colon R\to \HH(\D)^*_{S^1}$ denote the lifts\footnote{Canonical lifts  exist by \Cref{rem:traceofunit}.} of the apparent classes in $\HH(\C)\simeq R$ and $\HH(\C)^*\simeq R$. Note that $\eta_\C$ and $\eta_\C^*$ describe left and right $0$-Calabi--Yau structures on $\C$, respectively. 
\begin{enumerate}[(1)]
\item The $R$-linear functor 
\begin{equation*}\label{eq:funsigma123}S\coloneqq (\varsigma_1[1],\varsigma_2[2],\varsigma_3[3])\colon \C^{\times 3}\longrightarrow \on{Fun}(\Delta^1,\C)\end{equation*}
admits a unique left $1$-Calabi--Yau structure which restricts to the left $0$-Calabi--Yau structure $\eta_{\C}^{\times 3}$ on $\C^{\times 3}$.
\item The $R$-linear functor 
\begin{equation*}\label{eq:funrho123}(\varrho_1[-1],\varrho_2[-2],\varrho_3[-3])\colon \on{Fun}(\Delta^1,\C)\longrightarrow \C^{\times 3}\end{equation*}
admits a unique right $1$-Calabi--Yau structure which restricts to the right $0$-Calabi--Yau structure $(\eta_{\C}^*)^{\times 3}$ on $\C^{\times 3}$.
\end{enumerate}
\end{lemma}

\begin{proof}
We only prove part (1), part (2) can be proven similarly. The split localization sequence 
\[ \C\xrightarrow{\varsigma_2[2]}\on{Fun}(\Delta^1,\C)\xrightarrow{\varrho_3[-3]}\C\] 
provides us with a splitting $\HH(\on{Fun}(\Delta^1,\C))^{S^1}\simeq \HH(\C)^{S^1}\oplus \HH(\C)^{S^1}$. Using the adjunctions $\varrho_3[-3]\dashv \varsigma_3[3]$ and $\varsigma_2[2]\dashv \varrho_1[-2]$, we find that $\on{HH}(\varrho_3[-3])^{S^1}$ and $\on{HH}(\varrho_1[-2])^{S^1}$ are the two projection maps $(\HH(\C)^{S^1})^{\oplus 2}\to \HH(\C)^{S^1}$ and $\HH(\varsigma_3[3])^{S^1},\HH(\varsigma_2[2])^{S^1}$ are the two inclusion maps $\HH(\C)^{S^1}\to (\HH(\C)^{S^1})^{\oplus 2}$ of the direct sum decomposition. 

With the above, we have 
\[ \on{HH}\left(\on{Fun}(\Delta^1,\C),\C^{\times 3}\right)^{S^1}\simeq \on{cof}((\HH(\C)^{S^1})^{\oplus 3}\xrightarrow{\mathscr{M}} (\HH(\C)^{S^1})^{\oplus 2})\simeq \HH(\C)^{S^1}[1]\,,\]
where $\mathscr{M}=\begin{pmatrix}-1 & 1 & 0  \\ -1&  0 & 1  \end{pmatrix}$. The formula for $\mathscr{M}$ follows from 
\[ \varrho_1[-2]\circ \varsigma_1[1]\simeq [-1]\] 
and 
\[ \varrho_3[-3]\circ \varsigma_1[1]\simeq [-1]\,.\]

We let $\eta$ be the the class $R[1]\xrightarrow{\eta_\C} \HH(\C)^{S^1}[1] \simeq \HH\left(\on{Fun}(\Delta^1,\C),\C^{\times 3}\right)^{S^1}$.
The observation that $\mathscr{M}(1,1,1)=0$ implies that $\eta$ indeed restricts to $\eta_{\C}^{\times 3}$ on $\C^{\times 3}$. Furthermore, $\eta$ is clearly unique with this property.

Let $\sigma_\C=1\in R \simeq \HH(\C)$ be the Hochschild class underlying $\eta_\C$ and $\sigma=(\sigma_\C,\sigma_\C,\sigma_\C)$ the Hochschild class underlying $\eta$.  To complete the proof, it remains to show that $\sigma$ is non-degenerate.  The class $\sigma$ determines the diagram
\[
\begin{tikzcd}
{\on{id}_{\on{Fun}(\Delta^1,\C)}^!} \arrow[r, "\tilde{\unit}"] & {S_!(\on{id}_{\C^{\times 3}}^!)} \arrow[d, "S_!(\sigma_{\C}^{\times 3})"] &                                           \\
                                                                & {S_!(\on{id}_{\C^{\times 3}})} \arrow[r, "\on{cu}"]                       & {\on{id}_{\on{Fun}(\Delta^1,\C)}}
\end{tikzcd}
\]
together with a null-homotopy of the composite functor $\on{id}_{\on{Fun}(\Delta^1,\C)}^!\to \on{id}_{\on{Fun}(\Delta^1,\C)}$. Composing the first two morphisms in the above diagram, we obtain the sequence 
\begin{equation}\label{eq:D!SDD}
\on{id}_{\on{Fun}(\Delta^1,\C)}^!\longrightarrow S_!(\on{id}_{\C^{\times 3}})\xlongrightarrow{\on{cu}} \on{id}_{\on{Fun}(\Delta^1,\C)}\,.
\end{equation}
The morphism $\tilde{\unit}$ is by \Cref{lem:unitisunit} equivalent to the composite of the unit of the adjunction $S^L\dashv S$ with $\on{id}^!_{\on{Fun}(\Delta^1,\C)}$. By \Cref{lem:Serreunique,lem:serreinv,{lem:serreofDelta1C}}, there exists an equivalence between $\on{id}_{\on{Fun}(\Delta^1,\C)}^!$ and the functor $U^{-1}$ from \Cref{constr:U}. 

One can check that the functor $T$ is furthermore equivalent to the cotwist functor of the adjunction $S\dashv S^R$. There thus exists a fiber and cofiber sequence 
\begin{equation}\label{eq:TSSid}
T\xlongrightarrow{\unit'} SS^R\xlongrightarrow{\counit} \on{id}_{\on{Fun}(\Delta^1,\C)}\,,
\end{equation}
where $\unit'$ is up to equivalence a unit of the adjunction $S^L\dashv S$ composed with $T$, see \cite{DKSS21} or \cite[Remark 2.10]{Chr20}. The respective counit maps in \eqref{eq:D!SDD} and \eqref{eq:TSSid} describe the same counit map. The respective unit maps are also equivalent, up to composition with an autoequivalence of $SS^R$. To show that this autoequivalence may be chosen trivially, we inspect the $R$-module of all possible autoequivalences of $SS^R= S_!(\on{id}_{\C^{\times 3}})$. We have 
\[ \on{Map}(S_!(\on{id}_{\C^{\times 3}}),S_!(\on{id}_{\C^{\times 3}}))\simeq  \on{Map}(\on{ladj}(S_!)S_!(\on{id}_{\mathcal{D}^{\times 3}}),\on{id}_{\C^{\times 3}})\,,\]
where $\on{ladj}(S_!)=S^L(\mhyphen)S^{RR}$, with $S^{RR}$ the right adjoint of $S^R$. Since $S^R\simeq S^L\circ T$, we have $S^{RR}\simeq T^{-1}\circ S$ and thus 
\[\on{ladj}(S_!)S_!(\on{id}_{\mathcal{D}^{\times 3}})\simeq S^LSS^RS^{RR}\simeq S^LSS^LS\,.\]

The functor $S^LS$ splits as 
\[ S^LS\simeq \on{id}_{\C^{\times 3}}\oplus P\,,\]
where $P$ is the twist functor of the adjunction $S^L\dashv S$. It acts via rotation, meaning that $P$ sends the $i$-th component of the direct sum to the $(i-1)$-th component of the direct sum for all $i\in \mathbb{Z}/3\mathbb{Z}$ and then acts with some suspensions on the three components. There are no non-zero natural transformations between $\on{id}_{\C^{\times 3}}$ and $P$ or $P^2$. 
It follows that the morphism
\[ \on{RMod}_R^{\oplus 3}\simeq \on{Map}_{\on{Lin}_R(\C^{\times 3},\C^{\times 3})}(\on{id}_{\C^{\times 3}},\on{id}_{\C^{\times 3}})\xrightarrow{S_!} \on{Map}_{\on{Lin}_R(\on{Fun}(\Delta^1,\C),\on{Fun}(\Delta^1,\C))}(SS^R,SS^R)\]
is an equivalence. We thus find that every possible autoequivalence of $SS^R$ can be accommodated for by choosing a different Hochschild class of $\C^{\times 3}$. We may hence conclude from the existence of the cofiber sequence \eqref{eq:TSSid} that there exists some choice of Hochschild class $\sigma'\in R\simeq \on{HH}\left(\on{Fun}(\Delta^1,\C),\C^{\times 3}\right)^{S^1}[-1]$ which restricts to the class $(\sigma')^{\times 3}\in R^{\oplus 3}\simeq \HH(\C^{\times 3})$, which turns \eqref{eq:D!SDD} into a cofiber sequence. Since $\on{ladj}(S_!)S_!$ contains the identity as a direct summand, we find that $S_!$ is a conservative functor. The fact that $\sigma'$ induces an equivalence $S_!(\on{id}^!_{\C^{\times 3}})\simeq S_!(\on{id}_{\C^{\times 3}})$ thus implies that $(\sigma')^{\times 3}\colon \on{id}^!_{\C^{\times 3}}\to \on{id}_{\C^{\times 3}}$ is already an equivalence. It follows that $\sigma'\in \pi_0(R)$ must be an invertible element. Upon composing $\sigma'$ with its inverse in the ring $\pi_0(R)$, the cofiber sequence \eqref{eq:D!SDD} clearly remains a cofiber sequence. This shows that $\eta$ already describes a left $1$-Calabi--Yau structure, concluding the proof.
\end{proof}

\begin{proof}[Proof of \Cref{prop:loccy}.]
We only prove part (2), the proof of part (1) is analogous. We suppose that $F\colon \D\to \C$ is a spherical functor with a right $n$-Calabi--Yau structure restricting to a right $(n-1)$-Calabi--Yau structure on $\C$. 

We first assume that $F=0_{\C}\colon 0\rightarrow \C$. As shown in \cite[Lemma 4.26]{Chr22}, there is a pullback diagram in $\on{LinCat}_R$
\[
\begin{tikzcd}
\mathcal{V}^{m}_{0_{\C}} \arrow[r] \arrow[d] \arrow[rd, "\lrcorner", phantom] & \mathcal{V}^{m-1}_{0_{\C}} \arrow[d, "{\varrho_1[-1]}"] \\
\mathcal{V}^{3}_{0_{\C}} \arrow[r, "{\varrho_3[-3]}"]                         & \C                                               
\end{tikzcd}
\]
such that the functor $\varrho_i[-i]\colon \mathcal{V}^m_{0_\C}\to \C$ factors for $i=1,2$ as 
\[ \mathcal{V}^m_{0_\C}\to \mathcal{V}^{3}_{0_{\C}} \xrightarrow{\varrho_i[-i]}  \C\]
and for $i=3,\dots,m$ as
\[ \mathcal{V}^m_{0_\C}\to \mathcal{V}^{m-1}_{0_{\C}} \xrightarrow{\varrho_{i-2}[2-i]}  \D\,.\]
To show that $\mathrm{R}_\C^m$ admits the desired right $n$-Calabi--Yau structure, it thus suffices by \Cref{thm:rightCYglue} to show this in the case $m=3$. This case follows from combining \Cref{lem:1cy} and \Cref{lem:cytp}.

Let now the above functor $F\colon \D \to \C$ again be arbitrary. Again, by \cite[Lemma 4.26]{Chr22}, there exists a pullback diagram in $\on{LinCat}_R$.
\[
\begin{tikzcd}[column sep=huge]
\mathcal{V}^m_{F} \arrow[d] \arrow[r] \arrow[rd, "\lrcorner", phantom] & \D \arrow[d, "{F}"] \\
\mathcal{V}^{m+1}_{0_\C} \arrow[r, "{\varrho_{m+1}[-m-1]}"]                                                      & \C                                                    
\end{tikzcd}
\]
By \Cref{thm:rightCYglue}, the above constructed right $n$-Calabi--Yau structure on $\mathrm{R}_{0_\D}^{m+1}$ glues with the $n$-Calabi--Yau structure on $F$ to the desired right $n$-Calabi--Yau structure on $\mathrm{R}_F^m$. 
\end{proof}

\subsection{Calabi--Yau structures on global sections}\label{subsec:schoberCYglobal}

We begin by recording a direct consequence of the gluing property of Calabi--Yau structures.

\begin{theorem}\label{thm:schobercy}
Let $\mathcal{F}\colon \on{Exit}(\rgraph)\rightarrow \on{LinCat}_R^{\on{dual}}$ be a $\rgraph$-parametrized perverse schober.
\begin{enumerate}[(i)]
\item Suppose that $\mathcal{F}$ takes values in smooth $R$-linear $\infty$-categories. Suppose further that:
\begin{itemize}
\item For each vertex $v$ of $\rgraph$ with incident halfedges $a_1,\dots,a_m$ and corresponding edges $e_1,\dots,e_m$, the functor 
\[\prod_{i=1}^m \on{ladj}(\mathcal{F}(v\xrightarrow{a_i}e_i))\colon \prod_{i=1}^m\mathcal{F}(e_i)\longrightarrow \mathcal{F}(v)\] 
carries a left $n$-Calabi--Yau structure 
\[ \eta_v\colon R[n]\to \HH(\mathcal{F}(v),\prod_{i=1}^m \mathcal{F}(e_i))^{S^1}\,.\]
We denote the restriction of $\eta_v$ along the functor $\on{ladj}(\mathcal{F}(v\xrightarrow{a_i}e_i))$ by
\[ \eta_{e,a_i}\colon R[n-1]\to \HH(\mathcal{F}(e_{i}))^{S^1}\,.\] 
\item For each internal edge $e$ of $\rgraph$ with incident halfedges $a\neq b$, we have $\eta_{e,a}\simeq -\eta_{e,b}$.
\end{itemize} 
Then the $R$-linear $\infty$-category of global sections $\glsec(\rgraph,\mathcal{F})$ is smooth and the functor from equation \eqref{eq:cupfunctor}
\[\partial \mathcal{F}\colon \prod_{e\in \rgraph_1^\partial}\mathcal{F}(e)\longrightarrow \glsec(\rgraph,\mathcal{F})\] 
admits a left $n$-Calabi--Yau structure.
\item  Suppose that $\mathcal{F}$ takes values in proper $R$-linear $\infty$-categories. Suppose further that: 
\begin{itemize}
\item For each vertex $v$ of $\rgraph$ with incident halfedges $a_1,\dots,a_m$ and corresponding edges $e_1,\dots,e_m$, the functor
\[\prod_{i=1}^m \mathcal{F}(v\xrightarrow{a_i}e_i)\colon \mathcal{F}(v)\longrightarrow \prod_{i=1}^m\mathcal{F}(e_i)\] 
carries a right $n$-Calabi--Yau structure
\[ \eta_v\colon R[n]\to \HH(\mathcal{F}(v),\prod_{i=1}^m \mathcal{F}(e_i))^*_{S^1}\,.\]
We denote the restriction of $\eta_v$ along the functor $\mathcal{F}(v\xrightarrow{a_i}e_i)$ by
\[ \eta_{e,a_i}\colon R[n-1]\to \HH(\mathcal{F}(e_{i}))^*_{S^1}\,.\] 
\item For each internal edge $e$ of $\rgraph$ with incident halfedges $a\neq b$, we have $\eta_{e,a}\simeq -\eta_{e,b}$. 
\end{itemize} 
Then the evaluation functor at the external edges 
\[\prod_{e\in \rgraph_1^\partial}\on{ev}_e\colon \cptglsec(\rgraph,\mathcal{F}) \longrightarrow \prod_{e\in \rgraph_1^\partial}\mathcal{F}(e)\]
admits a right $n$-Calabi--Yau structure.
\end{enumerate}
\end{theorem}

\begin{proof}
Part (ii) follows from repeated application of \Cref{thm:rightCYglue}, by using that we can compute the limit over $\on{Exit}(\rgraph)$ via repeated pullbacks. Part i) follows by a similar argument from \Cref{thm:leftCYglue}, by using that $\glsec(\rgraph,\mathcal{F})$ is equivalent to the colimit in $\on{LinCat}_R$ of the left adjoint diagram of $\mathcal{F}$.
\end{proof}

Given a parametrized perverse schober without singularities in the sense of \Cref{def:schobersingularity}, also called a locally constant perverse schober, whose generic stalk admits a Calabi--Yau structure, the next \Cref{thm:FukayaCY} states that its global sections admit a Calabi--Yau structure if its monodromy with respect to any framing of the surface, see \Cref{subsec:monodromy}, acts trivially on the corresponding negative or dual cyclic homology class. Note that a direct variation on this result for arbitrary perverse schobers does not hold, as follows from a variant of \Cref{ex:monodromy}. \Cref{thm:FukayaCY} generalizes the construction of relative Calabi--Yau structures on the topological Fukaya categories of framed surfaces of \cite{BD19}.

\begin{theorem}\label{thm:FukayaCY}
Let $\mathcal{F}\colon \on{Exit}(\rgraph)\to\on{LinCat}_R^{\on{dual}}$ be a $\rgraph$-parametrized perverse schober without singularities. Fix an edge $e$ of $\rgraph$ and let $\mathcal{N}=\mathcal{F}(e)$ be the generic stalk of $\mathcal{F}$.
\begin{enumerate}[(i)]
\item Suppose that $\mathcal{N}$ is smooth and admits a left $(n-1)$-Calabi--Yau structure \[\eta\colon R[n-1]\to \HH(\mathcal{N})^{S^1}\,.\]
Suppose that the local system, see \Cref{rem:monodromyHH},
\[
\HH(\mathcal{L}\mathcal{F})^{S^1}\colon \pi_1(\Sigma_{\rgraph})\longrightarrow \pi_0\on{Aut}_{\on{RMod}_R}(\HH(\N)^{S^1})
\]
preserves $\eta$. Then the functor 
\begin{equation}\label{eq:evfuntobdry1}
\partial \mathcal{F}\colon \prod_{e'\in \rgraph_1^\partial}\mathcal{F}(e') \longrightarrow \glsec(\rgraph,\mathcal{F})
\end{equation}
admits a left $n$-Calabi--Yau structure. 
\item Suppose that $\mathcal{N}$ is proper and admits a right $(n-1)$-Calabi--Yau structure 
\[\eta\colon R[n-1]\to \HH(\mathcal{N})^*_{S^1}\,.\]
Suppose that the local system
\[
\HH(\mathcal{L}\mathcal{F})^*_{S^1}\colon \pi_1(\Sigma_{\rgraph})\longrightarrow \pi_0\on{Aut}_{\on{RMod}_R}(\HH(\N)^*_{S^1})
\]
preserves $\eta$. Then the functor
\[
\prod_{e'\in \rgraph_1^\partial}\on{ev}_{e'}\colon \cptglsec(\rgraph,\mathcal{F}) \longrightarrow \prod_{e'\in \rgraph_1^\partial}\mathcal{F}(e')
\]
admits a right $n$-Calabi--Yau structure. 
\end{enumerate}
\end{theorem}

\begin{proof}
We only prove part (i), part (ii) is analogous. Using \Cref{prop:constraction} and \Cref{lem:invcontraction}, we may assume that $\rgraph$ has a single vertex $v$. We choose a framing $\xi$ on $\Sigma_{\rgraph}$. Let $m$ be the valency of $v$. We choose a total order of the halfedges incident to $v$. Applying \Cref{prop:loccy} to the spherical adjunction $F=0_{\mathcal{N}}\colon 0\leftrightarrow\mathcal{N}\noloc G$, with the left $n$-Calabi--Yau structure on $G$ arising from $\eta$, yields a left $n$-Calabi--Yau structure on $R_{0_{\mathcal{N}}}^m$, which restricts on $\mathcal{N}^{\times m}$ to $\eta^{\times m}$. The diagram $R_{0_{\mathcal{N}}}^m$ gives rise to a perverse schober $\mathcal{G}_v'$ on the $m$-spider $\rgraph_m$, assigning to the incidence of the $i$-th halfedge with $v$ the functor $\varrho_i[-i]$.  

Consider an internal edge $h$ of $\rgraph$, which is by assumption a loop. The loop consists of two halfedges which lie in positions $1\leq i<j\leq m$ and we orient $h$ so that it first traces along the $i$-th halfedge and then the $j$-th halfedge. We modify $\mathcal{G}'_v$ by composing the functor $\mathcal{G}_v'(v\xrightarrow{j}h)$ with the autoequivalence $\xi^*\mathcal{L}\mathcal{F}(h)[j-i-W(h)]\colon \mathcal{N}\to \mathcal{N}$. 

We do this for each such internal edge $h$ and denote the arising perverse schober on the $m$-spider by $\mathcal{G}_v$. We let $\mathcal{G}$ be the $\rgraph$-parametrized perverse schober which restricts along $\on{Exit}(\rgraph_m)\to \on{Exit}(\rgraph)$ to $\mathcal{G}_v$. We have defined $\mathcal{G}$, such that
$\mathcal{F}$ and $\mathcal{G}$ have equivalent monodromy local systems. It follows by \Cref{prop:schobersfrommonodromy} that $\mathcal{F}\simeq \mathcal{G}$.

Using the above relative Calabi--Yau structure on $\mathcal{G}_v'$, that the monodromy of $\mathcal{F}$ acts trivially on $\eta$, that the winding numbers of a framing are all even and that $\HH([1])=-\on{id}_{\HH(\mathcal{N})}$, we find that $\mathcal{F}$ satisfies the assumptions of \Cref{thm:schobercy}. It follows that the functor \eqref{eq:evfuntobdry1} admits the desired left $n$-Calabi--Yau structure.
\end{proof}

\subsection{Weak right Calabi--Yau structures on spherical functors}\label{subsec:sphericalCY}

Consider a dualizable $R$-linear functor $F\colon \D\to \C$ between proper $R$-linear $\infty$-categories with right adjoint $G$. If $F$ admits a weak right Calabi--Yau structure, the arising fiber and cofiber sequence 
\[ \on{id}_\D\xlongrightarrow{\unit} GF \longrightarrow \on{id}_\D^*[1-n]\]
exhibits the shifted Serre functor $\on{id}_\D^*[1-n]$ as the twist functor of the adjunction $F\dashv G$. If $\D$ is smooth, then $\on{id}_{\D}^*$ is an equivalence. To check that the adjunction $F\dashv G$ is spherical it thus suffices to show that the cotwist functor is also an equivalence, or alternatively that the unit of the adjunction $F\dashv G$ commutes with $\on{id}_{\D}^*$ and that $G$ admits a right adjoint $H$ such that $\on{Im}(F)\simeq \on{Im}(H)$, see \cite[Prop.~4.5]{Chr20}.

Conversely, suppose that $F\dashv G$ is a spherical adjunction, satisfying that the twist functor $T_{\mathcal{D}}$ is equivalent to $\on{id}_\D^*[1-n]$. The unit and counit maps of spherical adjunctions exhibit a rather special behavior: in the fiber and cofiber sequence 
\[\on{id}_\D\xlongrightarrow{\unit} GF \xlongrightarrow{\counit'} T_\D\] 
the map $\counit'$ is a counit map of the adjunction $E\dashv F$ composed with $T_\D$, up to composition with an autoequivalence $GF\simeq T_\D\circ EF$. By \Cref{lem:unitisunit}, this fiber and cofiber sequence looks very similar to the diagram \eqref{eq:defrightCY} appearing in the definition of a weak relative right Calabi--Yau structure on $F$. It is thus natural to ask whether $F$ already admits a weak right $n$-Calabi--Yau structure. In this section, we prove that $F$ can indeed be equipped with a weak right $n$-Calabi--Yau structure, under the assumption that $\C$ is weak right $(n-1)$-Calabi--Yau and $\C,\D$ are compactly generated, see \Cref{prop:sphericalrightCY}.

The proof is rather indirect and relies on first lifting the spherical adjunction $F\dashv G$ to a perverse schober on the $3$-spider, which might be thought of as kind of resolution, as it renders trivial certain commutativity problems of diagrams involved in checking the existence of the relative Calabi--Yau structure. We then use an explicitly description of the Serre functor on the global sections of this perverse schober that is only available in the proper setting. We get back to the original spherical adjunction by gluing with the Calabi--Yau structure of the zero functor $\C^{\times 2}\to 0$. We leave it as an interesting problem to find an alternative argument which applies in the smooth setting.

\begin{proposition}\label{prop:sphericalrightCY}
Let $F\colon \mathcal{D}\leftrightarrow \mathcal{C}\noloc G$ be a spherical adjunction of compactly generated, proper $R$-linear $\infty$-categories. Let $T_{\mathcal{D}}$ be the twist functor of $F\dashv G$. If there exists an equivalence $T_{\mathcal{D}}\simeq \on{id}_{\mathcal{D}}^*[1-n]$ and a weak right $(n-1)$-Calabi--Yau structure on $\mathcal{C}$, then $F$ admits a weak right $n$-Calabi--Yau structure.
\end{proposition}

We will apply \Cref{prop:sphericalrightCY} to examples in \Cref{subsec:relGinzburg}
 
\begin{lemma}\label{lem:CYfromspherical}
Let $F\colon \mathcal{D}\leftrightarrow \mathcal{C}\noloc G$ be a spherical adjunction of compactly generated, proper $R$-linear $\infty$-categories. Let $T_{\mathcal{D}}$ be the twist functor of $F\dashv G$. Suppose that there exists an equivalence $T_{\mathcal{D}}\simeq \on{id}_{\mathcal{D}}^*[1-n]$ and that $\mathcal{C}$ admits a weak right $(n-1)$-Calabi--Yau structure. Consider the spherical adjunction, see \Cref{prop:rotationaltwist},
\begin{equation}\label{eq:SSL}
S^L\coloneqq (\varrho_1[-1],\varrho_2[-2],\varrho_3[-3])\colon \mathcal{V}^3_{F}\longleftrightarrow \mathcal{C}^{\times 3}\noloc S\coloneqq (\varsigma_1[1],\varsigma_2[2],\varsigma_3[3])\,.
\end{equation}
The functor $S^L$ admits a weak right $n$-Calabi--Yau structure.
\end{lemma}

\begin{proof}
We show in \Cref{lem:SerrefunV3F} below, that the Serre functor $\on{id}_{\mathcal{V}^3_F}^*$ is equivalent to a suspension of the twist functor $T_{\mathcal{V}^3_F}$ of $S^L\dashv S$. The definition of $T_{\mathcal{V}^3_F}$ thus gives us a fiber and cofiber sequence of endofunctors of $\mathcal{V}^3_F$:
\begin{equation}\label{eq:deftwistV3} 
\on{id}_{\mathcal{V}^3_f}\xlongrightarrow{\unit} SS^L \xlongrightarrow{\eta} \on{id}_{\mathcal{V}^3_F}^*[1-n]
\end{equation}
By assumption, $\C$ admits a weak right $(n-1)$-Calabi--Yau structure, corresponding to an equivalence $\alpha\colon \on{id}_{\C}\simeq \on{id}_{\C}^*[1-n]$. This gives us an equivalence
\[ 
S\alpha^{\times 3}S^L\colon S^*(\on{id}_{\C^{\times 3}})=SS^L\simeq S \on{id}_{\C^{\times 3}}^*[1-n] S^L \simeq S^*(\on{id}_{\C^{\times 3}}^*)[1-n]\,.
\]
The fiber and cofiber sequence \eqref{eq:deftwistV3} gives rise to a commutative diagram in $\on{Lin}_R(\mathcal{V}^3_F,\mathcal{V}^3_F)$
\begin{equation}\label{eq:CYfromspherical}
\begin{tikzcd}
\on{id}_{\mathcal{V}^3_F} \arrow[r, "\unit"] \arrow[d, "\simeq"'] & S^*(\on{id}_{\C^{\times 3}}) \arrow[d, "S^*(\alpha^{\times 3})"'] \arrow[rd, "\eta", dotted] \arrow[r] & \on{cof} \arrow[d, "\simeq"]       \\
\on{fib} \arrow[r]                                                & {S^*(\on{id}_{\C^{\times 3}}^*)[1-n]} \arrow[r, "\nu"]                                                        & {\on{id}_{\mathcal{V}^3_F}^*[1-n]}
\end{tikzcd}
\end{equation}
with horizontal fiber and cofiber sequences and vertical equivalences. The map $\eta$ is up to composition with an equivalence $S^*(\on{id}_{\C^{\times 3}})\simeq \on{id}_{\mathcal{V}^3_F}^*S^{LL}S^L[1-n]$ given by a counit map of the adjunction $S^{LL}\dashv S^L$, which can be seen at follows. By \cite[Cor.~2.5.16]{DKSS21}, there exists an equivalence $e\colon \on{id}_{\mathcal{V}^3_F}^*[1-n]S^{LL}\simeq  T_{\mathcal{V}^3_F}S^{LL}\simeq S$. By \cite[Lemma 2.10]{Chr20}, the natural transformation $\eta\circ eS^L$ evaluates at each object of $\V_F^{\times 3}$ to a counit map of $S^{LL}\dashv S^L$, so that $\eta\circ eS^L$ is adjoint to a pointwise autoequivalence $\on{id}_{\mathcal{V}^3_F}^*S^{LL}[1-n]\simeq \on{id}_{\mathcal{V}^3_F}^*S^{LL}[1-n]$. This implies that $\eta\circ eS^L$ is already a counit composed with $\on{id}_{\V_F^3}^*[1-n]$. By \Cref{lem:unitisunit}, this shows that the natural transformation $\nu$ agrees with the counit $\tilde{\counit}$ from \Cref{constr:unitcounit}, up to composition with an autoequivalence $\beta$ of $S^*(\on{id}_{\C^{\times 3}}^*)[1-n]$. 

As we show next, the adjunction $S^L\dashv S$ has the special feature, that the map 
\[ \on{Map}(\on{id}_{\mathcal{C}^{\times 3}}^*[1-n],\on{id}_{\mathcal{C}^{\times 3}}^*[1-n])\xlongrightarrow{S\circ (\mhyphen)\circ S^L}\on{Map}(S^*(\on{id}_{\C^{\times 3}}^*)[1-n],S^*(\on{id}_{\C^{\times 3}}^*)[1-n])\] 
is an equivalence. We denote the inverse image of $\beta$ under this map by $\beta'$. We have already seen this in the special case that $\D=0$ and $\C=\on{RMod}_R$ in the proof of \Cref{lem:1cy}. The argument here is very analogous: the functor $S^LS$ splits as $S^LS\simeq \on{id}_{\mathcal{C}^{\times 3}}\oplus P$ with $P$ the cotwist functor of $S^L\dashv S$, which permutes the three factors of $\C^{\times 3}$ cyclically by one step and then acts on each component as suspension or the cotwist functor of $F\dashv G$. We thus have $\on{Mor}(P,\on{id}_{\C^{\times 3}})\simeq 0$ and $\on{Mor}(P^2,\on{id}_{\C^{\times 3}})\simeq 0$ and 
\begin{align*} 
\on{Map}(S^*(\on{id}_{\C^{\times 3}}^*),S^*(\on{id}_{\C^{\times 3}}^*))
&\simeq \on{Map}(S^LS\on{id}_{\mathcal{C}^{\times 3}}^*S^LS,\on{id}_{\mathcal{C}^{\times 3}}^*)\\
&\simeq \on{Map}(\on{id}_{\mathcal{C}^{\times 3}}^*,\on{id}_{\mathcal{C}^{\times 3}}^*)\oplus \on{Mor}((P\oplus P \oplus P^2)\circ \on{id}_{\mathcal{C}^{\times 3}}^*,\on{id}_{\mathcal{C}^{\times 3}}^*)\\
&\simeq \on{Map}(\on{id}_{\mathcal{C}^{\times 3}}^*,\on{id}_{\mathcal{C}^{\times 3}}^*)\,.
\end{align*}
We adapt the choice of Calabi--Yau structure on $\C$, by postcomposing $\alpha\colon \on{id}_{\C}\simeq \on{id}_{\C}^*[1-n]$ with $\beta'$. Note that $\nu \circ S^*((\beta')^{\times 3})$ is by construction equivalent to $\tilde{\counit}$. The existence of the diagram \eqref{eq:CYfromspherical} thus induces a relative dual Hochschild homology class $\sigma \in \HH(\mathcal{V}^3_F,\C^{\times 3})^*$ which defines a weak right $n$-Calabi--Yau structure on $S^L$. We finally remark that $\sigma$ restricts at $\C^{\times 3}$ to the class corresponding to $(\beta'\circ \alpha)^{\times 3}$.
\end{proof}

\begin{lemma}\label{lem:SerrefunV3F}
Under the assumptions of \Cref{lem:CYfromspherical}, the shifted twist functor $T_{\mathcal{V}^3_F}[n-1]$ of the adjunction \eqref{eq:SSL} is equivalent to the Serre functor $\on{id}_{\mathcal{V}^3_F}^*$.
\end{lemma}

\begin{proof}
During this proof, we will use the following simplified and abusive notation for elements of $\mathcal{V}^3_F$: given an element of $\mathcal{V}^3_F$, i.e.~a diagram $d\to c_1\to c_2$, with the morphism $d\to c_1$ lying in the Grothendieck construction of $F$, meaning that it encodes a morphism $F(d)\to c_1$, we simply write it as a tuple $(d,c_1,c_2)$. We similarly write elements $c_1\to c_2$ of $\on{Fun}(\Delta^1,\mathcal{C})$ as pairs $(c_1,c_2)$ and elements $d\to c$ of $\mathcal{V}^2_F\simeq \on{fib}(\varrho_1)\subset \mathcal{V}^3_F$ as pairs $(d,c)$. Given elements $x=(c_1,c_2)\in \on{Fun}(\Delta^1,\mathcal{C})$ and $d\in \mathcal{D}$, we will also write $(d,x)$ for $(d,c_1,c_2)$. 

The Serre functor of $\mathcal{C}$ is $U_\mathcal{C}\simeq [n-1]$, the Serre functor of $\mathcal{D}$ is given by $T_\mathcal{D}[n-1]$ and the Serre functor of $\on{Fun}(\Delta^1,\mathcal{C})$ is denoted by $U_2$. The functor $U_2$ is given by the tensor product of the Serre functors of $\on{Fun}(\Delta^1,\on{RMod}_R)$ and $\C$. By \Cref{lem:serreofDelta1C} and \Cref{rem:Serreexplicit}, $U_2$ is thus given by the assignment 
\[ U_2\colon (c_1,c_2)\mapsto (c_2,\on{cof}(c_1\to c_2))[n-1] \,.\]
A straightforward computation shows that the twist functor $T_{\mathcal{V}^3_F}\colon \mathcal{V}^3_F\to \mathcal{V}^3_F$ of $S^L\dashv S$ is given by the assignment 
\[
T_{\mathcal{V}^3_F}\colon (d,c_1,c_2) \mapsto (\on{cof}(d\to G(c_2)),\on{cof}(F(d)\to c_2),\on{cof}(c_1\to c_2))\,.
\]
This assignment is to be understood as in \Cref{rem:Serreexplicit}, meaning the apparent functor corresponding to the above formulas constructed using the universal properties of the involved lax limits.

Consider the relative suspension functor $\tau\colon \mathcal{V}^2_F\to \mathcal{V}^2_F$ of \cite[Def.~2.5.8]{DKSS21}, which is on objects given by mapping $(d,c)$ to $(\on{cof}(d\to G(c)),\on{cof}(F(d)\to c))$. The functor $\tau$ is an equivalence by the sphericalness of $F$ and \cite[Cor.~2.5.10]{DKSS21}, and this implies that
\begin{align*}
\on{Mor}_{\mathcal{V}^2_F}((d',c'),(d,F(d)))&\simeq \on{Mor}_{\mathcal{V}^2_F}(\tau((d',c')),\tau((d,F(d))))\\
&\simeq \on{Mor}_{\mathcal{V}^2_F}((\on{cof}(d'\to G(c')),\on{cof}(F(d')\to c')),(T_\mathcal{D}(d),0))\\
&\simeq \on{Mor}_{\mathcal{D}}(\on{cof}(d'\to G(c')),T_\mathcal{D}(d))\,,
\end{align*}
bifunctorial in $(d',c')\in (\mathcal{V}^2_F)^{\on{op}}$ and $d\in \mathcal{D}$.

We thus have the following equivalences, bifunctorial in $(d,c_1,c_2)\in (\mathcal{V}_F^3)^{\on{c}}$ and $(d',c_1',c_2')\in (\mathcal{V}_F^3)^{\on{c},\on{op}}$:
\begin{align*}
&\on{Mor}((d,c_1,c_2),(d',c_1',c_2'))\\
&\simeq \on{cof}(\on{Mor}((d[-1],0,0),(d',c_1',c_2'))\to \on{Mor}((0,c_1,c_2),(d',c_1',c_2')))\\
&\simeq \on{cof}(\on{Mor}_{\mathcal{D}}(d,\on{cof}(d'\to G(c_1')))\to \on{Mor}_{\on{Fun}(\Delta^1,\mathcal{C})}((c_1,c_2),(c_1',c_2')))\\
&\simeq \on{cof}(\on{Mor}(\on{cof}(d'\to G(c_1')),T_\mathcal{D}(d)[n-1])^*\to \on{Mor}((c_1',c_2'),U_2((c_1,c_2)))^*)\\
&\simeq \on{cof}(\on{Mor}((d',c_1',c_2'),(d,F(d),0)[n-1])^* \to \on{Mor}((d',c_1',c_2'),(G(c_2),U_2(c_1,c_2))[n-1])^*)\\
&\simeq \on{Mor}((d',c_1',c_2'), T_{\mathcal{V}^3_F}(d,c_1,c_2)[n-1])^*
\end{align*}
This shows that $T_{\mathcal{V}^3_F}[n-1]$ is a Serre functor and thus equivalent to $\on{id}_{\mathcal{V}^3_F}^*$.
\end{proof}

\begin{proof}[Proof of \Cref{prop:sphericalrightCY}.]
Consider the functor $S^L\colon \mathcal{V}^3_F\to \C^{\times 3}$ from \Cref{lem:CYfromspherical} which admits a weak right $n$-Calabi--Yau structure. Applying \Cref{thm:rightCYglue} to the pullback diagram in $\on{LinCat}_R$ 
\[
\begin{tikzcd}[column sep=large]
\D \arrow[d] \arrow[r] \arrow[rr, "F", bend left=15] \arrow[rd, "\lrcorner", phantom] & \mathcal{V}^3_F \arrow[d, "{(\varrho_1,\varrho_2)}"] \arrow[r, "{\varrho_3[-2]}"'] & \C \\
0 \arrow[r]                                                                        & \C^{\times 2}                                                                              &   
\end{tikzcd}
\]
yields the desired weak right $n$-Calabi--Yau structure on $F$.
\end{proof}

\section{Examples}\label{subsec:firstCYexamples}

We begin in \Cref{subsec:FScats} by describing Fukaya--Seidel categories as the global sections of perverse schobers on the disc and using this to construct relative Calabi--Yau structures on these. In \Cref{subsec:periodicFukaya}, we describe a special case of \Cref{thm:FukayaCY} concerning relative Calabi--Yau structures on periodic topological Fukaya categories of marked surfaces. Finally, we observe in \Cref{subsec:relGinzburg} that the derived categories of relative Ginzburg algebras of $n$-angulated surfaces admit relative left $n$-Calabi--Yau structures, and further exhibit in some cases weak right $n$-Calabi--Yau structures on the $R$-linear versions of the finite derived categories of these relative Ginzburg algebras, where $R$ is an arbitrary base $\mathbb{E}_\infty$-ring spectrum.

\subsection{Fukaya--Seidel categories} \label{subsec:FScats}

The cosheaves of partially wrapped Fukaya categories of \cite{GPS24} give rise to perverse schobers. We will explain this in this section in the setting of Lefschetz fibrations over the disc. We remark that a related construction of Fukaya--Seidel categories using perverse schobers appears in \cite{KSS20}. 

We work in the setup for partially wrapped Fukaya categories of \cite{GPS24}. The original construction of Fukaya--Seidel categories in the different setup of \cite{Sei08}, as a directed $A_\infty$-category, can be treated in a similar way using perverse schobers. 

Let $\mathbb{C}_{\geq 0}$ be the half-plane (considered as a Liouville sector). Let $\pi\colon X\to \mathbb{C}_{\geq 0}$ be a Lefschetz fibration (in the sense of \cite{GPS24}), with $X$ a Liouville sector, with regular Weinstein fiber $F$ and core $\mathfrak{f}\subset F$. Let $2n$ be the dimension of $M$. We assume that $2n\geq 4$. We further assume that the wrapped Fukaya category $\mathcal{W}(F)$ of the fiber is weak left $(n-1)$-Calabi--Yau, which is shown under minor assumption on $F$ in \cite{Gan13}.

The wrapped Fukaya category $\mathcal{W}(X)$ is equivalent to the partially wrapped Fukaya category $\mathcal{W}(\bar{X},\mathfrak{f})$ of the Liouville manifold $\bar{X}$ arising from $X$ with a stop at $\mathfrak{f}$. This $A_\infty$-category is called the Fukaya--Seidel category of the Lefschetz fibration $\pi$. We will denote it by $\on{FS}(\pi)\coloneqq \mathcal{W}(X)$.

From $\on{FS}(\pi)$, we obtain a $k$-linear stable $\infty$-category $\D(\on{FS}(\pi))\in \on{LinCat}_{k}$, by first choosing a quasi-equivalent dg category to $\on{FS}(\pi)$ and then passing to its derived $\infty$-category. 

As described in Example 1.31 of \cite{GPS24}, the Fukaya--Seidel category $\on{FS}(\pi)$ arises as the homotopy colimit of a diagram of $A_\infty$-categories with values given by the $A_\infty$-categories $\mathcal{W}(F)\otimes A_2, \mathcal{W}(F)$ and $\on{Perf}(k)$. This diagram describes a perverse schober and is described in more detail in \Cref{constr:FSschober} below. This allows us to obtain a weak relative left $n$-Calabi--Yau structure on $\D(\on{FS}(\pi))$:

\begin{theorem}\label{thm:FSschober}
Let $\pi\colon X\to \mathbb{C}_{\geq 0}$ be a Lefschetz fibration as above.
\begin{enumerate}[(i)]
\item The derived $\infty$-category of the Fukaya--Seidel category $\D(\on{FS}(\pi))$ arises as the $\infty$-category of global sections of the perverse schober $\mathcal{F}$ on $\mathbb{C}_{\geq 0}$ from \Cref{constr:FSschober} with singularities at the singular values of $\pi$ and generic stalk $\D(\mathcal{W}(F))$. 
\item Passing to derived $\infty$-categories, the canonical functor 
\[ \mathcal{W}(F)\to \on{FS}(\pi)\] 
known as the cup/Orlov functor\footnote{See \cite{Syl19}. The left adjoint is called the cap functor.} agrees with the spherical boundary corestriction functor $\partial \mathcal{F}$, see Equation \eqref{eq:cupfunctor}. The functor $\partial \mathcal{F}$ admits a weak left $n$-Calabi--Yau structure, exhibiting $\D(\on{FS}(\pi))$ as weakly relative left $n$-Calabi--Yau.
\end{enumerate} 
\end{theorem}

\begin{remark}
The derived Fukaya--Seidel category $\on{FS}(\pi)\in \on{LinCat}_k$ is smooth as the colimit of smooth $\infty$-categories and proper as it is generated by the thimbles. There is a pushout diagram in $\on{LinCat}_k$ of the following form, see \cite[Thm.~1.20]{GPS24}:
\[
\begin{tikzcd}
\D(\mathcal{W}(F)) \arrow[r, "\partial \mathcal{F}"] \arrow[d] \arrow[rd, "\ulcorner", phantom] & \D(\on{FS}(\pi)) \arrow[d] \\
0 \arrow[r]                                                                                     & \D(\mathcal{W}(\bar{X}))  
\end{tikzcd}
\] 
Thus $\on{FS}(\pi)$ can be seen as a smooth and proper resolution of the smooth $\D(\mathcal{W}(\bar{X}))$. The Serre functor on $\D(\on{FS}(\pi))$ is given by part (ii) of \Cref{thm:FSschober} by a shift of the cotwist functor of the spherical adjunction $\partial \mathcal{F}\leftrightarrow \on{radj}(\partial \mathcal{F})$.
Note that spherical functors commute with their (co)twist functors, see \cite[Lem.~2.2]{Chr20}. The Serre functor of $\on{FS}(\pi)$ thus stabilizes the stable subcategory generated by the image of $\partial \mathcal{F}$.

Furthermore, by gluing with the Calabi--Yau functor $\D(\mathcal{W}(F))\to 0$ via \Cref{thm:leftCYglue}, we find that $\D(\mathcal{W}(\bar{X}))$ inherits a weak left $n$-Calabi--Yau structure, recovering the result of \cite{Gan13}.
\end{remark}

\begin{construction}\label{constr:FSschober}
Let $n$ be the number of singular values of the Lefschetz fibration $\pi$. Consider the following ribbon graph $\rgraph_\pi$:
\[
\begin{tikzcd}
                        &                                            & s_n \arrow[rd, no head] &                         \\
                        & \dots                                      &                         & v_n \arrow[ld, no head] \\
s_1 \arrow[rd, no head] &                                            & \dots                   &                         \\
                        & v_1 \arrow[ru, no head] \arrow[d, no head] &                         &                         \\
                        & {}                                         &                         &                        
\end{tikzcd}
\]
We can embed $\rgraph_\pi$ into $\mathbb{C}_{\geq 0}$, making it into a spanning ribbon graph, such that each vertex $s_i$ lies at a singular value of $\pi$ and the external edge incident to $v_1$ ends on the unique boundary component of $\mathbb{C}_{\geq 0}$. The embedding of $\rgraph_\pi$ into $\mathbb{C}_{\geq 0}$ decomposes $\mathbb{C}_{\geq 0}$ into $A_2$-sectors, lying near the vertices $v_1,\dots,v_n$, not containing any singular values, as well as $n$ half-planes containing the singular values $s_i$. The derived wrapped Fukaya categories of the inverse images of the half-planes are each equivalent to $\D(k)$. The Fukaya--Seidel category arises by \cite[Example 1.31]{GPS24} as the homotopy colimit of a diagram of $A_\infty$-categories, indexed by the opposite of the exit path category of $\rgraph_\pi$. This diagram assigns
\begin{itemize}
\item to each edge of $\rgraph_\pi$ an $A_\infty$-category Morita-equivalent to the wrapped Fukaya category $\mathcal{W}(F)$ of the fiber,
\item to each vertex $v_i$ an $A_\infty$-category Morita-equivalent to $\mathcal{W}(F)\otimes A_2$, and
\item to each vertex $s_i$ an $A_\infty$-category Morita-equivalent to $\on{Perf}(k)$.
\end{itemize} 
Passing to derived $\infty$-categories and right adjoint functors yields a diagram $\mathcal{F}\colon \on{Exit}(\rgraph)\to \on{LinCat}_k$, which is readily verified to describe a perverse schober with singularities at $s_1,\dots,s_n$. The spherical adjunctions at the vertices $s_i$ arise from the spherical objects in $\mathcal{W}(F)$ given by the vanishing cycles of the Lefschetz fibration.

Note that we can contract the ribbon graph $\rgraph_\pi$ to the following ribbon graph $\rgraph_\pi'$:
\[
\begin{tikzcd}
s_1 \arrow[rd, no head] & \dots                                    & s_n \\
                        & v \arrow[d, no head] \arrow[ru, no head] &     \\
                        & {}                                       &    
\end{tikzcd}
\]
This allows to understand the $\infty$-category of global sections of $\mathcal{F}$ as a  non-full subcategory of $\mathcal{D}(\mathcal{W}(F)\otimes A_n)\simeq \on{Fun}(\Delta^{n-1},\mathcal{D}(\mathcal{W}(F)))$, similar to the directed subcategory construction of \cite{Sei08}. From this perspective, the thimbles of the Lefschetz fibration amount to coCartesian sections of the $\rgraph_\pi'$-parametrized perverse schober of the following form: For $X_i\in \D(\mathcal{W}(F))$ the $i$-th vanishing cycle, the corresponding thimble $Y_i$ is given as follows,
\[
\begin{tikzcd}[column sep=0]
0 \arrow[d] & & k \arrow[d]                                                                                                                                                                                       & & 0 \arrow[d] \\
0           & \rlap{~~~~~~~~~~~~\dots} & X_i                                                                                                                                                                                               & \llap{\dots~~~~~~~~~~~~} & 0           \\
            &       & {\left(0\rightarrow \dots \rightarrow \underbrace{0}_{m-i\text{-th}} \rightarrow X_i \xrightarrow{\on{id}} \dots \xrightarrow{\on{id}} X_i\right)[1]} \arrow[d] \arrow[u] \arrow[llu] \arrow[rru] &       &             \\
            &       & X_i                                                                                                                                                                                               &       &            
\end{tikzcd}
\]
satisfying that the restriction of $Y_i$ to the $j$-th internal edge $e_j$ is given by $X_i$ if $i=j$ and $0$ if $i\neq j$. 
\end{construction}

We next briefly discuss spherical objects and then show that these give rise to functors with Calabi--Yau structures.

We let $k$ be a commutative ring and $\mathcal{C}$ a dualizable $k$-linear $\infty$-category. We fix an object $X\in \mathcal{C}^{\on{c}}$ whose endomorphism object is equivalent to the singular complex of the $(n-1)$-sphere for some $n\geq 2$, meaning that
\[ \on{Map}_{\mathcal{C}}(X,X)\simeq k\oplus k[-n+1] \in \mathcal{D}(k)\,.\]
The object $X$ gives rise to a $k$-linear adjunction 
\begin{equation}\label{eq:spobjadj} 
\mhyphen\otimes_k X\colon \mathcal{D}(k)\longleftrightarrow \mathcal{C}\noloc \on{Map}_{\mathcal{C}}(X,\mhyphen)\,.
\end{equation}
By the assumption on $X$, the twist functor $T_{\mathcal{D}(k)}\simeq \on{cof}(\on{id}_{\mathcal{D}(k)}\rightarrow \mhyphen\otimes_k \on{Map}_{\mathcal{C}}(X,X))$ is equivalent to the $(n-1)$-fold loop functor $[-n+1]$ and thus an equivalence. We call the object $X$ an $(n-1)$-spherical object, if the cotwist functor $T_{\mathcal{C}}$ is also an equivalence and the adjunction \eqref{eq:spobjadj} thus a spherical adjunction. In this case, the right adjoint of $\on{Map}_{\mathcal{C}}(X,\mhyphen)$ is given by $\mhyphen\otimes_k X[n-1]$. If $\mathcal{C}$ is proper and compactly generated, with Serre functor $U$, the adjunction is spherical if and only if $U(X)\simeq X[n-1]$, so that we specialize to the usual notion of a spherical object, see for instance \cite[Def.~8.1]{Huy06}: One implication of this can be proven using \cite[Prop.~4.5]{Chr20} and the fact that iterated adjoints between proper, compactly generated $\infty$-categories are obtained by compositions with powers of the Serre functors.

\begin{lemma}\label{lem:sphCY}
Let $n\not =1$ and $X\in \mathcal{C}$ an $(n-1)$-spherical object in a dualizable $k$-linear $\infty$-category.
\begin{enumerate}[(1)]
\item If $\C$ is a proper and admits a weak right $(n-1)$-Calabi--Yau structure, then the functor
\[
\mhyphen\otimes_k X\colon \D(k)\longrightarrow \C 
\]
admits a compatible weak right $n$-Calabi--Yau structure. 
\item If $\C$ is smooth and admits a weak left $n$-Calabi--Yau structure, then the functor
\[
\on{Map}_\C(X,\mhyphen)\colon \mathcal{C}\longrightarrow \mathcal{D}(k)
\]
admits a compatible weak left $n$-Calabi--Yau structure.
\end{enumerate}
\end{lemma}

\begin{proof}
We only show part (1), part (2) can be shown analogously.
Let $\sigma\colon k[n-1]\to \on{HH}(\C)^*$ be a weak right $(n-1)$-Calabi--Yau structure on $\C$. The class $\sigma$ gives rise to the following diagram:
\[
\begin{tikzcd}
{\on{id}_{\mathcal{D}(k)}} \arrow[r, "\unit"] & {\mhyphen\otimes_k \on{Mor}_{\mathcal{C}}(X,X)} \arrow[d, "\simeq"] &                          \\
                                            & {\mhyphen\otimes_k \on{Mor}_{\mathcal{C}}(X,U(X))[-n+1]} \arrow[r, "\tilde{\counit}"]                               & \on{id}_{\mathcal{D}(k)}^*[-n+1]
\end{tikzcd}
\]
Note that $\on{id}_{\mathcal{D}(k)}^*\simeq \on{id}_{\mathcal{D}(k)}$. Since there are no natural transformations from $\on{id}_{\mathcal{D}(k)}$ to $\on{id}_{\mathcal{D}(k)}[-n+1]$, the above diagram is equivalent to the following diagram:
\[
\begin{tikzcd}
{\on{id}_{\mathcal{D}(k)}} \arrow[r, hook] & {\on{id}_{\mathcal{D}(k)}\oplus \on{id}_{\mathcal{D}(k)}[-n+1]} \arrow[d, "="]       &                          \\
                                                & {\on{id}_{\mathcal{D}(k)}\oplus \on{id}_{\mathcal{D}(k)}[-n+1]} \arrow[r, two heads] & {\on{id}_{\mathcal{D}(k)}[-n+1]}
\end{tikzcd}
\]
This diagram clearly admits a null-homotopy which defines a non-degenerate relative dual Hochschild class $k[n]\to \HH(\mathcal{D}(k),\C)^*$, restricting to $\sigma$ on $\on{HH}(\C)^*$, thus exhibiting the desired weak right $(n+1)$-Calabi--Yau structure on the functor $\mhyphen\otimes_k X\colon \D(k)\longrightarrow \C $.
\end{proof}

\begin{proof}[Proof of \Cref{thm:FSschober}]
We begin with showing part (i). By the cosheaf properties of partially wrapped Fukaya categories, see Example 1.31 in \cite{GPS24}, we find that $\on{FS}(\pi)$ arises as the homotopy colimit of the diagram of $A_\infty$-categories indexed by $\on{Exit}(\rgraph_\pi)^{\on{op}}$ described in \Cref{constr:FSschober}. It remains to show that the passage to the derived $\infty$-category turns the $A_\infty$-categorical homotopy colimit\footnote{We point out that in contrast to the category of dg categories, the category of $A_\infty$-categories does not admit a suitable model structure, so that the notion of 'homotopy colimit' employed in \cite{GPS24} is not understood in a model categorical sense. One can nevertheless expect that the passage to the derived $\infty$-categories to always turns such homotopy colimits into $\infty$-categorical colimits.} into an $\infty$-categorical colimit. This follows from two observations. Firstly, by the universal property of the colimit, there is a comparison functor $\mathcal{H}(\rgraph_\pi,\mathcal{F})\to \D(\on{FS}(\pi))$. Secondly, both stable $\infty$-categories are generated by the thimbles and on these the above functor is a quasi-equivalence.  

To obtain the weak left Calabi--Yau structure in part (ii), we combine \Cref{thm:leftCYglue}, \Cref{lem:sphCY}, as well as \Cref{prop:loccy} applied to the spherical adjunction $\mathcal{D}(\mathcal{W}(F))\leftrightarrow 0$. The existence of the weak right Calabi--Yau structure similarly follows using the observation that the limit of $\mathcal{F}$ is equivalent to the limit of a diagram of proper subcategories in $\on{LinCat}_k^{\on{dual}}$. 
\end{proof}

\subsection{Periodic topological Fukaya categories}\label{subsec:periodicFukaya} 

Let $k$ be a field. We fix an integer $n\geq 1$ and denote by $k[t_{n}^{\pm}]$ the graded ring of Laurent polynomials in a formal variable $t_{n}$ in degree $n$. Note that if $n$ is even, then $k[t_{n}^{\pm}]$ is graded commutative. A $k[t_{n}^{\pm}]$-module amounts to an $n$-periodic $k$-linear chain complexes. If $n$ is even, we thus refer to $k[t_{n}^{\pm}]$-linear $\infty$-categories as $n$-periodic $k$-linear $\infty$-categories. In the following, we discuss how \Cref{thm:FukayaCY} specializes to periodic topological Fukaya categories. 

If $n$ is even, we set $m=n$. Otherwise, we set $m=2n$. As explained in \cite{DK15}, given an oriented marked surface ${\bf S}$ if $m=2$, or more generally a marked $m$-spin surface, with spin structure denoted by $\xi$, one can associate a topological Fukaya of ${\bf S}$ with values in the $n$-periodic stable $\infty$-category $\mathcal{D}(k[t_n^\pm])$. The analogue of its construction in \cite{DK15} in terms of the global sections of perverse schobers is as follows: there is a unique (up to equivalence) $k[t_{m}^{\pm}]$-linear locally constant perverse schober $\mathcal{F}$ on ${\bf S}$ with generic stalk $\mathcal{D}(k[t_n^\pm])$ and whose monodromy local system with respect to $\xi$, in the sense of \Cref{rem:spinstructure}, is trivial. We call its global sections the $n$-periodic topological Fukaya category of ${\bf S}$ and denote it by $\on{Fuk}({\bf S},\mathcal{D}(k[t_n^\pm]))$. The corresponding $\infty$-category of locally compact global sections of $\mathcal{F}$ is by \Cref{lem:loccpt=Indfin} equivalent to $\on{Ind}\on{Fuk}({\bf S},\mathcal{D}(k[t_n^\pm]))^{\on{fin}}$.

\begin{theorem}\label{thm:relCYperiodicFukaya}
Let $k$ be a field with $\on{char}(k)\neq 2$ and $m,n$ as above. 
\begin{enumerate}[(1)]
\item The $k[t_{m}^{\pm}]$-linear topological Fukaya category $\on{Fuk}({\bf S},\mathcal{D}(k[t_n^\pm]))$ valued in the derived $\infty$-category of $n$-periodic chain complexes admits a relative left $(n+1)$-Calabi--Yau structure. 
\item The $\on{Ind}$-finite subcategory $\on{Ind}\on{Fuk}({\bf S},\mathcal{D}(k[t_n^\pm]))^{\on{fin}}$ is proper and admits a relative right $(n+1)$-Calabi--Yau structure. Further, if each boundary component of ${\bf S}$ has at least one marked point, then $\on{Fuk}({\bf S},\mathcal{D}(k[t_n^\pm]))\simeq \on{Ind}\on{Fuk}({\bf S},\mathcal{D}(k[t_n^\pm]))^{\on{fin}}$.
\end{enumerate}
\end{theorem}

\begin{proof}
Combine \Cref{thm:FukayaCY}, \Cref{lem:periodicCY} for $n$ odd and \Cref{rem:cptglsec=glsec}.
\end{proof}

\begin{lemma}\label{lem:periodicCY}
Let $k$ be a field with $\on{char}(k)\neq 2$. Let $n\geq 1$ be odd. Then $\mathcal{D}(k[t_{n}^\pm])$ is smooth and proper as a $k[t_{2n}^\pm]$-linear $\infty$-category and further admits left and right $n$-Calabi--Yau structures.
\end{lemma}

\begin{proof}
Denote $A\coloneqq k[t_n^\pm]$. The $k[t_{2n}^\pm]$-linear enveloping algebra $A^e$ of $A$ is given by the graded commutative dg algebra $k[t^\pm,s^\pm]/(s^2-t^2)$, with generators $t,s$ in degrees $n$ and satisfying $st=(-1)^{m^2}ts=-ts$ (graded commutativity), as well as $s^2=t^2$. As a right $A^e$-module, $A$ is equipped with the action $1.t=t$ and $1.s=-t$. As a left $(A^e)^{\on{op}}$-module, $A$ is equipped with the action $t.1=-t$ and $s.1=t$. We denote by $\bar{A}$ the the right $A^e$-module $A$ with the action $1.t=t$ and $1.s=t$. 

We consider $A^e$ as a right module over itself. There is a retract of right $A^e$-modules  
\[
A\xrightarrow{1\mapsto 1-s^{-1}t}A^e\xrightarrow{1\mapsto 1}A
\]
since the composite is given by multiplication by $1-1.s^{-1}t=1+t^{-1}t=2\neq 0$ and thus invertible. There is a similar retract
\[
\bar{A}\xrightarrow{1\mapsto 1+s^{-1}t}A^e\xrightarrow{1\mapsto 1}\bar{A}
\]
and $A^e\simeq A \oplus \bar{A}$. 

It follows that $A$ is compact as a right $A^e$-module. The inverse dualizing functor $\on{id}_{\mathcal{D}(A)}^!$ is given by the tensor product with the left $A^e$-module $A^!=\on{RHom}_{A^e}(A,A^e)$. One finds $\on{RHom}_{A^e}(A,\bar{A})\simeq \on{RHom}_{A^e}(\bar{A},A)\simeq 0$. We thus have $\on{RHom}_{A^e}(A,A^e)\simeq \on{RHom}_{A^e}(A,A)\simeq \on{RHom}_{A^e}(A^e,A)\simeq A$ on $k$-linear homology, with $1\in A$ being the image of $\phi\colon A\to A^e,\,1\mapsto 1-s^{-1}t$. The element $t\in A$ corresponds to $\phi'\colon 1\mapsto t-s$. The left action of $A^e$ on $\on{Hom}_{A^e}(A,A^e)\simeq A$ is determined by $t.\phi=\phi'$ and $s.\phi=-\phi'$. It follows that $A^!=\on{RHom}_{A^e}(A,A^e)\simeq A[-n]$ as left $A^e$-modules, since the shift by $n$ preserves the homology of $A$ over $k$ but flips the signs of the actions of $t$ and $s$ (since $n$ is odd). This shows that $\on{id}_{\mathcal{D}(A)}^!\simeq \on{id}_{\mathcal{D}(A)}[-n]$, as desired. Composing with $\on{id}_{\mathcal{D}(A)}^*$ also yields $\on{id}_{\mathcal{D}(A)}\simeq \on{id}_{\mathcal{D}(A)}^*[-n]$.

It remains to show that the (dual) Hochschild homology classes of these weak left and right $m$-Calabi--Yau structures lift to negative cyclic homology and dual cyclic homology, respectively. We show this by proving the triviality of the $S^1$-action on $\HH(\mathcal{D}(A))$. We have $\on{HH}(\mathcal{D}(A))\simeq \on{RHom}_{A^e}(A^!,A)\simeq \on{RHom}_{A^e}(A,A)[n]\simeq A[n]\simeq A$. The $S^1$-action on the $k[t_{2n}^\pm]$-linear Hochschild homology $\on{HH}(\mathcal{D}(k[t_{2n}^\pm]))\simeq k[t_{2n}^\pm]$ is trivial by \Cref{rem:traceofunit} and the same thus holds for its image $k[t_{2n}^\pm]\subset A$ under 
\[\on{HH}(\mhyphen \otimes_{k[t_{2n}^\pm]} A)\colon \HH(\mathcal{D}(k[t_{2n}^\pm]))\to \HH(\mathcal{D}(A))\,.\] 
We are left with determining the $S^1$-action on the summand $k[t_{2n}^\pm][n]\subset A\simeq k[t_{2n}^\pm]\oplus k[t_{2n}^\pm][n]$. An $S^1$-action on a $k[t_{2n}^\pm]$-module is the same as a $k[t_{2n}^\pm][S^1]$-module structure, where $k[t_{2n}^\pm][S^1]\simeq k[t_{2n}^\pm]\otimes_k k[S^1]\simeq k[t_{2n}^\pm]\otimes_k k[s_1]/(s_1^2=0)$, with $|s_1|=1$, see for instance \cite[Prop.~3.3]{HR20} for the latter equivalence. Since $2n\geq 2$, the action of $s_1$, and thus of $S^1$, on $k[t_{2n}^\pm][n]$ is trivial for degree reasons, concluding the proof.
\end{proof}

\begin{remark}
There is a $2$-periodic version of the sphere spectrum and topological Fukaya categories with coefficients in the modules over this ring spectrum have been considered in \cite{LurWaldhausen}. Interestingly, the $2$-periodic sphere spectrum is not an $\mathbb{E}_\infty$-ring spectrum, but only an $\mathbb{E}_2$-ring spectrum, so that \Cref{thm:relCYperiodicFukaya} cannot be directly lifted to this setting.
\end{remark}

\subsection{Relative Ginzburg algebras of surfaces}\label{subsec:relGinzburg}

Fix a base $\mathbb{E}_\infty$-ring spectrum $R$ and let $n\geq 3$. Let ${\bf S}$ be a marked surface, equipped with an $n$-valent spanning ribbon graph $\rgraph$. It is dual to a so-called  ideal $n$-angulation, roughly meaning a decomposition of ${\bf S}$ into $n$-gons with vertices at the marked points of ${\bf S}$. There is an associated perverse schober $\mathcal{F}_{\rgraph}(R)$, see \cite{Chr21b}. If $R$ is discrete, i.e.~a commutative ring, the $\infty$-category of global sections $\glsec({\rgraph},\mathcal{F}_{\rgraph})$ is equivalent to the derived $\infty$-category $\mathcal{D}(\mathscr{G}_\rgraph)$ of a relative Ginzburg algebra $\mathscr{G}_\rgraph$, see \cite{Chr21b}. The generic stalk of $\mathcal{F}_{\rgraph}$ is given by the $\infty$-category $\on{Fun}(S^{n-1},\on{RMod}_R)$ of $\on{RMod}_R$-valued local systems on the $(n-1)$-sphere. At every vertex of ${\rgraph}$, the spherical adjunction underlying the perverse schober $\mathcal{F}_{\rgraph}(R)$ is given by the adjunction
\[ f^*:\on{RMod}_R\longleftrightarrow \on{Fun}(S^{n-1},\on{RMod}_R):f_*\]
arising from the pullback functor along the inclusion of the boundary $f:S^{n-1}\rightarrow D^n\simeq \ast$ of the $n$-ball. 

As shown in \cite{Chr20}, there is an equivalence of $R$-linear $\infty$-categories
\[ \on{RMod}_{R[t_{n-2}]}\simeq \on{Fun}(S^{n-1},\on{RMod}_R)\,,\]
where  $R[t_{n-2}]$ denotes the free $R$-linear algebra generated by $R[n-2]\in \on{RMod}_R$. Under this equivalence, the functor $f_*$ is identified with the pullback functor $\phi^*$ along $R[t_{n-2}]\xrightarrow{t_{n-2}\mapsto 0}R$. Note that if $R=k$ is a field, then $k[t_{n-2}]$ is the graded polynomial algebra with generator in degree $|t_{n-2}|=n-2$.

If $R=k$ is a field, it is shown in \cite[Thm.~5.7]{BD19} that the functor $f_*$ admits a relative left $n$-CY structure, which further restricts to a left $(n-1)$-Calabi--Yau structure on $\on{Fun}(S^{n-1},\on{RMod}_k)$. The next theorem states, that \Cref{thm:schobercy} applies to give a relative Calabi--Yau structure on the global sections of $\mathcal{F}_{\rgraph}(k)$ if either $n$ is odd, or the spanning graph $\rgraph$ is orientable in the following sense.

\begin{definition}
Let $n$ be even. We call the $n$-valent spanning graph $\rgraph$ of ${\bf S}$ orientable, if there exist choices of orientations of the edges of $\rgraph$, such that the directions of the halfedges at any vertex of $\rgraph$ alternate in their cyclic order. 
\end{definition}

\begin{theorem}\label{thm:ginzburgcy}
Let $\rgraph$ be an $n$-valent spanning graph of a marked surface ${\bf S}$ and $\mathcal{D}(\mathscr{G}_{\rgraph})$ the derived $\infty$-category of the corresponding relative Ginzburg algebra $\mathscr{G}_{\rgraph}$. If $n$ is odd or $\rgraph$ orientable, then the functor
\[ \prod_{e\in \rgraph_1^\partial} \mathcal{D}(k[t_{n-2}])\simeq  \prod_{e\in {\rgraph}_1^\partial}\mathcal{F}_{\rgraph}(k)(e)\xlongrightarrow{\partial \mathcal{F}_{\rgraph}(k)} \glsec({\rgraph},\mathcal{F}_{\rgraph}(k))\simeq \mathcal{D}(\mathscr{G}_{\rgraph})\]
admits a $k$-linear left $n$-Calabi--Yau structure. 
\end{theorem} 

\begin{proof}
For the construction of $\mathcal{F}_{\rgraph}(k)$, each vertex $v$ of ${\rgraph}$ is equipped with a choice of total order of its incident halfedges, whose corresponding edges we denote by $e_1,\dots,e_n$. The functor \[\mathcal{F}_\rgraph(k)(v\to e_i)\colon \mathcal{V}^n_{f^*}\to \on{Fun}(S^{n-1},\on{RMod}_k)\]
is either given by $\varrho_i$ or $\varphi^*\circ \varrho_i$, with $\varphi^*$ the autoequivalence of $\on{Fun}(S^{n-1},\on{RMod}_k)\simeq \mathcal{D}(k[t_{n-2}])$ given by pullback along $\varphi\colon k[t_{n-2}]\xrightarrow{t_{n-2}\mapsto (-1)^{n}t_{n-2}}k[t_{n-2}]$. Inspecting \Cref{prop:loccy}, using $\HH([1])^{S^1}=-\on{id}$, and ignoring the equivalence $\varphi^*$ for the moment, we see that the signs alternate cyclically of the classes describing the left $(n-1)$-Calabi--Yau structures of $\mathcal{F}_\rgraph(k)(e_i)$, arising from restricting the relative left $n$-Calabi--Yau structure of $\mathcal{F}_\rgraph(k)(v)$. If $n$ is odd, the map $\HH(\varphi^*)^{S^1}$ reverses the sign of the class, whereas if $n$ is even, $\varphi^*=\on{id}_{\mathcal{D}(k[t_{n-2}])}$ fixes the class.

We suppose that $n$ is even and choose an orientation of $\rgraph$. Choose further for each vertex $v$ of $\rgraph$ an incident halfedge; if it points outwards from $v$, we equip $\mathcal{F}_\rgraph(k)(v)$ with the relative Calabi--Yau structure from \Cref{prop:loccy}, and if $v$ points inwards, we equip $\mathcal{F}_\rgraph(k)(v)$ with the same relative Calabi--Yau structure, except for reversing the sign of the relative dual cyclic homology class. With these choices, \Cref{thm:schobercy} applies. 

It remains to consider the case that $n$ is odd. For all vertices $v$, we equip $\mathcal{F}_\rgraph(k)(v)$ with the relative Calabi--Yau structure from \Cref{prop:loccy}. Consider the two vertices $v,v'$ incident to an edge $e_i$, in position $i$ in the total order of halfedges at $v$ and in the $i'$-th position in the total order of halfedges at $v'$. The equivalence $\varphi^*$ appears in one of the functors $\mathcal{F}_\rgraph(k)(v\to e_i)$, $\mathcal{F}_\rgraph(k)(v'\to e_i)$ if and only if the difference $i-i'$ of the two positions in the total orders of the two halfedges of $e_i$ is even. We thus see that the induced Calabi--Yau structures on $\mathcal{F}_\rgraph(k)(e_i)$ are compatible for any edge $e_i$, so that \Cref{thm:schobercy} again applies.
\end{proof}

We next note a variant of \Cref{thm:ginzburgcy} on the existence of a weak right $n$-Calabi--Yau structure on the proper $R$-linear $\infty$-category $\cptglsec(\rgraph,\mathcal{F}_{\rgraph}(R))$, with $n$ even and $R$ an arbitrary $\mathbb{E}_\infty$-ring spectrum.

\begin{theorem}\label{thm:spectralginzburgcy}
Suppose that $n$ is even and $\rgraph$ orientable. Then the functor
\[\prod_{e\in {\rgraph}_1^\partial}\on{ev}_e\colon \cptglsec(\rgraph,\mathcal{F}_{\rgraph}(R)) \longrightarrow \prod_{e\in (\rgraph)_1^\partial}\mathcal{F}_{\rgraph}(R)(e)\]
admits an $R$-linear weak right $n$-Calabi--Yau structure. 
\end{theorem}

\begin{proof}
The proof of \Cref{thm:ginzburgcy} directly translates by using \Cref{lem:cyf*}.
\end{proof}

The adjunction $f^*\dashv f_*$ restricts to the adjunction of proper $R$-linear $\infty$-categories:
\[ \bar{f}^*\colon \on{RMod}_R \longleftrightarrow \on{Ind}\on{Fun}(S^{n-1},\on{RMod}_{R}^{\on{perf}})\noloc \bar{f}_*\,.\]
We note that
\[ \on{Fun}(S^{n-1},\on{RMod}_{R}^{\on{perf}})\simeq \on{Fun}(S^{n-1},\on{RMod}_{R})^{\on{fin}}\,.\]

\begin{lemma}\label{lem:cyf*}
For any $n\geq 1$, the $R$-linear $\infty$-category $\on{Ind}\on{Fun}(S^{n-1},\on{RMod}_{R}^{\on{perf}})$ admits a weak right $(n-1)$-Calabi--Yau structure and the functor $\bar{f}^*$ admits a compatible weak right $n$-Calabi--Yau structure.
\end{lemma}

\begin{proof}
The adjunction $f^*\dashv f_*$ is spherical with twist functor $T_{\on{RMod}_R}\simeq [1-n]$, see \cite{Chr20}, the same thus holds for $\bar{f}^*\dashv \bar{f}_*$. 
By \Cref{prop:sphericalrightCY}, the functor $\bar{f}^*$ thus admits a weak right $n$-Calabi--Yau structure if $\on{Ind}\on{Fun}(S^{n-1},\on{RMod}_{R}^{\on{perf}})$ admits a weak right $n$-Calabi--Yau structure. 

Applying $\on{Fun}(\mhyphen,\on{RMod}_R^{\on{perf}})$ to the following pushout diagram of spaces
\[
\begin{tikzcd}
S^i \arrow[r, "f"] \arrow[d, "f"] \arrow[rd, "\ulcorner", phantom] & \ast \arrow[d] \\
\ast \arrow[r]                                                     & S^{i+1}       
\end{tikzcd}
\]
and $\on{Ind}$-completing, we obtain the pullback diagram of compactly generated $R$-linear $\infty$-categories 
\begin{equation}\label{funpueq}
\begin{tikzcd}
{\on{Ind}\on{Fun}(S^{i+1},\on{RMod}_{R}^{\on{perf}})} \arrow[r] \arrow[d] \arrow[rd, "\lrcorner", phantom] & \on{RMod}_R \arrow[d, "\bar{f}^*"]           \\
\on{RMod}_R \arrow[r, "\bar{f}^*"]                                                                            & {\on{Ind}\on{Fun}(S^{i},\on{RMod}_{R}^{\on{perf}})}
\end{tikzcd}
\end{equation}
Applying \Cref{thm:rightCYglue} to \eqref{funpueq}, it now follows by induction on $i$ that $\on{Ind}\on{Fun}(S^{i},\on{RMod}_{R}^{\on{perf}})$ admits a weak right $i$-Calabi--Yau structure and that $\bar{f}^*\colon \on{RMod}_R \to \on{Ind}\on{Fun}(S^{i},\on{RMod}_{R}^{\on{perf}})$ admits a weak right $(i+1)$-Calabi--Yau structure, concluding the proof.
\end{proof}

\bibliography{biblio} 
\bibliographystyle{alpha}

\textsc{MC: Mathematisches Institut, Universität Bonn, Endenicher Allee 60, 53115 Bonn, Germany.}

\textit{Email address:} \texttt{christ@math.uni-bonn.de}

\end{document}